\documentclass[11pt]{amsart}

\usepackage[utf8]{inputenc}
\usepackage[english]{babel}
\usepackage[a4paper,twoside,top=1.2in, bottom=1.2in, left=0.8in, right=0.8in]{geometry}

\usepackage{graphicx}
\usepackage{caption} 
\usepackage{amsmath}
\usepackage{amsthm}
\usepackage{amssymb}
\usepackage{esint} 
\usepackage{mathrsfs} 
\usepackage{xcolor}
\usepackage{array}
\usepackage{hhline}
\usepackage{enumitem} 
\usepackage{comment} 
\usepackage[toc,page]{appendix} 
\usepackage{xparse} 
\usepackage{mathtools}
\usepackage{fancyhdr} 
\usepackage{ifthen} 
\usepackage{forloop} 
\usepackage{xstring}
\usepackage{emptypage} 
\usepackage{setspace}
\usepackage[initials, alphabetic]{amsrefs} 

\usepackage{tikz}
\usetikzlibrary{arrows,shapes,patterns,calc,fadings,decorations.pathreplacing,decorations.markings,decorations.pathmorphing,backgrounds,cd}

\usepackage{musicography}

\usepackage{hyperref} 
\hypersetup{
	colorlinks = true,
	linkcolor = {blue},
	urlcolor = {red},
	citecolor = {blue}
}

\usepackage[nameinlink,capitalise]{cleveref} 

\newcounter{results}[section] 

\theoremstyle{plain}
\newtheorem{theorem}[results]{Theorem}
\newtheorem{lemma}[results]{Lemma}
\newtheorem{proposition}[results]{Proposition}
\newtheorem{corollary}[results]{Corollary}

\newtheorem*{theorem*}{Theorem}
\newtheorem*{lemma*}{Lemma}
\newtheorem*{proposition*}{Proposition}
\newtheorem*{corollary*}{Corollary}
\newtheorem*{exercise*}{Exercise}
\newtheorem*{fact*}{Fact}

\theoremstyle{remark}
\newtheorem{remark}[results]{Remark}

\newtheorem*{remark*}{Remark}
\newtheorem*{question*}{Question}

\theoremstyle{definition}
\newtheorem{definition}[results]{Definition}
\newtheorem{example}[results]{Example}

\newtheorem*{definition*}{Definition}
\newtheorem*{example*}{Example}

\numberwithin{equation}{section}

\crefname{figure}{Figure}{Figures}

\newcommand{\Z}{\ensuremath{\mathbb Z}}
\newcommand{\Q}{\ensuremath{\mathbb Q}}
\newcommand{\R}{\ensuremath{\mathbb R}}
\newcommand{\C}{\ensuremath{\mathbb C}}

\DeclareMathOperator{\vol}{vol} 

\newcommand\restr[2]{{
  \left.\kern-\nulldelimiterspace 
  #1 
  \vphantom{\big|} 
  \right|_{#2} 
  }}

\renewcommand{\H}{\mathbb{H}}
\renewcommand{\P}{\mathbb{P}}
\newcommand{\vphi}{\varphi}

\newcommand{\av}[1]{\lvert #1 \rvert}

\newcommand{\Id}{\mathrm{Id}}
\newcommand{\la}{\langle}
\newcommand{\ra}{\rangle}
\newcommand{\pd}[1]{\partial_{#1}}

\newcommand{\im}{\mathrm{im}}

\renewcommand{\Im}{\mathrm{Im}}

\DeclareMathOperator{\Aut}{Aut}

\DeclareMathOperator{\supp}{supp}

\DeclareMathOperator{\Diff}{Diff}
\DeclareMathOperator{\Spin}{Spin}
\DeclareMathOperator{\G2}{G_2}
\DeclareMathOperator{\SU}{SU}
\DeclareMathOperator{\SL}{SL}
\DeclareMathOperator{\U}{U}
\DeclareMathOperator{\SO}{SO}
\DeclareMathOperator{\GL}{GL}
\DeclareMathOperator{\Sp}{Sp}

\DeclareMathOperator{\T}{\mathbb{T}}
\DeclareMathOperator{\Sym}{Sym}

\newcommand{\LieSU}{\mathfrak{su}}
\newcommand{\LieT}{\mathfrak{t}}

\newcommand{\mL}{\mathcal{L}}

\colorlet{myGray}{gray}
\colorlet{myBlue}{blue}
\colorlet{myBlack}{black}
\colorlet{myBackground}{gray!10}

\title[]{On $G_2$ manifolds with cohomogeneity two symmetry}
\author{Benjamin Aslan and Federico Trinca}
\newcommand\printaddress{{
\setlength{\parindent}{17pt}
\footnotesize
{\scshape \noindent Benjamin Aslan, Federico Trinca}
\newline Department of Mathematics, University College London, Gower Street, London, WC1E 6BT, United Kingdom
\newline
\textit{E-mail address:} \texttt{ucahbas@ucl.ac.uk, f.trinca@ucl.ac.uk}
\par
}} 

\begin{document}
\begin{abstract}
We consider $\G2$ manifolds with a cohomogeneity two $\T^2\times \SU(2)$ symmetry group. We give a local characterization of these manifolds and we describe the geometry, including regularity and singularity analysis, of cohomogeneity one calibrated submanifolds in them. We apply these results to the manifolds recently constructed by Foscolo--Haskins--Nordstr\"om and to the Bryant--Salamon manifold of topology $S^3\times \R^4$. In particular, we describe new large families of complete $\T^2$-invariant associative submanifolds in them.
\end{abstract}

\maketitle

\thispagestyle{empty}

\section{Introduction}
In a Riemannian manifold, parallel transport with respect to the Levi-Civita connection is used to define its Riemannian holonomy group. The groups that can appear as the holonomy of a simply-connected, nonsymmetric and irreducible Riemannian manifold were classified by Berger in \cite{Berger1953}. All but two elements of Berger's list come in a countable family depending on the dimension of the manifold. The exceptional cases are $\G2$ and $\Spin(7)$, which are only related to Riemannian manifolds of dimension $7$ and $8$, respectively. Manifolds with holonomy $\G2$, called $\G2$ manifolds, are Ricci-flat \cite{Salamon1989}*{Lemma 11.8} and admit two natural classes of volume minimizing submanifolds: the associative $3$-folds and the coassociative $4$-folds, which are, in particular, calibrated submanifolds \cite{HarveyLawson1982}. 

Bryant and Salamon constructed the first complete $\G2$ manifolds with full holonomy more than 30 years ago in \cite{BryantSalamon1989}. Since then, much effort has been spent to construct new examples (e.g. \cites{BrandhuberGomisGubserGukov2001, Bogoyavlenskaya2013, FoscoloHaskinsNordstrom2021a, FoscoloHaskinsNordstrom2021b, Foscolo2021, MadsenSwann2012, MadsenSwann2018}) and study their calibrated submanifolds (e.g. \cites{Kawai2018, KarigiannisLeung2012, KarigiannisLotay2021, KarigiannisMinoo2005}). Even though we now have a lot of examples of complete non-compact manifolds with Riemannian holonomy $\G2$ (mainly because of the seminal work by Foscolo--Haskins--Nordstr\"om and Foscolo \cites{FoscoloHaskinsNordstrom2021a,FoscoloHaskinsNordstrom2021b,Foscolo2021}), only a few non-trivial associative and coassociative submanifolds were constructed in them.
  
One of the most successful techniques used to construct non-compact $\G2$ manifolds is symmetry reduction, which means that the manifold admits a structure-preserving, hence isometric, Lie group action. Particular attention has been given to the cohomogeneity one and to the abelian case. Indeed, under the former assumption, the system of PDEs characterising the $\G2$ holonomy condition becomes a system of ODEs and many examples were constructed in this way (cfr. \cites{BrandhuberGomisGubserGukov2001, Bogoyavlenskaya2013, BryantSalamon1989,FoscoloHaskinsNordstrom2021b}). Under the latter assumption, the problem reduces to finding a torus bundle with curvature constraints over a lower dimensional manifold with some special structure (cfr. \cites{ApostolovSalamon2004, ChiossiSalamon2002, MadsenSwann2012,MadsenSwann2018}). This technique often relies on the multi-moment maps introduced by Madsen and Swann in \cites{MadsenSwann2012, MadsenSwann2013}, which are generalisations of classical moment maps in symplectic geometry. The authors are not aware of any previous attempt towards a better understanding of the intermediate case, i.e. non  abelian groups of higher cohomogeneity.

For what concerns calibrated geometry, associative and coassociative submanifolds are in general hard to construct. Indeed, they are solutions of a system of non-linear PDEs. However, in the setting above, we have special calibrated submanifolds which are easier to study: the ones that are invariant under a cohomogeneity one symmetry. Indeed, the invariance turns the system of PDEs into a system of ODEs on the set of orbits. This idea was successful on the flat $\R^7$ with the standard $\G2$-structure \cites{HarveyLawson1982,Lotay2005, Lotay2006b} and on the Bryant--Salamon manifold of topology $\Lambda^2_-(S^4)$ and $\Lambda^2_-(\C\P^2)$ for coassociative submanifolds \cites{Kawai2018,KarigiannisLotay2021}. Note that in both cases the $\G2$-structure of the manifold is explicit, and so is the system of ODEs. 

By the local existence and uniqueness theorem for associatives and coassociatives \cite{HarveyLawson1982} (or simply by ODE theory), the calibrated submanifolds constructed in this way do not intersect and are smooth in the principal set of the action. However, this may not be the case in the singular set (i.e., the set where the orbits of the action are lower dimensional). Indeed, there are examples of singular and/or intersecting cohomogeneity one calibrated submanifolds, such as the $\T^2$-invariant special Lagrangian cone in $\C^3$, called Harvey--Lawson cone, which induces a $\T^2$-invariant associative cone in $\R^7$ (see \cites{HarveyLawson1982,KarigiannisLotay2021,Lotay2005,Lotay2006b} for further examples). 

If we consider $\T^3$-invariant coassociatives, Madsen and Swann observed in \cite{MadsenSwann2018} that the multi-moment maps related to the $\T^3$-action are first integrals of the coassociative system, which completely determine the desired submanifolds for dimensional reasons. Afterwards, the connection between non-abelian multi-moment maps and calibrated submanifolds was investigated by Karigiannis--Lotay \cite{KarigiannisLotay2021} and the second named author \cite{Trinca2023} on the $\G2$ Bryant--Salamon manifolds and on the $\Spin(7)$ Bryant--Salamon manifold, respectively. 

Another method used on the Bryant--Salamon spaces $\Lambda^2_-(S^4)$ and $\Lambda^2_-(\C\P^2)$ was to look for calibrated submanifolds which are (possibly twisted) vector subbundles over suitable submanifolds of the zero section \cites{KarigiannisLeung2012,KarigiannisMinoo2005}. Neither the cohomogeneity one nor the vector subbundle technique were adapted to the Bryant--Salamon manifolds of topology $S^3\times\R^4$, where the only known calibrated submanifolds were the zero section, which is associative, and the fibres over a given point, which are coassociatives. To the best knowledge of the authors, the last idea used to construct non-trivial examples of complete calibrated submanifolds in non-flat and non-compact $\G2$ manifolds is by using fixed sets of involutions \cite{KovalevNordstrom2010}. 

Note that even though we lose the calibrated condition, hence the volume minimizing property, the notion of associative and coassociative submanifolds makes sense and has been studied for weaker notions of $\G2$ manifolds, such as closed, co-closed or nearly-parallel $\G2$ manifolds (cfr. \cites{BallMadnick2020,BallMadnick2021,BallMadnick2022,Kawai2015, Lotay2012} and references therein). 

  An additional important aspect of manifolds with special holonomy, which we only tangentially touch upon in this paper, is finding and making use of calibrated fibrations. These objects are not only interesting from a mathematical perspective but should also play a crucial role in mathematical physics (cfr. the SYZ conjecture \cite{StromingerYauZaslow1996} and its generalizations \cite{GukovYauZaslow2003}). For this reason, calibrated fibrations in manifolds of special holonomy have been widely studied by both communities (e.g. \cites{Acharya1998, Baraglia2010,Donaldson2017,KarigiannisLotay2021, YangLi2019, LeeLeung2009,Trinca2023}).
   
\subsection{Main results} In this work, we investigate $\G2$ manifolds endowed with a structure-preserving, cohomogeneity two action of the non-abelian Lie group $\T^2\times\SU(2)$, and the related calibrated geometry. Note that there are a lot of $\G2$ manifolds with such a group action. For instance, the large class of examples constructed by Foscolo--Haskins--Nordstr\"om in \cite{FoscoloHaskinsNordstrom2021b} (FHN manifolds for brevity) has the desired symmetry, in fact, they admit a $\SU(2)\times\SU(2)\times\U(1)$ cohomogeneity one and structure-preserving action. Moreover, all simply-connected complete $\G2$-manifolds with $\SU(2)\times\SU(2)\times\U(1)$-symmetry arise in this way \cite{FoscoloHaskinsNordstrom2021b}*{Theorem 7.3}. Special elements of this family are the Bryant--Salamon manifold of topology $S^3\times\R^4$ and the asymptotically locally conical manifolds constructed by Bogoyavlenskaya \cite{Bogoyavlenskaya2013}, which were previously predicted by Brandhuber--Gomis--Gubser--Gukov \cite{BrandhuberGomisGubserGukov2001}. Apart from these, which have symmetry group bigger than $\T^2\times\SU(2)$, one can find examples with exactly a $\T^2\times\SU(2)$-action of cohomogeneity two in \cite{Foscolo2021}*{Theorem 4.12}. 
	In the co-closed case, Alonso has recently constructed examples of $\G2$ manifolds with $\SU(2)\times\SU(2)\times\U(1)$-symmetry \cite{Alonso2022}.
  
As a first step, we study the stabilizer subgroups that can arise in this setting (\cref{thm: class stab groups}). Then we give a local characterization of such manifolds in the principal set (\cref{thm: gibbons-hawking}).
\begin{theorem*}
    Let $(M,\vphi)$ be a $\G2$ manifold with a $\T^2\times\SU(2)$ cohomogeneity-two action. In the principal set, it can be locally reconstructed from two nested systems of ODEs and a suitable two-form, representing the curvature of a $\T^2$-bundle. 
\end{theorem*}
Afterwards, we consider $\T^2\times\Id_{\SU(2)}$-invariant associatives, $\T^3\cong\T^2\times S^1$-invariant coassociatives and $\Id_{\T^2}\times \SU(2)$-invariant coassociatives. In particular, we give a nice characterization of these objects in the $\T^2\times\SU(2)$-quotient of the principal set (\cref{thm: associatives as level sets in B}, \cref{thm: T3-invariant coassociatives in the quotient} and \cref{thm: co-associatives as level sets in B}), which is a surface locally parametrized by the $\T^2$-invariant associatives and the $\T^3$-invariant coassociatives (\cref{cor: Associative/coassociative parametrization}). In the associative case, we also give a characterization in the singular set (\cref{thm: associatives singular set}). Along the way (\cref{cor: Associative fibrations corollary}), we prove that, under some mild topological conditions, the $\T^2$-invariant associatives form an associative fibration, in the same sense as in \cites{KarigiannisLotay2021,Trinca2023}.

We then study the regularity of such submanifolds and we deduce the following (cfr. \cref{thm: regularity associatives}, \cref{thm: regularity T3invariant} and \cref{thm: regularity SU(2)coassociatives}):
\begin{theorem*}
     Let $(M,\vphi)$ be a $\G2$ manifold with a $\T^2\times\SU(2)$ cohomogeneity-two action. Then $\T^2\times\Id_{\SU(2)}$-invariant $\vphi$-calibrated integer rectifiable currents and $\Id_{\T^2}\times\SU(2)$-invariant $\ast\vphi$-calibrated integer rectifiable currents are smooth, while $\ast\vphi$-calibrated integer rectifiable currents that are invariant under $\T^2\times S^1$ for any $S^1$-subgroup of $\SU(2)$ can admit singularities with a tangent cone modelled on the Harvey--Lawson cone times $\R$.
\end{theorem*}

We also outline when our results can be extended to manifolds with closed or co-closed $\G2$-structures (cfr. \cref{rmk: extension associatives in manifolds with torsion}, \cref{rmk: T3-invariant coassociatives extension G2 manifolds with torsion} and \cref{rmk: SU(2)-invariant coassociatives extension G2 manifolds with torsion}).

We conclude by applying the aforementioned discussion to the FHN manifolds and to the Bryant--Salamon manifolds of topology $S^3\times\R^4$. In particular, we obtain new large families of complete $\T^2$-invariant associatives (\cref{thm: associative manifolds FHN} and \cref{thm: Associatives in the Bryant--Salamon}). 
\begin{theorem*}
    Let $(M,\vphi)$ be one of the complete $\G2$ manifolds with $\SU(2)\times\SU(2)\times\U(1)$-symmetry constructed by Foscolo--Haskins--Nordstr\"om in \cite{FoscoloHaskinsNordstrom2021b}. For every $\T^2\cong \Id_{\SU(2)}\times \U(1)\times\U(1)< \SU(2)\times\SU(2)\times\U(1)$ (or $\T^2\cong \U(1)\times \Id_{\SU(2)}\times\U(1)< \SU(2)\times\SU(2)\times\U(1)$), there are the following families of distinct complete $\T^2$-invariant associatives: \begin{enumerate}
        \item a $4$-parameter one with elements of topology $\T^2\times\R$,
        \item two distinct $2$-parameter ones whose elements are of topology $S^1\times\R^2$,
        \item depending on the topology of $M$, one single $S^3$ or, alternatively, a $2$-parameter family of topological Lens spaces as elements. 
    \end{enumerate}
    Conversely, any complete associative with such a $\T^2$-symmetry belongs to this list. 
\end{theorem*}
In the BGGG and in the Bryant--Salamon manifolds, Fowdar independently constructed the same family of $S^1\times\R^2$ associatives in \cite{Fowdar2022}.

Furthermore, we extend to $S^3\times\R^4$ the description of (possibly twisted) calibrated subbundles in manifolds of exceptional holonomy started by Karigiannis, Leung and Min-Oo \cites{KarigiannisLeung2012,KarigiannisMinoo2005} (\cref{prop: affine plane}). 

\subsection{Overview of the paper}
Before getting into the main content of this work, we provide, in \cref{sec:Prelim}, a brief introduction to $\G2$ geometry and to the related calibrated submanifolds. Inspired by \cites{KarigiannisLotay2021,Trinca2023}, we also give a definition of calibrated fibrations in which fibres are allowed to be singular and to intersect. 

In \cref{sec: FHN manifolds}, we briefly recall the construction of complete simply-connected non-compact $\G2$ manifolds with $\SU(2)^2\times\U(1)$-symmetry as described by Foscolo--Haskins--Nordstr\"om \cite{FoscoloHaskinsNordstrom2021b}. For convenience, we refer to these objects as FHN manifolds.

In \cref{sec: T^2xSU(2) symmetry section}, we study the geometry of the $\T^2\times\SU(2)$-action. As a first step, we discuss how to take quotients of the Lie group, and of its $\T^2$ or $\SU(2)$ components, so that the action passes to suitable quotients of the $\G2$ manifold. Even though the group is non-abelian, we are able to classify the stabiliser types and the slice action on the normal bundle (\cref{thm: class stab groups}). It turns out that there are no exceptional orbits (i.e., 5-dimensional orbits of non-principal type) and, using the orbit type theorem, we are able to split our manifold into a stratification given by a principal set $M_P$, where the stabilizer is zero-dimensional, and $\mathcal{S}_i$ for $i=1,2,3,4$, where the stabilizer is $i$-dimensional. Finally, we untangle the definition of multi-moment maps \cite{MadsenSwann2013}*{Definition 3.9} for this group action, and we establish their invariance and equivariance.

Afterwards, in \cref{sec: local characterization}, we investigate the local structure of $\G2$ manifolds with the given cohomogeneity two symmetry. In our setting, we independently consider the $\T^2$ and the $\SU(2)$ factors as follows. Madsen and Swann \cite{MadsenSwann2013} showed that, under the presence of a $\T^2$-symmetry, Hitchin's flow preserves the level sets of the $\T^2$ moment map $\nu$, and the quotient $\chi_t = \nu^{-1}(t)/\T^2$ admits a coherent tri-symplectic structure. They also showed how to reconstruct the $\G2$ manifold with $\T^2$-symmetry from such a four manifold. In our setup, $\chi_t$ inherits an additional $\SU(2)$-symmetry. We classify these tri-symplectic structures as solutions of a matrix valued ODE system. In \Cref{thm: gibbons-hawking}, we summarise these results and state that, finding a $\G2$ manifold with $\T^2 \times \SU(2)$-symmetry, decomposes into solving the ODE system of $\chi_t$, constructing a certain two-form on this space, and solving the rescaled Hitchin's flow equation for the hypersurfaces $\nu^{-1}(t)$. 

In \cref{sec: section T^2-invariant associatives}, we turn our attention to $\T^2$-invariant associatives. The first key observation is that these objects correspond, in the $\T^2$-quotient, to integral curves of a vector field. Since such integral curves respect the stratification induced from \cref{thm: class stab groups}, it is sensible to split our discussion into associatives in the principal set, $M_P$, and associatives in the various strata, $\mathcal{S}_i$, which form the singular set.

Using our knowledge of the possible slice actions, we show in \cref{thm: associatives singular set} that each stratum, $\mathcal{S}_i$, naturally decomposes into smooth $\T^2$-invariant associatives. In the principal part $M_P$, we characterise $\T^2$-invariant associatives as horizontal lifts of a level set on the quotient $B:=M_P/(\T^2 \times \SU(2))$, which is two-dimensional (\cref{thm: associatives as level sets in B}). 

Moreover, we determine under which topological conditions they are fibres of a global fibration map on $M_P$ (\cref{thm: construct trivialization}) and, hence, when they form an associative fibration (\cref{cor: Associative fibrations corollary}). A priori, the $\T^2$-invariant associatives in $M_P$ could approach and intersect the singular set of the $\T^2$-action, where singularities and intersection can occur. However, the aforementioned characterisation allows us to exclude such behaviour, and to conclude, in \cref{thm: regularity associatives}, that all $\T^2$-invariant associatives are smooth. This is particularly interesting because there are classical examples of singular $\T^2$-invariant associatives, e.g. the Harvey--Lawson cone in $\R^7$ with the standard $\G2$-structure \cite{HarveyLawson1982}. It follows that the enhanced symmetry rules out singularities.

Fixing a $\T^3\cong\T^2\times S^1$ inside $\T^2 \times \SU(2)$, we study $\T^3$-invariant coassociatives and $\SU(2)$-invariant coassociatives in \cref{sec: T3-invariant/SU(2)-invariant coassociatives}. In general, $\T^3$-invariant coassociatives are easy to find. Indeed, Madsen and Swann showed in \cite{MadsenSwann2018} that they are the level sets of $\T^3$ multi-moment maps. Similarly to the $\T^2$-invariant associatives case, we can also characterize them them in the quotient $B$ (\cref{thm: co-associatives as level sets in B}). The "surviving" multi-moment map forms, together with the defining function of the $\T^2$-invariant associatives, a local orthogonal parametrization of $B$, which we call associative/coassociative in \cref{cor: Associative/coassociative parametrization}. Unfortunately, $\SU(2)$-invariant coassociatives do not have a nice level set description, and only project on $B$ to integral curves of a non-trivial vector field. Using a blow-up argument and some geometric measure theory machinery, which we recall in \cref{sec: blow-up and regularity}, we show that $\SU(2)$-invariant coassociatives are smooth and that $\T^3$-invariant coassociatives can exhibit singularities. All singularities have a tangent cone modelled on the product of the Harvey--Lawson cone with $\R$.

In \cref{sec: section examples}, we apply these ideas to the FHN-manifolds, which are characterized by implicit solutions of an ODE system. Under some conditions, this system extends to a singular initial value, which corresponds to a connected smooth submanifold and it is determined by one of the following Lie groups: $K=\Delta\SU(2)$, $K=\left\{1_{\SU(2)}\right\}\times\SU(2)$ or $K=K_{m,n}$  (see \cref{sec: FHN manifolds} for further details). We compute the various multi-moment maps and we are able to characterise the aforementioned calibrated submanifolds. In particular, in every FHN manifold with $\SU(2)\times\SU(2)\times\U(1)$-symmetry, we find a new $4$-dimensional family of $\T^2$-invariant associatives with topology $\T^2\times\R$ which are bounded away from the singular initial value, and two $S^2$-families of $\T^2$-invariant associatives with the same topology which extend, together with the system, to smooth associatives of topology $S^1\times\R^2$ for every $K$. If the solution extends to an initial value characterized by $K=\Delta\SU(2)$ or $K=\left\{1_{\SU(2)}\right\}\times\SU(2)$, then we have an additional $\T^2$-invariant associative of topology $S^3$. When $K=K_{m,n}$, there is an $S^2$-family of $\T^2$-invariant associatives of topology a lens space depending on $n,m$ and two additional $\T^2$-invariant associatives of topology $S^2\times S^1$. See \cref{thm: associative manifolds FHN} for the precise statement of this result and \cref{fig: image of alpha} for a graphical representation of the submanifolds. Moreover, when the solution extends to the singular initial value, we satisfy the topological conditions of \cref{thm: construct trivialization} and we obtain an associative fibration. As an explicit special case of the FHN manifolds, we consider the Bryant--Salamon space of topology $S^3\times\R^4$ (see \cite{KarigiannisLotay2021}*{Section 3}) and we construct a new family of (possibly twisted) associative vector subbundles over a geodesic of $S^3$. 

It is well-known that all the Bryant--Salamon manifolds are vector bundles with calibrated fibres. In \cite{KarigiannisLotay2021}, Karigiannis and Lotay considered the $\G2$ manifolds with associative fibres, namely $\Lambda^2_-(S^4)$ and $\Lambda^2_-(\mathbb{CP}^2)$, and constructed coassociative fibrations on them. In some sense, they interchanged the role of associative and coassociative submanifolds. As a byproduct of \cref{cor: Associative fibrations corollary}, we obtain the opposite result, i.e. we construct on the natural coassociative fibre bundle, $S^3\times\R^4$, an associative fibration. We visualize this fibration in \cref{figure levelsets BS}.

\subsection*{Acknowledgements} The authors wish to express their gratitude to Lorenzo Foscolo and Jason D. Lotay for all the valuable feedback, enlightening conversations and encouragement. The authors would like to thank Jakob Stein for explaining to them important aspects of the construction in \cite{FoscoloHaskinsNordstrom2021b} and for the helpful comments on a first draft of this work. The authors also thank the anonymous referees for useful suggestions and for greatly improving the clarity of the paper.

\section{Preliminaries}\label{sec:Prelim}
In this section, we provide the basic definitions and properties of $\G2$ manifolds, associative submanifolds and coassociative submanifolds. \subsection{\texorpdfstring{$\G2$}{G2}  Manifolds} The linear model we consider for a $\G2$ manifold is $\R^7\cong\R^3\oplus\R^4$ parametrized by $(x_1,x_2,x_3)$ and $(a_0,a_1,a_2,a_3)$, respectively. On $\R^7$, we consider the associative $3$-form $\vphi_0$:
\begin{align*}
	\vphi_0=dx_1\wedge dx_2\wedge dx_3+\sum_{i=1}^3 dx_i\wedge \Omega_i,
\end{align*} 
where the $\Omega_i$s are the standard ASD two-forms of $\R^4$ endowed with the Euclidean metric, i.e., $\Omega_i=da_0\wedge da_i-da_j\wedge da_k$ for $(i,j,k)$ cyclic permuation of $(1,2,3)$.
The Hodge dual of $\vphi_0$ in $\R^7$ is also of great geometrical interest:
\begin{align*}
	\ast\vphi_0=da_0\wedge da_1 \wedge da_2\wedge da_3-\sum_{i=1}^3 dx_j\wedge dx_k\wedge \Omega_i,
\end{align*}
where $(i,j,k)$ is again a positive permutation of $(1,2,3)$. 

Since the stabilizer of $\vphi_0$ is isomorphic to $\G2$, the automorphism group of ${\mathbb{O}}$, we can see $(\R^7,\vphi_0)$ as the linear model for manifolds with $\G2$-structure group.

\begin{definition}
	Let $M$ be a manifold and $\vphi$ a $3$-form on $M$. We say that $\vphi$ is a $\G2$-structure on $M$ if at each point $x\in M$ there exists a linear isomorphism $p_x:\R^7\to T_x M$ which identifies $\vphi_0$ with $\restr{\vphi}{x}$, i.e., $p_x^\ast {\vphi}=\vphi_0$.
\end{definition}

A $\G2$-structure $\vphi$ induces a metric $g_\vphi$ and an orientation $\vol_\vphi$ on $M$ satisfying:
\begin{align}\label{eqn: G2 structure induces metric and volume}
(u\lrcorner\vphi)\wedge (v\lrcorner\vphi)\wedge\vphi=-6g_\vphi(u,v)\vol_\vphi,
\end{align}
for all $u,v\in T_xM$ and all $x\in M$.
 This makes $p_x$ an orientation preserving isometry. From $g_\vphi$ and $\vol_\vphi$, one can also construct the coassociative $4$-form $\ast_\vphi \vphi$.

\begin{definition}
	Let $M$ be a manifold and let $\vphi$ be a $\G2$-structure on $M$. We say that $(M,\vphi)$ is a $\G2$ manifold if $\vphi$ and $\ast_{\vphi} \vphi$ are closed.
\end{definition}
This terminology is justified by the theorem of Fern\'andez and Gray \cite{FernandezGray1982}, which states that
in this case, the Riemannian holonomy group of $(M,g_\vphi)$ is contained in $\G2$. Every $\G2$ manifold is Ricci-flat.

The octonionic structure on the tangent space equips the tangent bundle with a natural cross product. 
\begin{definition}
	Let $(M,\vphi)$ be a manifold with a $\G2$-structure. The cross product on the tangent bundle $\times_\vphi$ is defined as follows:
	\begin{align*}
		\times_{\vphi}:& TM\times TM\to TM\\
		&\hspace{0pt} (U,V)\to (V\lrcorner U\lrcorner \vphi)^\#,
	\end{align*}
	where $\#$ denotes the Riemannian musical isomorphism.
	\end{definition}

\subsection{Associative and coassociative submanifolds} Harvey and Lawson \cite{HarveyLawson1982} showed that $\vphi$ and $\ast\vphi$ have co-mass equal to one. It follows that if $(M,\vphi)$ is a $\G2$ manifold, then $\vphi$ and $\ast\vphi$ are calibrations. 
\begin{definition}
	Let $F\subset (\R^7,\vphi_0)$ be a $3$-dimensional vector subspace. The subspace $F$ is an associative plane if $\restr{\vphi_0}{F}=\vol_F$. A submanifold $L$ of a $\G2$ manifold $(M,\vphi)$ is associative if it is calibrated by $\vphi$, i.e. for every $x\in L$ the subspace $T_xL$ is an associative plane in $T_xM$.
\end{definition}
\begin{definition}
	Let $F\subset (\R^7,\vphi_0)$ be a $4$-dimensional vector subspace. The subspace $F$ is a coassociative plane if $\restr{\ast\vphi_0}{F}=\vol_F$. A submanifold $\Sigma$ of a $\G2$ manifold $(M,\vphi)$ is coassociative if it is calibrated by $\ast\vphi$, i.e. for every $x\in \Sigma$ the subspace $T_x\Sigma$ is a coassociative plane in $T_xM$.
\end{definition}
\begin{remark}
	A submanifold $\Sigma$ is associative or coassociative if and only if $T_x\Sigma$ is an associative or a coassociative plane of $(\R^7,\vphi_0)$ for every $x\in \Sigma$ under the isomorphism $p_x$. 
\end{remark}
We now state some well-known properties of associative and coassociative planes which will be useful in the discussion below. We can translate this statement to the tangent space $(T_x M,\restr{\vphi}{x})$ of a $\G2$ manifold through $p_x$. 
\begin{proposition}[Harvey--Lawson \cite{HarveyLawson1982}]\label{prop: characterization associative planes}
	Let $ F\subset (\R^7,\vphi_0)$ be a $3$-dimensional subspace. Then the following are equivalent:
	\begin{enumerate}
		\item $F$ is an associative plane,
		\item  $F^{\perp}$ is a coassociative plane,
		\item if $u,v\in F$, then  $u\times_{\vphi_0}v\in F$,
		\item if $u\in F$ and $v\in F^\perp$, then $u\times_{\vphi_0}v\in F^\perp$,
		\item if $u,v\in F^{\perp}$, then  $u\times_{\vphi_0}v\in F$,
		\item if $u,v,w\in F$, then $w\lrcorner v\lrcorner u\lrcorner\ast_{\vphi_0}\vphi_0=0$,
		\item if $u,v,w\in F^{\perp}$, then $w\lrcorner v\lrcorner u\lrcorner\vphi_0=0$.
		
	\end{enumerate}
Moreover, it follows that for every $u,v$ linearly independent vectors of $\R^7$ there exists a unique associative plane containing them. Analogously, if $u,v, w$ are linearly independent vectors of $\R^7$ such that $\vphi_0(u,v,w)=0$ there exists a unique coassociative plane containing them.
\end{proposition}

\subsubsection{Local existence and uniqueness}
In the rest of this paper, we will make extensive use of the following local existence and uniqueness theorem for associative and coassociative submanifolds. The proof relies on Cartan-K\"ahler theorem.

\begin{theorem}[Harvey--Lawson \cite{HarveyLawson1982}*{Section IV.4}] \label{thm: local existence and uniqueness}
	Let $N$ be a real analytic submanifold of a $\G2$ manifold $(M,\vphi)$. If $N$ is $2$-dimensional, then there exists a unique associative real-analytic submanifold $L$ such that $N\subset L$. If $N$ is $3$-dimensional and $\restr{\vphi}{N}\equiv 0$, then there exists a unique coassociative real-analytic submanifold $\Sigma$ such that $N\subset \Sigma$.
\end{theorem}

When a $\G2$ manifold $(M,\vphi)$ admits a Lie group action $G$ with $2$-dimensional principal orbits, \cref{thm: local existence and uniqueness} applied to any such $G$-orbit yields (locally) the unique $G$-invariant associative submanifold passing through it. Obviously, we can then extend any such local associative submanifold $L$ until we "hit" the singular part of the $G$-action. There, $L$ can intersect another associative and/or admit a singularity. Conversely, any $G$-invariant $\vphi$-calibrated integer rectifiable current intersecting the principal part of the action admits such description. A similar discussion works for coassociatives, i.e., $\ast\vphi$-calibrated integer rectifiable currents. In this case, the principal $G$-orbits need to be $3$-dimensional and $\vphi$ must vanish when restricted to them.

\begin{remark}
	Note that in the $G$-invariant case, \cref{thm: local existence and uniqueness} is equivalent to the local existence and uniqueness for ODEs in the quotient space of the principal part.
\end{remark}

\subsubsection{Calibrated fibrations}
Inspired by \cites{KarigiannisLotay2021, Trinca2023}, we consider a definition of calibrated fibrations where fibres are allowed to be singular and to intersect. 

	\begin{definition}\label{def: calibrated fibration definition}
Let $(M,\alpha)$ be a $n$-manifold with a $k$-calibration $\alpha$. The manifold $M$ admits an $\alpha$-calibrated fibration if there exists a family of  $\alpha$-calibrated submanifolds $N_b$ (possibly singular) parametrized by a $(n-k)$-dimensional topological space $\mathcal{B}$ satisfying the following properties: 
\begin{itemize}
\item $M$ is covered by the family $\{N_b\}_{b\in\mathcal{B}}$,
\item there exists an open dense set $\mathcal{B}^{\circ}\subset\mathcal{B}$ such that $N_b$ is smooth for all $b\in \mathcal{B}^{\circ}$,
\item there exists an open dense subset $M'\subset M$, an open dense set $\mathcal{B}'\subset \mathcal{B}$ which admits the structure of a smooth manifold and a smooth fibre bundle $\pi: M'\to \mathcal{B}'$ with fibre $N_b$ for all $b\in\mathcal{B}'$.
\end{itemize}
\end{definition}

\begin{remark}
The set $M\setminus M'$ is where the calibrated submanifolds can intersect and can be singular. When we restrict the calibrated submanifolds to $M'$, these can cease to be complete and they can have a different topology from the original ones. 
\end{remark}

 \section{The Foscolo--Haskins--Nordstr\"om manifolds} \label{sec: FHN manifolds}
In this section, we recall the construction of complete simply-connected non-compact $\G2$ manifolds due to Foscolo, Haskins and Nordstr\"om in \cite{FoscoloHaskinsNordstrom2021b}. For brevity, we will refer to them as the FHN manifolds. Note that this is not standard terminology. It is costumary to distinguish three different subfamilies inside the manifolds constructed by Foscolo--Haskins--Nordstr\"om: $\mathbb{B}_7$ (predicted in \cite{BrandhuberGomisGubserGukov2001} and previously constructed in \cite{Bogoyavlenskaya2013}), $\mathbb{C}_7$ (predicted in \cites{Brandhuber2002, CveticGibbonsLuPope2004}) and $\mathbb{D}_7$ (predicted in \cites{Brandhuber2002, CveticGibbonsLuPope2002}).

As we will apply the theory we develop in \cref{sec: section T^2-invariant associatives,sec: T3-invariant/SU(2)-invariant coassociatives} to these spaces, we believe that it is useful to fix some key notation here.

\subsection{The topology of the FHN manifolds} Let $(M,\vphi)$ be a non-compact, simply-connected $\G2$ manifold, with a structure-preserving $\SU(2)\times\SU(2)$ cohomogeneity one action. Then it is well-known that $M/\SU(2)\times\SU(2)$ is an open or half-closed interval $I$, and hence, the cohomogeneity one structure can be encoded by a pair of closed subgroups: $K_0\subset K\subset \SU(2)\times\SU(2)$, which are referred to as the group diagram of $M$. In particular, $\SU(2)\times\SU(2)/K_0$ is diffeomorphic to the principal orbits of the $\SU(2)\times\SU(2)$-action and corresponds to the interior of $I$, while $\SU(2)\times\SU(2)/K$ is diffeomorphic to the singular orbit and corresponds to the boundary of $I$, if it exists.

In the case of our interest, we either have $K_0=\{1_{\SU(2)\times\SU(2)}\}$ or $K_0=K_{m,n}\cap K_{2,-2}$, where $m,n$ are coprime integers and $K_{m,n}\cong U(1)\times\Z_{\textup{gcd}(n,m)}$ is defined by:
\[
K_{m,n}:=\left\{(e^{i\theta_1},e^{i\theta_2})\in\T^2: e^{i(m\theta_1+n\theta_2)}=1\right\}<\SU(2)\times\SU(2),
\]
where $\T^2$ is the maximal torus in $\SU(2)\times\SU(2)$.
If $m,n$ are coprime the isomorphism between $K_{m,n}< \SU(2)\times\SU(2)$ and $U(1)$ is:
\begin{align}\label{eqn: isom U(1) Kmn}
    e^{i\theta}\mapsto (e^{in\theta},e^{-im\theta}),
\end{align}
moreover, $K_{m,n}\cap K_{2,-2}\cong\Z_{2\av{m+n}}$.
Up to automorphisms of $\SU(2)\times\SU(2)$, the subgroup $K$ determining the singular orbit $\SU(2)\times\SU(2)/K$ is one of the following:
\[
\Delta\SU(2), \quad \left\{1_{\SU(2)}\right\}\times\SU(2), \quad K_{m,n},
\]
where $\Delta\SU(2)$ denotes the $\SU(2)$ sitting diagonally in $\SU(2)\times\SU(2)$. Note that the singular orbit is diffeomorphic to $S^3$ for the first two cases, and to $S^2\times S^3$ for the third one. 

\subsection{The \texorpdfstring{$\G2$}{G2}-structure}\label{sec: appendix G2 structure FHN} We now describe the $\G2$-structure on the principal part of $M$, diffeomorphic to $(\SU(2)\times\SU(2))/K_0 \times \textup{Int}(I)$. 

Consider on $\SU(2)\times\SU(2)$ the basis $\{e_1,e_2,e_3,f_1,f_2,f_3\}$ of left-invariant $1$-forms satisfying:
\[
de_i=2e_j\wedge e_k, \quad df_i=2f_j\wedge f_k, 
\]
and denote by $E_1,E_2,E_3,F_1,F_2,F_3$ the dual vector fields. On the principal part of $M$, these can be explicitly described as follows:
\begin{align*}
   E_1(p,q,r)&=-(pi,0,0),\quad  E_2(p,q,r)=-(pj,0,0),\quad E_3(p,q,r)=-(pk,0,0), \\
    F_1(p,q,r)&=-(0,qi,0),\quad  F_2(p,q,r)=-(0,qj,0),\quad F_3(p,q,r)=-(0,qk,0),
\end{align*}
where the product is by quaternionic multiplication. Let $c_1,c_2\in\R$ and let $a_1,a_2,a_3$ be three functions only depending on the interval $I$. The following closed 3-form on $(\SU(2)\times\SU(2))/K_0 \times \textup{Int}(I)$:
\begin{align}\label{eq: phi FHN}
\vphi=-8c_1 e_1\wedge e_2\wedge e_3-8c_2 f_1\wedge f_2\wedge f_3+ 4d(a_1 e_1\wedge f_1+ a_2 e_2\wedge f_2+a_3e_3\wedge f_3)
\end{align}
is a $\G2$-structure such that the interval $I$ is the arc-length parameter along a geodesic meeting orthogonally all the principal orbits if and only if the following conditions are satisfied:
\[
\dot{a}_i>0, \quad \Lambda(a_1,a_2,a_3)<0, \quad 2\dot{a}_1 \dot{a}_2 \dot{a}_3=\sqrt{-\Lambda(a_1,a_2,a_3)},
\]
where 
\begin{align*}
    \Lambda(a_1,a_2,a_3)=&a_1^4+a_2^4+a_3^4-2a_1^2 a_2^2-2a_2^2a_3^2-2a_3^2a_1^2+4(c_1-c_2)a_1 a_2 a_3+ \\
    &+2c_1 c_2(a_1^2+a_2^2+a_3^2)+c_1^2 c_2^2.
\end{align*}
Furthermore, if $K_0=K_{m,n}\cap K_{2,-2}$, we require $a_2=a_3$ unless there exists a $d\in\Z$ such that $(d+1)m+(d-1)n=0$. 
\begin{remark}
    Under these conditions, the interval $I$ is the arc-length parameter along a geodesic meeting all the principal orbits orthogonally. 
\end{remark}

The torsion free condition becomes the Hamiltonian system associated to the potential:
\[
H(x,y)=\sqrt{-\Lambda(y_1,y_2,y_3)}-2\sqrt{x_1 x_2 x_3},
\]
where $y_i=a_i$ and $x_i=\dot{a}_j \dot{a}_k$ for every $(i,j,k)$ cyclic permutation of $(1,2,3)$. 
If $t$ denotes the parametrization of $I$, then the dual form of $\vphi$ is given by:
\begin{equation}\label{eq: astphi FHN}
\begin{aligned}
    \ast\vphi=&16\sum_{i=1}^3 \dot{a}_j \dot{a}_k e_j\wedge f_j\wedge e_k\wedge f_k+\\
    &+\frac{8}{\sqrt{-\Lambda}} dt\wedge \bigg( (2a_1 a_2 a_3-c_1(a_1^2+a_2^2+a_3^2+c_1 c_2))e_1\wedge e_2\wedge e_3 \\
     &\hspace{70pt}+(2a_1 a_2 a_3+c_2(a_1^2+a_2^2+a_3^2+c_1 c_2))f_1\wedge f_2\wedge f_3\\
     &\hspace{70pt}+\sum_{i=1}^3 \big((a_i (a_i^2-a_j^2-a_k^2+c_1 c_2)-2c_2 a_j a_k)e_i\wedge f_j\wedge f_k \\
     &\hspace{100pt}+(a_i(a_i^2-a_j^2-a_k^2+c_1 c_2)+2c_1 a_ja_k)f_i\wedge e_j\wedge e_k\big)\bigg).
     \end{aligned}
\end{equation}

\textbf{Enhanced symmetry.} We now restrict our discussion to the case where $a_2=a_3$. Under this additional condition, the symmetry of $(\SU(2)\times\SU(2))/K_0 \times \textup{Int}(I)$ becomes $\SU(2)\times\SU(2)\times \U(1)$, where the action of $(\gamma_1,\gamma_2,\lambda)\in \SU(2)\times\SU(2)\times\U(1)$ on $([p,q],t)\in (\SU(2)\times\SU(2))/K_0 \times \textup{Int}(I)$ is as follows:
\begin{align}\label{eq: SU(2)^2xU(1) FHN action}
(\gamma_1,\gamma_2,\lambda)\cdot([p,q],t)=([\gamma_1 p \overline{\lambda},\gamma_2 q \overline{\lambda}],t),
\end{align}
where $\lambda$ is given by the $\U(1)<\SU(2)$ generated by quaternionic multiplication by $i$. 
\begin{remark}
    Note that this enhanced symmetry allows us to find $\T^2\times\SU(2)$ subgroups of the automorphism group of $(M,\vphi)$.
\end{remark}

Under this enhanced symmetry, we denote by $a:=a_2=a_3$ and $b:=a_1$, and the form of ${\Lambda}(a,b)$ simplifies to:
\begin{align}\label{eqn: biglambda}
-{\Lambda}(a,b)=4a^2(b-c_1)(b+c_2)-(b^2+c_1 c_2)^2,
\end{align}
and the same holds for the Hamiltonian system, which becomes:
\begin{align*}
    &\dot{x}_1=-\frac{\Lambda_a(y_1,y_2)}{4\sqrt{-\Lambda(y_1,y_2)}}, \quad \dot{x}_2=-\frac{\Lambda_b(y_1,y_2)}{2\sqrt{-\Lambda(y_1,y_2)}},\\
    &\dot{y}_1=\frac{x_1 x_2}{\sqrt{x_1^2 x_2}}, \hspace{53pt} \dot{y}_2=\frac{x_1^2}{\sqrt{x_1^2 x_2}},
\end{align*}
where $y_1=a, y_2=b, x_1=\dot{a}\dot{b}, x_2=\dot{a}^2$ and $\Lambda_a$, $\Lambda_b$ denote the derivative of $\Lambda(a,b)$ with respect to the first or the second component, respectively. 
\begin{remark}\label{rmk: key remark assfibr FHN}
    From $-{\Lambda}(a,b)>0$, we deduce that $a, b-c_1, b+c_2$ have definite sign, and hence, $\dot{x}_1$ has definite sign as well.
\end{remark}

\begin{example} The Bryant--Salamon manifolds can be seen as special examples of FHN manifolds such that, for some $c>0$:
\begin{align}\label{eqn: Bryant--Salamon as FHN}
    a_1=a_2=a_3=\frac{\sqrt{3}}{2}r^2, \quad c_1=-\frac{3}{8}\sqrt{3}c, \quad c_2=0, \quad K= \left\{1_{\SU(2)}\right\}\times\SU(2) 
\end{align}
or 
\[
a_1=a_2=a_3=\frac{1}{6}r^3-\frac{1}{3}c^3, \quad c_1=-c_2=c^3, \quad K=\Delta\SU(2), 
\]
where $r(t)$ is a reparametrization of $t$ such that $dr/dt=1/2 (c+r^2)^{1/6}$ in the first case and $dr/dt=1/\sqrt{3} \sqrt{1-8c^3r^{-3}}$ in the second case. \end{example}
\subsection{Extension to the singular orbit and forward completeness}\label{sec: extension FHN to singular orbit} Now, we state under which conditions the $\G2$-structure extends smoothly to the singular orbit and when it is forward complete. 

First, we know from the slice theorem that a neighborhood of the singular orbit $\SU(2)\times\SU(2)/K$ is equivariantly diffeomorphic to a small disk bundle of:
\[
(\SU(2)\times\SU(2))\times_{K} V,
\]
for some vector space $V$ endowed with a representation of $K$. We now summarise when the $\G2$-structure defined in \cref{eq: phi FHN} extends smoothly to the zero section of such a bundle (cfr. \cite{FoscoloHaskinsNordstrom2021b}*{Proposition 4.1}).

\textbf{Case 1 (\texorpdfstring{$K=\Delta\SU(2)$}{K=DSU2}).} In this case, $V=\C^2$ and $\SU(2)$ acts in the usual way on it. The $\SU(2)\times\SU(2)$-invariant $\G2$-structure defined above extends smoothly to the zero-section if and only if:
\begin{enumerate}
    \item $c_1+c_2=0$,
    \item the functions $\{a_i\}$ are even and have the following development near $0$: $a_i(t)=c_1+\frac{1}{2}\alpha t^2+O(t^4)$ for some $\alpha\in\R$,
    \item $8\alpha^3=c_1>0$.
\end{enumerate}
\textbf{Case 2 (\texorpdfstring{$K=\{1_{\SU(2)}\}\times\SU(2)$}{K=1SU2xSU2}).} As in the previous case, $V=\C^2$ and $\SU(2)$ acts in the usual way on it. The $\G2$-structure defined above extends smoothly to the zero-section if and only if:
\begin{enumerate}
    \item $c_2=0$,
    \item the functions $\{a_i\}$ are even and have the following development near $0$: $a_i(t)=\frac{1}{2}\alpha_i t^2+O(t^4)$ for some $\alpha_i\in\R^+$,
    \item $8\alpha_1\alpha_2\alpha_3=-c_1>0$.
\end{enumerate}

\textbf{Case 3 (\texorpdfstring{$K=K_{m,n}$}{K=Kmn}).} In this situation, $V=\R^2$ and $K_{m,n}\cong\U(1)$ acts on it with weight $2\av{m+n}$. The $\G2$-structure defined above extends smoothly to the zero-section if and only if:
\begin{enumerate}
    \item $mn>0$,
    \item $c_1=-m^2r_0^3$ and $c_2=n^2r_0^3$ for some $r_0\in\R\setminus\{0\}$,
    \item the function $a_1$ is even and satisfies: $a_1(0)=mnr_0^3$, $dot{a}_1(0)>0$,
    \item the function $a_2+a_3$ is odd and satisfies: $\dot{a}_2(0)+\dot{a}_3(0)>0$,
    \item we either have $a_2=a_3$ or $m=n=\pm1$; if the $a_2$ and $a_3$ do not coincide, then their difference is an even function with $\av{a_2(0)-a_3(0)}<2\av{r_0}^3$.
\end{enumerate}

The forward completeness of the local solutions constructed above and the metric completeness is discussed in \cite{FoscoloHaskinsNordstrom2021b}*{Section 6, Section 7} for the case we have the enhanced symmetry $\SU(2)\times\SU(2)\times\U(1)$. Moreover, they showed that the complete $\G2$ manifolds they obtain are all the possible complete $\G2$-manifolds with $\SU(2)\times\SU(2)\times \U(1)$-symmetry.

\section{\texorpdfstring{$\G2$}{G2} manifolds with \texorpdfstring{$\T^2\times \SU(2)$}{T2xSU2}-symmetry} \label{sec: T^2xSU(2) symmetry section}
In this section, we prove some properties of a $\G2$ manifold $(M,\vphi)$ with a structure-preserving $\T^2\times \SU(2)$-action of cohomogeneity two, i.e. the maximal dimension achieved by the orbits is $5$. We will make extensive use of the theory of differentiable transformation groups (cfr. \cref{app: differentiable transformation groups}).  

If $\Gamma$ represents the kernel of the homomorphism $\T^2\times\SU(2)\to\Aut(M,\vphi)$, we prove that the Lie group $(\T^2\times\SU(2))/\Gamma$, which acts effectively on $(M,\vphi)$, has trivial principal stabilizer. Afterwards, we characterize the group structure and the slice action of each $(\T^2\times\SU(2))/\Gamma$-stabilizer using only its dimension. As a consequence of this technical result, we deduce that there are no exceptional orbits and that the singular set of $(\T^2\times\SU(2))/\Gamma$ "splits" into smooth embedded submanifolds. We conclude the first part of the section by studying the properties of these submanifolds. 

In the second part of the section, we specialize to our setting the notion of multi-moment maps, which were introduced in \cites{MadsenSwann2012, MadsenSwann2013}. Then we study the properties, including invariance and equivariance, that we will need in the rest of the paper.

\subsection{\texorpdfstring{$\T^2\times \SU(2)$}{Lg}-symmetry} \label{sec: T^2xSU(2) symmetry subsection}
To understand the action of $\T^2\times \SU(2)$ on $M$, let $\Gamma$ be the kernel of the homomorphism $\T^2\times \SU(2)\to \mathrm{Aut}(M)$, which is discrete by assumption. Once we rewrite it as $\Gamma=\{(a_i,b_i)\in \T^2\times \SU(2): i\in I\}$, we define $\Gamma_1:=\{a\in \T^2: (a,\Id_{\SU(2)})\in \Gamma\}$ and $\Gamma_2:=\{b\in \SU(2):(\Id_{\T^2},b)\in \Gamma\}$, which are subgroups of $\T^2$ and $\SU(2)$ respectively.

Consider the $\T^2$ action on $M$ given by $\T^2\times \Id_{\SU(2)}<\T^2\times {\SU(2)}$. Since $$\Gamma_1\times \Id_{\SU(2)}=(\T^2\times\Id_{\SU(2)})\cap \Gamma,$$
we see that the action of $\T^2/\Gamma_1$ is effective, and, as $\T^2/\Gamma_1$ is diffeomorphic to $\T^2$, we can assume, without loss of generality, that $\Gamma_1$ is trivial and that the action of $\T^2\cong\T^2\times\Id_{\SU(2)}$ is effective. We denote by $\mathcal{S}$ the singular set of this action, i.e. the complement of the principal set with respect to this action.

Analogously, we have an $\SU(2)$-action on $M$ given by $\SU(2)\cong\Id_{\T^2} \times\SU(2)< \T^2\times\SU(2)$, which induces an effective action of $\SU(2)/\Gamma_2$. 
\begin{remark}
	Observe that $\Gamma$ does not need to be equal to $\Gamma_1\times \Gamma_2$. For instance, if $\Gamma=\{\pm(1,1)\}$, then $\Gamma_1$ and $\Gamma_2$ are trivial.
	\end{remark}

Now, we show that $\Gamma$ is in the center of $\T^2\times \SU(2)$: $Z(\T^2\times\SU(2))=\T^2\times \{\pm 1\}$. 

\begin{lemma}\label{lemma: class discrete stab}
Let $x\in M$ be such that the stabilizer $(\T^2\times\SU(2))_x$ is discrete. Then the stabilizer is a subgroup of the center $Z(\T^2\times\SU(2))$.
\end{lemma}
\begin{proof}
  We show that the adjoint representation of $(\T^2\times\SU(2))_x$ on $\mathfrak{t}^2\oplus \mathfrak{su}(2)$ is trivial, which implies the statement by naturality of the exponential map. 
   
 Let $N$ be the normal space at $x$ of the $\T^2\times\SU(2)$-orbit, whose tangent space is identified with $\LieT^2\oplus\LieSU(2)$ in the usual manner. Then the representation of $(\T^2\times\SU(2))_x$ on $T_xM$ splits as 
\begin{align}
\label{eqn: tangent splitting}
  T_x M=\mathfrak{t}^2\oplus \mathfrak{su}(2)\oplus N,
\end{align}
and coincides with the adjoint representation on the $\LieT^2\oplus\LieSU(2)$ part. Being abelian, the action on $\mathfrak{t}^2$ is trivial and the same holds for the cross product of the $\LieT^2$-generators. This vector is obviously orthognal to $\mathfrak{t}^2\oplus\{0\}$ and, because of \cref{eqn: U1xU2 in N}, to $\{0\}\oplus \LieSU(2)$. We deduce that the cross product of the $\LieT^2$-generators span a linear subspace $N_1$ of $N$. Note that we used that the action of $(\T^2\times\SU(2))_x$ preserves the $\G2$-structure. 

Denote by $N_2$ the orthogonal complement of $N_1$ in $N$, which is invariant under the action. Being an isometry, every element $g\in  (\T^2\times\SU(2))_x$ acts on $N_2$ by multiplication of $\lambda_g$, where $\lambda_g \in \{-1,+1\}$. 

Finally, we show that $\lambda_g$ cannot be $-1$. In order to do so, we consider the map 
$ (\mathfrak{t}^2 \oplus N_1) \otimes N_2 \to \mathfrak{su}(2)$ which is the composition of the cross product and the projection onto the $\mathfrak{su}(2)$ component in the splitting given by \cref{eqn: tangent splitting}. Since $\mathfrak{t}^2\oplus N_1$ is an associative subspace, this map is an isomorphism of representations. Hence, $g$ acts on $\mathfrak{su}(2)$ by multiplication of $\lambda_g$. We conclude because there is no element in $\T^2\times \SU(2)$ whose adjoint action on $\mathfrak{su}(2)$ is multiplication by $-1$. 
\end{proof}
\begin{corollary}\label{cor: Gamma in the centre}
Since $\T^2\times\SU(2)$ acts on $M$ with cohomogeneity two, $\Gamma$ is in the centre of $\T^2\times \SU(2)$. Hence, $\SU(2)/\Gamma_2$ is either $\SU(2)$ or $\SO(3)$.
\end{corollary}
\begin{corollary} \label{cor: principal stabilizer G trivial}
The principal stabilizer of $(\T^2\times\SU(2))/\Gamma$ is trivial.
\end{corollary}
\begin{proof}
As a consequence of \cref{lemma: class discrete stab}, all principal stabilizer subgroups are not only conjugate, but equal to each other. Since the action is effective after the quotient, the principal stabilizer needs to be trivial.
\end{proof}

From now on, we consider the action of $G:=(\T^2\times\SU(2))/\Gamma\leqslant \Aut(M,\vphi)$, and we denote by $M_P$ its principal set. This is going to greatly simplify our arguments: indeed, the $G$-action is effective and with trivial principal stabilizer. 

We will make use of two additional actions induced from the original $\T^2\times\SU(2)$. Let $\tilde{\Gamma}_1:=\{a_i:(a_i,b_i)\in \Gamma\}$ and let $\tilde{\Gamma}_2:=\{b_i:(a_i,b_i)\in \Gamma\}$, which is either trivial or $\{\pm1\}$ by \cref{cor: Gamma in the centre}. We state the following lemma without proof.

\begin{lemma}\label{lemma: GSU(2) action on M/T^2 and GT2 action}
Let $\T^2\cong {\T^2}\times\Id_{\SU(2)}$ acting on $M$. Then there exists an induced action of $G^{\T^2}:=\T^2/\tilde{\Gamma}_1$ on $M_P/(\SU(2)/\Gamma_2)$ which is free. In particular, $M_P/(\SU(2)/\Gamma_2)$ becomes a principal $G^{\T^2}$-bundle over $B:= M_P/G$. Similarly, there exists a $G^{\SU(2)}:=\SU(2)/\tilde{\Gamma}_2$ action induced by $\SU(2)\cong \Id_{\T^2}\times{\SU(2)}$ on $M_P/\T^2$ which is free. As before, $M_P/\T^2$ becomes a principal $G^{\SU(2)}$-bundle over $B$.
\end{lemma}

The various quotients are summarised in the following diagram:
\[
\begin{tikzcd}
                                                      & M_P \arrow[ld, "/{(\mathrm{SU}(2)/{\Gamma_2})}"'] \arrow[rd, "/\mathbb{T}^2"] \arrow[dd, "/G"] &                                                      \\
M_P/{(\mathrm{SU}(2)/{\Gamma_2})} \arrow[rd, "/G^{\mathbb{T}^2}"'] &                                                                                 & M_P/{\mathbb{T}^2} \arrow[ld, "/G^{\mathrm{SU}(2)}"] \\
                                                      & B                                                                               &                                                     
\end{tikzcd}.\]
\subsection{The stratification} \label{sec: subsection stratification}
Applying the orbit type stratification theorem and the principal orbit type theorem to our setting, where $G=(\T^2 \times \SU(2))/\Gamma$ acts effectively on $M$, we see that $M$ decomposes as the union of $G$-orbit types, and there exists one of them which is open and dense in $M$. In this subsection, we study the geometry of the $G$-action to understand this stratification.

To simplify our notation, we fix a point $x\in M$ and denote by $T$ the tangent space of $Gx$ at $x$ and by $N$ its normal space, i.e. the orthogonal complement of $T$ in $T_xM$.

In the discussion of the stratification, we will need the following standard lemma:
\begin{lemma}
\label{lemma: max torus in g2}
Let $\T^2$ be a maximal torus in $\G2$. Then the representation of $\T^2$ on $\R^7$ splits as $V\oplus W_1 \oplus W_2 \oplus W_3$, where $V$ is $1$-dimensional and each $W_i$ is $2$-dimensional. Each $V\oplus W_i$ is an associative subspace of $\R^7$ with respect to $\vphi_0$.
\end{lemma}
\begin{proof}
A maximal torus in $\G2$ induces a splitting $\R^7 = \R \times \C^3$, where $\C^3$ is equipped with its standard Calabi-Yau structure and the torus acts as a maximal torus of $\SU(3)$. A submanifold $\R\times W$ is associative if and only if $W$ is a holomorphic curve, which is clearly the case for the complex linear subspaces $W_i$.
\end{proof}

Recall that $\mathcal{S}$ is the singular set of the $\T^2$-action and, as a consequence of the following theorem, it is also the set where the generators of the $\T^2$-component are linearly dependent, i.e. there are no exceptional orbits (cfr. \cite{MadsenSwann2018}*{Lemma 2.6}). 
\begin{theorem}
\label{thm: class stab groups}
  The dimension of the stabilizer $G_x$ is not bigger than $4$, and,
 \begin{itemize}
  \item if $\dim(G_x)=0$, then $G_x$ is trivial, i.e. there are no exceptional orbits,
  \item if $\dim(G_x)=1$, then $x\notin \mathcal{S}$ and $G_x$ is isomorphic to $\SO(2)$. The action of $G_x$ on $N$ splits as $N_1\oplus N_2$ with $\dim(N_1)=1,\dim(N_2)=2$ where $G_x$ acts trivially on $N_1$ and faithfully by rotations on $N_2$,
  \item if $\dim(G_x)=2$, then $x \in \mathcal{S}$ and the identity component of $G_x$ is isomorphic to $\T^2$ and acts diagonally on $N\cong\C^2$. The $G$-orbit $Gx$ is an associative submanifold of $M$,
  \item if $\dim(G_x)=3$, then $x \notin \mathcal{S}$ and $G_x$ is isomorphic to $\SU(2)$. The action of $ G_x$ on $N$ leaves a $1$-dimensional subspace $N_1\subset N$ invariant and acts on the orthogonal complement $N_2$ via the standard embedding $\SU(2)\to \SO(4)$,
  \item if $\dim(G_x)=4$, then $x\in \mathcal{S}$ and the identity component of $G_x$ is isomorphic to $\U(2)$. The action on the normal bundle $N$ is via the embedding 
  \[U(2)\to \SU(3), \quad A \mapsto \begin{pmatrix} A & 0 \\ 0 & \det{A}^{-1}\end{pmatrix}.\]
 \end{itemize}
 Consequently, the singular set of the $G$-action can be decomposed into $\mathcal{S}_1\cup \mathcal{S}_2 \cup \mathcal{S}_3 \cup \mathcal{S}_4$, where $\mathcal{S}_i$ is the set of points with $i$ dimensional stabilizer.
\end{theorem}
\begin{proof}
The first part of the proposition follows from the fact that the rank of $\mathfrak{t}^2\oplus \mathfrak{su}(2)$ is three, while the rank of $\mathfrak{g}_2$ is two. Hence, since $G_x< \G2$ under the identification of $(T_xM,\vphi_x)\cong(\R^7,\vphi_0)$, the dimension of $G_x$ cannot be equal to 5. 

By the slice theorem, a neighbourhood of $Gx$ is equivariantly diffeomorphic to a neighbourhood of the zero section of 
$G \times _{G_x} N.$
It follows that the representation of $G_x$ on $N$ is faithful. Indeed, every neighbourhood of the orbit $Gx$ intersects $M_P$, on which $G_x$ acts freely because of \cref{cor: principal stabilizer G trivial}.

If $\mathrm{dim}(G_x)=0$, then an argument similar to the one used for \Cref{lemma: class discrete stab} shows that $G_x$ acts trivially on $N$. This means that $G_x$ is trivial by the faithfulness of the $G_x$-action on $N$.

We now consider the case {$\mathrm{dim}(G_x)=1$} and $x\in \mathcal{S}$. This means that $\bar{G}_x=G_x \cap (\T^2\times \Id_{\SU(2)})/\Gamma$ is not trivial and, being a subgroup of $(\T^2\times \Id_{\SU(2)})/\Gamma$, it acts trivially on $T\cong\mathfrak{g}/\mathfrak{g}_x$. Since the cross-product restricted to any $4$-dimensional subspace generates $T_xM$, we deduce that $\bar{G}_x$ acts trivially on all of $T_xM$. This is a contradiction as $\bar{G}_x\leq G_x$ and hence it has to act faithfully on $N$. We have shown that if $\mathrm{dim}(G_x)=1$, then $x \notin \mathcal{S}$. So it remains to show that $G_x$ is isomorphic to $S^1$. Since $x \notin \mathcal{S}$ the intersection of $\mathfrak{t}^2\oplus\{0\}\subset\LieT^2\oplus\LieSU(2)$ with $\mathfrak{g}_x$ is trivial. This means that $\mathfrak{g}/\mathfrak{g}_x$ splits into $\mathfrak{t}^2$, on which $G_x$ acts trivially, and a $2$-dimensional subspace $\mathfrak{m}$. As before, the normal space splits into $N_1\oplus N_2$, where $N_1$ is spanned by the cross product on $\mathfrak{t}^2$ and $N_2$ is its orthogonal complement in $N$. So $G_x$ acts trivially on $N_1$. To summarise, the action of $G_x$ on $T_x M$ splits as 
\[T_xM =\mathfrak{t}^2\oplus \mathfrak{m}\oplus N_1\oplus N_2.\]
The action of $G_x$ is isometric and faithful on the $2$-dimensional space $N_2$. So, $G_x$ is either isomorphic to $\SO(2)$ or to $\mathrm{O}(2)$. In the latter case, there is an element $\tau$ of order two and a subspace $N_3 \subset N_2$ that is fixed by $\tau$. The cross products of $\mathfrak{t}^2\oplus N_1\oplus N_3$ generate all of $T_xM$ so that $\tau$ acts trivially on all of $T_xM$. This is impossible since the action on $N$ must be faithful.

When {$\mathrm{dim}(G_x)=2$}, we first assume, for the sake of contradiction, that $x \notin \mathcal{S}$. 
 Consider the Lie algebra homomorphism $\psi\colon \mathfrak{g}_x\to\mathfrak{su}(2)$ coming from the projection 
$\mathfrak{t}^2\oplus\mathfrak{su}(2) \to \mathfrak{su}(2)$. The image of $\psi$ would be a $2$-dimensional Lie subalgebra of $\mathfrak{su}(2)$ which does not exist. It follows that $x\in \mathcal{S}$ and the identity component of $G_x$ is isomorphic to $\T^2$. Since the action of the identity component of $G_x$ on $T_xM$ splits as $T\oplus N$, we can apply \cref{lemma: max torus in g2} to see that $T$ is isomorphic to $V$ plus one of the $W_i$, for convenience say $W_1$, and $N$ to the sum of $W_2\oplus W_3$ and the statement follows.

We now deal with the {$\mathrm{dim}(G_x)=3$} case. Consider the Lie algebra homomorphism $\psi\colon \mathfrak{g}_x\to\mathfrak{su}(2)$ as above. The image of $\psi$ is a Lie subalgebra of $\mathfrak{su}(2)$, hence, it is either $\mathfrak{su}(2)$ or a $1$-dimensional subalgebra. The second case is impossible: indeed, the condition implies $\mathfrak{t}^2\oplus\{0\}\subset \mathfrak{g}_x$, but $\mathfrak{g}_x$ also intersects $\mathfrak{su}(2)$ in a $1$-dimensional subspace, so $\mathfrak{g}_x\cong \mathfrak{t}^2 \oplus \psi(\mathfrak{g}_x)\cong \mathfrak{t}^3$. This is a contradiction since $\mathfrak{g}_x$ is a subalgebra of $\mathfrak{g}_2$, which has rank two. So $\psi$ is surjective, which means that $\mathfrak{g}_x$ intersects $\mathfrak{t}^2\oplus\{0\}$ transversally.
It remains to show that $G_x$ is diffeomorphic to $\SU(2)$, {which also implies that $x\notin \mathcal{S}$}. As before, $G_x$ acts trivially on $\mathfrak{g}/\mathfrak{g}_x=\mathfrak{t}^2$.
 The cross product of the generators of this $\mathfrak{t}^2$ lies in $N$ and spans a $1$-dimensional subspace $N_1$ on which $G_x$ acts trivially too. On the orthogonal complement $N_2$ of $N_1$ in $N$ the action of $G_x$ is faithful. So $G_x$ acts trivially on an associative three-plane, which means $G_x$ is a subgroup of $\SU(2)$. Since $G_x$ is $3$-dimensional, it is isomorphic to $\SU(2)$ and the action on $N_2$ is isomorphic to the standard action of $\SU(2)$ on $\C^2$. 
 
Finally, we consider {$\dim(G_x)=4$}. Similarly as above, we can show that $T$ is spanned by the generators of the $\T^2$-componenent of the action, it is $1$-dimensional, and it is fixed by $G_x$. The subgroup of $\G2$ that fixes a $1$-dimensional subspace is $\SU(3)$. So, the action of $G_x$ on the $6$-dimensional normal space, $N$, defines an embedding $G_x\to \SU(3)$, yielding a special unitary representation of $G_x$ on $\C^3$. We first show that, when restricted to the identity component, this representation must be reducible.
Indeed, every $4$-dimensional Lie subalgebra of $\mathfrak{g}$ is isomorphic to $\mathfrak{u}(2)=\mathfrak{su}(2)\oplus \mathfrak{u}(1)$. Since $G_x$ is compact, it suffices to show that every complex $3$-dimensional special unitary representation of $\SU(2)\times \U(1)$ is reducible. To see this, denote by $V_k$ the unique $k$-dimensional irreducible representation of $\SU(2)$ and by $W_m$ the representation of $\U(1)$ on $\C$ with weight $m$. All irreducible representations of the direct product $\SU(2)\times \U(1)$ are of the form $V_k\otimes W_m$. Those that are $3$-dimensional, namely $V_3\otimes W_m$, are not special unitary. Since the representation is faithful and special unitary, we conclude that it must be $(V_2 \otimes W_1) \oplus W_{-2}$, i.e. of the desired form. Moreover, the element $(-1,-1)$ acts trivially, so the identity component of $G_x$ must be $(\SU(2) \times \U(1))/{\Z_2} \cong \U(2)$.
\end{proof}

We have just proven that we can decompose the singular set of the $G$-action into four subsets, $\{\mathcal{S}_i\}_{i=1}^4$, which are characterized by having the dimension of the $G$-stabilizer fixed. We now study the properties of these subsets. 

From the proof of \cref{thm: class stab groups}, we can immediately see that the following holds.

\begin{corollary}
    The singular set of the $\T^2$-action $\mathcal{S}$ is $\mathcal{S}_2\cup \mathcal{S}_4$. Either the set $\mathcal{S}_3$ is empty, or $G^{\SU(2)}$ is isomorphic to $\SU(2)$.

\end{corollary}

Using the slice theorem and the slice action which we studied in \cref{thm: class stab groups}, we can also deduce the following. 
\begin{proposition}
\label{prop: dimensions S_i}
Each $\mathcal{S}_i$ is either empty or a smooth embedded submanifold of dimension:
$$\dim(\mathcal{S}_1)=5,\hspace{5pt} \dim(\mathcal{S}_2)=3, \hspace{5pt}\dim(\mathcal{S}_3)=3,\hspace{5pt} \dim(\mathcal{S}_4)=1.$$ 
Moreover, each connected component of $\mathcal{S}_2$ and $\mathcal{S}_4$ is a $G$-orbit.
\end{proposition}
\begin{proof} 
As before, for every point $x\in M$ we denote by $T$ the tangent space of $Gx$ at $x$ and by $N$ its normal space. 

To prove this statement, it is enough to find the linear subspaces of $N$ on which $G_x$ acts trivially. Indeed, if $V_i$ is such a vector subspace for a point $x\in\mathcal{S}_i$ and some $i=1,...,4$, we immediately see from the slice theorem that $\mathcal{S}_i$ is diffeomorphic to $G \times_{G_x} V_i$ in a neighbourhood of $Gx$. It is now clear that $\mathcal{S}_i$ is smooth and of dimension equal to the dimension of $Gx$ plus the dimension of $V_i$. Moreover, if $V_i$ is trivial, then each connected component of $\mathcal{S}_i$ is a $G$-orbit. 

From \cref{thm: class stab groups}, we can extrapolate that $V_2$ and $V_4$ are trivial and that $V_1$ and $V_3$ are $1$-dimensional.
\end{proof}

By considering subgroups of the stabilizer, we can use a similar argument to understand how the various $\mathcal{S}_i$s relate to each other (cfr. \cref{fig: repr how Sis intersect}). In particular, in a neighbourhood of each connected component of $\mathcal{S}_2$, there are two connected components of $\mathcal{S}_1$ whose closure contains the given connected component of $\mathcal{S}_2$. By the slice theorem, such subsets of $\mathcal{S}_1$ correspond to two vector subspaces of the normal bundle on which some $S^1$-subgroup of the stabilzer acts trivially. In a similar spirit, we can see that a connected component of $\mathcal{S}_4$ is close to a connected component of $\mathcal{S}_1$ and of $\mathcal{S}_3$. 

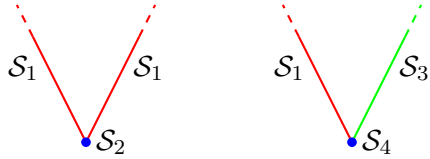
\begin{figure}
    \centering
    \begin{tikzpicture}[scale=0.7]
\draw[thick,red] (0,0)--(1,2);
\draw[thick,red] (0,0)--(-1,2);
\draw[thick,red, dashed] (1,2)--(1.3,2.6);
\draw[thick,red, dashed] (-1,2)--(-1.3,2.6);
\filldraw[blue] (0,0) circle [radius=0.08];
\node at (0,0) [right]{$\mathcal{S}_2$};
\node at (0.7,1.4) [right]{$\mathcal{S}_1$};
\node at (-0.7,1.4) [left]{$\mathcal{S}_1$};

\draw[thick,green] (5,0)--(6,2);
\draw[thick,red] (5,0)--(4,2);
\draw[thick,green, dashed] (6,2)--(6.3,2.6);
\draw[thick,red, dashed] (4,2)--(3.7,2.6);
\filldraw[blue] (5,0) circle [radius=0.08];
\node at (5,0) [right]{$\mathcal{S}_4$};
\node at (5.7,1.4) [right]{$\mathcal{S}_3$};
\node at (4.3,1.4) [left]{$\mathcal{S}_1$};
\end{tikzpicture}
    \caption{Representation of how the different $\mathcal{S}_i$s relate to each other.}
    \label{fig: repr how Sis intersect}
\end{figure}

\begin{remark}
    Note that the stratification induced by $\{\mathcal{S}_i\}$ is coarser than the one induced by the orbit type stratification theorem, as there could be different orbit types of the same dimension. However, we have seen in \cref{prop: dimensions S_i} that the tangent space of each $\mathcal{S}_i$ is spanned by the tangent space of the orbit and possibly the cross product of the $\T^2$ generators.
    Since the flow of this cross product preserves the orbit type (see \cref{lemma: flow U1xU2 preserves orbit type}), the orbit type is unchanged along every connected component of each $\mathcal{S}_i$ and, hence, we can reconstruct one stratification from the other.
\end{remark}
\subsection{Multi-moment maps}\label{sec: subsection multi-moment maps} 

In \cite{MadsenSwann2012} and \cite{MadsenSwann2013}, Madsen and Swann extended the classical notion of moment maps for symplectic manifolds to any closed geometry $(X,\alpha)$, i.e. a manifold $X$ endowed with a closed form $\alpha$. The idea is to take generators of a subgroup of $\Aut(X,\alpha)$ and contract them with $\alpha$ to reduce its degree to 1. Now, if these 1-forms are exact they can be integrated to functions in $C^{\infty}(X;\R)$ (defined up to additive constants) that they call multi-moment maps. In order to ensure closedness, Madsen and Swann introduced the notion of Lie kernel, which we omit for brevity.

In this work, the $G$ multi-moment maps will be crucial in studying cohomogeneity-one calibrated submanifolds of $(M,\vphi)$. Indeed, we will see in \cref{sec: section T^2-invariant associatives} and \cref{sec: T3-invariant/SU(2)-invariant coassociatives} that such submanifolds are contained in the level sets of some multi-moment maps and that a direction transveral to the orbits is parametrized by the gradient of a multi-moment map. Finally, multi-moment maps will also be used in \cref{sec: local characterization} to find natural hypersurfaces of $M$.

Assuming from now on that the $\G2$ manifold $(M,\vphi)$ is  simply connected (so that all closed $1$-forms are exact), we can then define the $G$ multi-moment maps related to $\vphi$ and $\ast\vphi$ bypassing the notion of Lie kernel and other difficulties. 
\begin{remark}
Observe that it makes sense to consider the multi-moment maps with respect to $\ast\vphi$ as well. Indeed, it is a closed form and, by \cref{eqn: G2 structure induces metric and volume}, a $\vphi$-preserving action will also preserve the metric $g_{\vphi}$ and the volume form $\vol_\vphi$. Therefore, $\ast\vphi$ will also be preserved. 
\end{remark}

First, we fix the notation for the generators of $G$. Let $U_1, U_2$ be the generators of $\LieT^2\oplus\{0\}\subset \LieT^2\oplus\LieSU(2)$ and let $V_1,V_2,V_3$ be the generators of $\{0\}\oplus\LieSU(2)\subset \LieT^2\oplus\LieSU(2)$. Clearly, we can choose them to satisfy:
\begin{align}\label{eqn: bracket relation between generators action}
	[U_l,U_m]=0,\hspace{15pt} [U_l,V_i]=0, \hspace{15pt} [V_i,V_j]=\epsilon_{ijk}V_k,
\end{align} 
for all $l,m=1,2$ and $i,j,k=1,2,3$.

\begin{definition}\label{def: phi-mmmap}
	The multi-moment maps with respect to $\vphi$ are the smooth functions (defined up to additive constants) $\theta^l:M \to\R^3$ for $l=1,2$ and $\nu:M\to \R$ characterized by:
	\begin{align}\label{eqn: formula dnu and dtheta}
	 d\theta^l_i:=\vphi(U_l,V_i,\cdot),
\hspace{15pt} d\nu:=\vphi(U_1,U_2,\cdot),
\end{align}
where $i=1,2,3$.
\end{definition}

\begin{definition}\label{def: astphi-mmmap}
	The multi-moment maps with respect to $\ast\vphi$ are the smooth functions (defined up to additive constants) $\mu:M \to\R^3$ and $\eta:M\to \R$ characterized by:
	\begin{align}\label{eqn: formula dmu and deta}
	d\mu_i:=\ast\vphi(U_1,U_2,V_i,\cdot), \hspace{15pt} d\eta:=\ast\vphi(V_1,V_2,V_3,\cdot),
\end{align}
where $i=1,2,3$.
\end{definition}

As a sanity check, one can show that the one-forms given on the right-hand-side are all closed. 

\begin{lemma}\label{lemma: formula mu_i} The multi-moment maps $\mu$ and $\theta$ have the form:
\begin{align}\label{eqn: formula mu_i}
\mu_k=-{\ast}\vphi(U_1,U_2,V_i,V_j), \hspace{15pt} \theta_k^l=-\vphi(U_l,V_i,V_j),
\end{align}
where $(i,j,k)$ is a cyclic permutation of $(1,2,3)$.
\end{lemma}
\begin{proof}
	The proof is a straightforward application of Cartan's formula, the identity $[\mL_X, i_Y]=i_{[X,Y]}$ for every vector field $X,Y$ and \cref{eqn: bracket relation between generators action}.
\end{proof}

Before considering the properties of the multi-moment maps, we state two trivial results that we will use throughout the paper.  

\begin{lemma}
\label{lemma: equivariance through differential SU2}
Let $M$ be a smooth manifold with an $\SU(2)$ action with generators  $V_1,V_2,V_3$ satisfying $[V_i,V_j]=\epsilon_{ijk}V_k$.
Then a smooth function $f \colon M \to \R^3$ is equivariant with respect to the action of $\SU(2)$ on $\R^3$ via the double cover $\SU(2)\to \SO(3)$ if and only if $f$ satisfies:
\begin{align*} 
\mL_{V_i} f_j = \epsilon_{ijk} f_k.
\end{align*}
\end{lemma}

\begin{lemma}
\label{lemma: invariance through differential G}
Let $M$ be a smooth manifold with the action of a connected Lie group $G$ with generators $U_1,...,U_l$.
Then a smooth function $f \colon M \to \R$ is invariant under the $G$-action if and only if $f$ satisfies:
\begin{align*} 
\mL_{U_i} f = 0,
\end{align*}
for every $i=1,...,l$.
\end{lemma}

\begin{proposition}\label{prop: T^2 invariance of mu and nu}
Let $\theta$, $\nu$, $\mu$ and $\eta$ be as in \cref{def: phi-mmmap} and \cref{def: astphi-mmmap}. If $\SU(2)$ acts on $\R^3$ via the double cover $\SU(2)\to \SO(3)$, then:
\begin{enumerate}
	\item $\nu$ is $\T^2\times\SU(2)$-invariant,
	\item $\mu$ is $\T^2$-invariant and $\SU(2)$-equivariant,
	\item $\av{\mu}$ is $\T^2\times\SU(2)$-equivariant,
	\item $\theta^1$ and $\theta^2$ are $\T^2$-invariant and $\SU(2)$-equivariant,
	\item $\av{\theta^1}$ and $\av{\theta^2}$ are $\T^2\times\SU(2)$-equivariant,
	\item $\eta$ is $\SU(2)$-invariant and, if the $\SU(2)/\Gamma_2$-action has a singular orbit, is $\T^2\times\SU(2)$-invariant.
\end{enumerate} 
	Moreover, each $\T^2$-invariant function on $M$ descends to a function on the topological space $M/\T^2$; each $\SU(2)$-invariant function on $M$ descends to a function on the topological space $M/\SU(2)$, and every $\T^2\times\SU(2)$-invariant function descends to a function on $M/\T^2\times\SU(2)$.
\end{proposition}
\begin{proof}
	The $\T^2$-invariance of $\nu,\mu$ is clear from \cref{lemma: invariance through differential G}, \cref{eqn: formula dnu and dtheta} and \cref{eqn: formula dmu and deta}, while the $\SU(2)$-equivariance of $\mu$ and $\theta^l$ follows from \cref{lemma: equivariance through differential SU2} and:
	\begin{align*}
		\mL_{V_i} \mu_j=\epsilon_{ijk}\mu_k, \hspace{15pt} \mL_{V_i} \theta^l_j=\epsilon_{ijk}\theta^l_k.
	\end{align*}
	If we show that $\vphi(U_1,U_2,V_i)=0$ for every $i=1,2,3$, then $\nu$ is $\SU(2)$-invariant and $\theta^l$ is $\T^2$-invariant. Cartan's formula, together with $[\mL_X, i_Y]=i_{[X,Y]}$, implies that $d(\vphi(U_1,U_2,V_i))=0$ and, hence, $\vphi(U_1,U_2,V_i)$ is a constant $c_i$. We conclude because:
	\begin{align}
	\label{eqn: U1xU2 in N}
	0=\mL_{V_j} c_i=V_j(\vphi(U_1,U_2,V_i))=-\vphi(U_1,U_2,V_k)=-c_k,
	\end{align}
	where we used again Cartan's formula and \cref{eqn: bracket relation between generators action}. Analogously, one can prove that $\eta$ is $\T^2$-invariant if the $\SU(2)/\Gamma_2$-action has a singular orbit. We conclude as $\eta$ is obviously $\SU(2)$-invariant.
	\end{proof}
Since the $\T^2\times\SU(2)$-action is structure preserving, and in particular, its generators are Killing vector fields, we can obtain the following result. Recall that the Lie derivative of a Killing vector field commutes with musical isomorphisms.
\begin{corollary}\label{cor: invariance gradients mmaps}
   Let $\nu,\mu,\eta$ be as defined in \cref{def: phi-mmmap} and \cref{def: astphi-mmmap}. Then:
   \begin{enumerate}
   	\item $\nabla\nu=U_1\times U_2$ is $\T^2\times \SU(2)$-invariant,
   	\item $\nabla\av{\mu}$ is $\T^2\times\SU(2)$-invariant,
   	\item $\nabla\eta$ is $\SU(2)$-invariant and, if the $\SU(2)/\Gamma_2$-action has a singular orbit, is $\T^2\times\SU(2)$-invariant.
   \end{enumerate}
    Moreover, each $H$-invariant vector field on $M$ descends to a vector field on the principal part of the $H$-action, for every $H\leqslant\T^2\times\SU(2)$.
\end{corollary}
\begin{remark}
    As an abuse of notation, we will use the same symbol for both the invariant functions (or vector fields) in the total space and in the quotients.
\end{remark}
We are also able to locate the zero set of the multi-moment map of $\mu$ in terms of the stratification given in \cref{thm: class stab groups}.
\begin{corollary}\label{cor: zero set of mu in barS}
	 Let $\mu$ as in \cref{eqn: formula mu_i}. Then: $$\mathcal{S}_2\cup \mathcal{S}_3\cup \mathcal{S}_4 \subset \mu^{-1}(0)\subset \mathcal{S}_1\cup \mathcal{S}_2\cup \mathcal{S}_3\cup \mathcal{S}_4.$$ 
\end{corollary}
\begin{proof}
The first inclusion is obvious from \cref{thm: class stab groups} and \cref{lemma: formula mu_i}. 

Assume by contradiction that the second inclusion does not hold. Hence, there exists a point $x\in M$ such that $\mu(x)=0$ and such that $U_1,U_2,V_1,V_2,V_3$ are linearly independent at $T_x M$. By \cref{eqn: U1xU2 in N}, $V_1,V_2,V_3$ spans a $3$-dimensional linear subspace of $T_x M$ which is orthogonal to $U_1\times U_2$ and transversal to the two-dimensional subspace spanned by $U_1,U_2$. Since the two-form $\ast\varphi(U_1,U_2,\cdot,\cdot)$ does not vanish on any such $3$-dimensional subspace, we can conclude. 
\end{proof}

\section{Local characterization of \texorpdfstring{$\G2$}{Lg} manifolds with \texorpdfstring{$\T^2\times \SU(2)$}{T2xSU2}-symmetry}\label{sec: local characterization} 

In this section, we provide a local characterization of $\G2$ manifolds with a structure-preserving $\T^2\times\SU(2)$-action. This characterization is local in the sense that we restrict our manifold $M$ to $M_P$, where $G=(\T^2\times\SU(2))/\Gamma$ acts freely. 

In the first subsection, we recall Madsen--Swann $\T^2$-reduction \cite{MadsenSwann2012}, which can be summarized as follows.  Any smooth hypersurface in a torsion-free $\G2$ manifold carries a half-flat $\SU(3)$-structure \cite{ChiossiSalamon2002}. Moreover, under the real-analytic condition, one can locally reverse this procedure through Hitchin's flow \cite{Hitchin2001}. As the manifold admits a free $\T^2$-action, it is natural to take a level set of $\nu$ (which can be defined as in \cref{def: phi-mmmap} even when the manifold is endowed with only a $\T^2$-action) as the given hypersurface. Madsen--Swann \cite{MadsenSwann2012} proved that the $\SU(3)$-structure on the level sets of $\nu$ is described as a $\T^2$-bundle over a four manifold $\chi$, with a coherent tri-symplectic structure.

Afterwards, we enhance the symmetry to $\T^2 \times \SU(2)$, which implies that the coherent tri-symplectic manifold $\chi$ admits a structure-preserving $G^{\SU(2)}$-action, and the curvature of the $\T^2$-bundle is also $G^{\SU(2)}$-invariant. In the second subsection, we describe the $G^{\SU(2)}$-invariant coherent tri-symplectic structure in a frame compatible with the action. In the third subsection, we characterize all such structures in terms of a solution of an ODE. Finally, in the last subsection, we explain how to deal with the $\T^2$-bundle structure and how to locally characterize $\G2$ manifolds with $\T^2\times\SU(2)$-symmetry.

\subsection{The \texorpdfstring{$\T^2$}{T2}-reduction}
\label{subsec: T2-reduction}
Let $(M,\vphi)$ be a $\G2$ manifold with a structure-preserving $\T^2$-action and singular set $\mathcal{S}$. On $M\setminus \mathcal{S}$, the level sets of $\nu$ are hypersurfaces oriented by $\nabla\nu=U_1\times U_2$, where $U_1,U_2$ are two generators of the $\T^2$-action. The $\T^2$-action passes to the level sets of $\nu$ and, hence, it endows $\nu^{-1}(t)$ with a $\T^2$-bundle structure over $\nu^{-1}(t)/\T^2$, which inherits the following additional structure (cfr. \cite{MadsenSwann2012}). 
\begin{definition}\label{def: coherent tri-symplectic structure}
	A $4$-manifold $\chi$ has a coherent tri-symplectic structure if it admits three symplectic forms $\overline{\sigma}_0, \overline{\sigma}_1, \overline{\sigma}_2$ such that $\overline{\sigma}_0\wedge \overline{\sigma}_i=0$ for $i=1,2$, $\overline{\sigma}_0\wedge \overline{\sigma}_0$ is a volume form of $\chi$ and the matrix $Q:=(Q_{ij})_{i,j=1,2}$ defined by $\overline{\sigma}_i\wedge\overline{\sigma}_j=Q_{ij}\overline{\sigma}_0\wedge\overline{\sigma}_0$ is positive definite. 
\end{definition}
The forms defining this structure on $\nu^{-1}(t)/\T^2$ are: 
\begin{align}\label{eqn: tri-symplictic forms}
    \overline{\sigma}_0 = \ast \varphi(U_1,U_2,\cdot,\cdot), \quad \overline{\sigma}_1 = \varphi(U_1,\cdot,\cdot), \quad \overline{\sigma}_2 = \varphi(U_2,\cdot,\cdot).
\end{align} 

Conversely (see \cite{MadsenSwann2012}*{Theorem 6.10}), assuming real analyticity, one can locally reconstruct a $\G2$ manifold with $\T^2$-symmetry from a coherent tri-symplectic four manifold $\chi$, equipped with a closed two-form $F \in \Omega^2(\chi,\R^2)$ with integral periods and whose self-dual part $F_+$ satisfies the orthogonality condition:
\begin{align}
\label{eqn: orthogonality curvature}
 F_+ = (\bar{\sigma}_1,\bar{\sigma}_2)A,
\end{align}
for some $A\in \GL(2,\R)$ such that $\mathrm{Tr}(AQ)=0$.
These conditions guarantee that $F_+$ is the curvature form of a $\T^2$-bundle $N$ over $\chi$. The $\G2$-structure is then constructed from $N$ by running rescaled Hitchin's flow. The resulting $\G2$-structure yields a moment map $\nu$ of which $N$ is a level set and rescaled Hitchin's flow evolves $N$ into other level sets of $\nu$. \par
When the symmetry is enhanced to $\T^2\times\SU(2)$, the remaining $G^{\SU(2)}$-symmetry passes to the quotient $\chi$ and preserves its coherent tri-symplectic structure (see \cref{eqn: tri-symplictic forms}). We now describe such four manifolds with a free $G^{\SU(2)}$-symmetry.

\subsection{On \texorpdfstring{$4$}{4}-manifolds with coherent symplectic triple and \texorpdfstring{$G^{\SU(2)}$}{GSU2}-symmetry}
\label{subsec: coherent symplectic}
Let $(\chi,\bar{\sigma}_1,\bar{\sigma}_2,\bar{\sigma}_3)$ be a coherent symplectic $4$-manifold with a $G^{\SU(2)}$ structure-preserving free action generated by the vector fields $V_1,V_2,V_3$ satisfying $[V_i,V_j] = \epsilon_{ijk} V_k$.  Since the action is structure-preserving, we have that $\mathcal{L}_{V_i}\bar{\sigma}_j =0$, therefore, $Q$ is $G^{\SU(2)}$-invariant. Moreover, as $Q$ is also positive definite, there exists a unique real symmetric, positive definite $2\times2$ matrix $T$ such that $T^{-2} = T^{-1} (T^{-1})^T =Q$, which is $G^{\SU(2)}$-invariant as well. 

Let $\vol_{\chi} := \frac{1}{2} \bar{\sigma}_0 \wedge \bar{\sigma}_0 $ and define the forms $\sigma_i := \sum_{j=1}^2 T_{ij} \bar{\sigma}_j$ for $i = 1,2$, which then satisfy $\sigma_i\wedge \sigma_j =2 \delta_{ij} \vol_{\chi}$. Define the metric:

\[g_{\chi}(u,v) \vol_{\chi} = \sigma_0 \wedge i_u \sigma_1 \wedge i_v \sigma_2,\]
for all $u,v\in T_x \chi$ and all $x\in \chi$.
 With respect to this metric, the vector fields $V_i$ are Killing for $g_{\chi}$.
 
 Using the standard cover $\SL(4,\R)\to \SO(3,3)$ induced by the map: 
 \[ \Lambda^2 \otimes \Lambda^2 \to \Lambda^4 \cong \R, \quad \alpha \otimes \beta \mapsto \alpha \wedge \beta,\]
 one can prove the following lemma.
\begin{lemma}\label{lemma: SD two-forms in alpha}
There are unique $g_\chi$-orthonormal one-forms $\alpha_i$ for $i=0,...,3$ such that 
\begin{align}
\label{eq: hk forms in alpha frame}
\begin{split}
 {\sigma}_0 &= \alpha_0 \wedge \alpha_1+\alpha_2 \wedge \alpha_3, \quad
 {\sigma}_1 = \alpha_0 \wedge \alpha_2+\alpha_3 \wedge \alpha_1, \\ 
 {\sigma}_2 &= \alpha_0 \wedge \alpha_3+\alpha_1 \wedge \alpha_2, \quad
\alpha_0 =\frac{1}{\sqrt{\det{{\hat{g}}_{\chi}}}} \mathrm{vol}_{\chi} (V_1,V_2,V_3,\cdot),
 \end{split}
\end{align}
where $\hat{g}_{\chi}$ is the matrix-valued function of entries $(g_{\chi}(V_i,V_j))_{i,j=1,2,3}$.
\end{lemma}

We define the unit vector field $X:=\alpha_0^{\musSharp}$, which satisfies the conditions $\alpha_0(X)=1$ and $\alpha_i(X)=0$ for $i=1,2,3$, and determines the $\alpha_i$s by $\alpha_i = {\sigma}_{i{-1}}(X,\cdot)$.
Consider the two $3\times 3$-matrix-valued functions $\eta=(\eta_{ij})$ and $\tau=(\tau_{ij})$, where $\eta_{ij}$ and $\tau_{ij}$ are defined by:
\[\eta_{ij} := {\sigma}_{i-1}(X,V_j)=\alpha_i(V_j), \hspace{10pt}\tau_{ij}:={\sigma}_{i-1}(V_k,V_l), \] 
for $(j,k,l)$ positive permutation of $(1,2,3)$. We also define the one-forms $\delta_0$ and $\delta_i$ for $i=1,2,3$ by: 
\[\delta_0={\sqrt{\det{\hat{g}_{\chi}}}}\alpha_0 = \mathrm{vol}_\chi (V_1,V_2,V_3,\cdot), \quad \delta_i(V_j) = \delta_{ij}, \quad \delta_i(X)=0.\]
which satisfies $\alpha_i = \sum_{j=1}^3 \eta_{ij} \delta_j.$

Using that $[V_i,X]=0$, standard computations yield the following.
\begin{lemma}\label{lemma: technical local characterization manifolds}
The matrix functions $\eta$ and $\tau$ have the following properties
\begin{itemize}
    \item $\tau =\mathrm{adj}(\eta^T)$, where $\mathrm{adj}$ denotes the adjugate matrix.
    \item The row vectors of $\tau$ and $\eta$ are $G^{\SU(2)}$-equivariant.
    \item The determinant of $\tau$ and the determinat of $\eta$ are $G^{\SU(2)}$-invariant,
    \item The $3\times3$-matrix-valued function $\hat{g}_\chi$ with entries $(g_{\chi}(V_i,V_j))_{i,j=1,2,3}$ is determined by $\eta$ via: 
\begin{align} \label{eq: relation eta g}
    \hat{g}_\chi = \eta^T \eta, \end{align}
    \item We have the matrix equation:
\begin{align}
\label{eqn: sigma in delta-frame}
 {\sigma} = \frac{1}{\det(\eta)} \delta_0 \wedge \eta \delta+ \tau \bar{\delta},
\end{align}
	where ${\sigma}=({\sigma}_1,{\sigma}_2,{\sigma}_3)^T$, $\delta = (\delta_1,\delta_2,\delta_3)^T$ and $\bar{\delta}=(\delta_2 \wedge \delta_3,\delta_3 \wedge \delta_1, \delta_1 \wedge \delta_2)^T$.
\end{itemize}
 
\end{lemma}

In this subsection, we have constructed a $G^{\SU(2)}$-compatible co-frame $\{\delta_i\}_{i=0}^3$ on $\chi$, and we have rewritten the orthogonalized coherent symplectic structure $\{{\sigma}_i\}_{i=1}^3$ in this co-frame (\cref{eqn: sigma in delta-frame}). Along the way, we have introduced on $\chi$ a compatible volume form, $\vol_\chi$, and a metric, $g_\chi$, which induces two $3\times3$-matrix-valued functions $\eta$ and $\tau$ representing $g_\chi$ on this $G^{\SU(2)}$-compatible co-frame.

\subsection{The differential equation} Now, we deduce how the equations $d\bar\sigma_i=0$ transform in the $G^{\SU(2)}$-compatible co-frame $\{\delta_i\}_{i=0}^3$ that we constructed in the previous section. 

We assume that $H^1(\chi, \R) = 0$ so that there is a function $R$ such that $d R = \delta_0$. The dual vector field $\partial_R$ is equal to ${(\det{\eta})^{-1}}X$, so it satisfies $[\pd{R},V_i]=0$, for every $i=1,2,3$. Morever, by \cref{lemma: technical local characterization manifolds} and the commutator relationships for $X$ and $V_i$, we deduce that $d \delta = -\bar{\delta}$ and $d (\frac{1}{\det \eta} \delta_0)=0$.

We recall the following version of \cref{lemma: equivariance through differential SU2} in terms of differential forms, which can be proven using Cartan's formula.
\begin{lemma}
\label{lemma: equivariance through differential v2}
A smooth function $f\colon \chi\to \R^3$ is $\SU(2)$-equivariant if and only if $(d f = f \times \delta) \mod \delta_0,$
for $(f\times\delta)_i=\epsilon_{ijk}f_j \delta_k$.
\end{lemma}

As a consequence of this lemma, we have
\[d\eta = \eta \times \delta+\frac{\partial \eta}{\partial R} \delta_0, \quad d\tau {=\tau \times \delta +\frac{\partial \tau}{\partial R}}\delta_0, \]
where $(\eta\times \delta)_{ij}=(\eta_i \times \delta)_j$ and $(\tau\times \delta)_{ij}=(\tau_i \times \delta)_j$, i.e. we are taking the cross products of the rows of $\eta$ with $\delta$.
Putting all together in \cref{eqn: sigma in delta-frame}, we get 
\begin{align*}
d {\sigma}&= \frac{1}{\det \eta}  \delta_0 \wedge (-d \eta \wedge \delta-\eta d \delta)+ d \tau \wedge \bar{\delta}
=\frac{1}{\det\eta} \delta_0 \wedge (-\eta \bar{\delta})+(\partial_R \tau) \delta_0 \wedge \bar{\delta}.
\end{align*}
The last step is due to the two identities:
\begin{align*}
(\eta \times \delta)\wedge \delta = 2 \eta \bar{\delta}, \quad (\tau \times \delta)\wedge \bar{\delta}=0.
\end{align*}
Extend $T$ to a $3\times3$ matrix by padding it with one in the $(1,1)$ entry and by zeros in the first row and column elsewhere. This extension is such that $\sigma=T\bar\sigma$, which implies:  
\begin{align}
\label{eqn: structure eqn sigma}
    d \sigma = d T \wedge \bar{\sigma} = \partial_R (T) T^{-1} \delta_0 \wedge \sigma =\partial_R (T) T^{-1} \tau \delta_0\wedge \bar{\delta}, 
\end{align}
where the first equality follows from $d\bar{\sigma}_i=0$, the second one from the $G^{\SU(2)}$-invariance of $T$ and the definition of $\sigma$, and the third one from \cref{eqn: sigma in delta-frame}.
Combining the two equations for $d \sigma$ and using $\frac{1}{\det \eta}\eta = (\tau^T)^{-1}$ gives:
\begin{align}
\label{eqn: ODE tau with 3-forms}
0=(\partial_R \tau -(\partial_R T) T^{-1} \tau -(\tau^T)^{-1})\delta_0 \wedge \bar{\delta}.
\end{align}
\begin{proposition}\label{thm: Local characterization}
 A coherent symplectic $4$-manifold $\chi$ with free $G^{\SU(2)}$-symmetry and intersection matrix $Q$ admits a matrix-valued function $\tau \colon \chi \to M_{3\times3}(\R)$ whose rows are equivariant with respect to the action of $\SO(3)$ on $\R^3$ and satisfying the following differential equation:
 \begin{align}
 \label{eqn: ODE tau}
 \partial_R \tau= (\partial_R T) T^{-1} \tau +(\tau^T)^{-1},\end{align}
 where $T:\chi \to M_{3\times3}(\R)$ is the, padded as above, matrix satisfying $Q=T^{-2}$.
 
 Conversely, let $T:(a,b)\to \Sym_{2\times2}(\R)$ be a function of positive-definite matrices, identified with $T:(a,b)\to \Sym_{3\times3}(\R)$ padded as above. Then equivariant solutions $\tau:(a,b)\times G^{\SU(2)}\to  M_{3\times3}(\R)$ of \cref{eqn: ODE tau} are in bijection with coherent symplectic structures on $(a,b)\times G^{\SU(2)}$ with intersection matrix $Q=T^{-2}$.
 \end{proposition}
\begin{proof}
The first statement follows from \cref{eqn: ODE tau with 3-forms} since the $\delta_0 \wedge \bar{\delta}_i$ are linearly independent on $\chi$. 
 
For the converse direction, define the frame $\delta_0,\dots,\delta_3$ on $(a,b)\times \SU(2)$ such that $\delta_0= d R$ and $\delta_i$ are the invariant one-forms on $\SU(2)$, hence, satisfying $d \delta_i =- \epsilon_{ijk}\delta_j \wedge \delta_k$.
\cref{lemma: equivariance through differential v2} and \cref{eqn: ODE tau} imply 
\begin{align}
\label{eqn: d tau}
d \tau =  \tau \times \delta+\left((\partial_R T) T^{-1} \tau +(\tau^T)^{-1})\right) \delta_0
\end{align}
Define the forms $\alpha_i$ by the equation $\alpha_i = \sum_{j=1}^3 \eta_{ij} \delta_j$, with $\eta:= \mathrm{adj}(\tau^T)$) as before. From the $\alpha_i$s, we can reconstruct the forms $\sigma$ by \cref{eq: hk forms in alpha frame} and then $\bar{\sigma}$ through the transformation matrix $T$. We deduce that $\bar{\sigma}_i$ are such that $\bar{\sigma}_0 \wedge \bar{\sigma}_i=0$ and $\bar{\sigma}_i\wedge \bar{\sigma_j}=Q_{ij}\frac{1}{2}\sigma_0\wedge\sigma_0$, where $Q=T^{-2}$.
Our previous computations show that \cref{eqn: d tau} implies that the forms $\bar{\sigma}_i$ are closed and, hence, we conclude.
\end{proof}
\begin{remark}
If $Q$ is the identity matrix, then $g_\chi$ is hyperk\"ahler and by rotating $\sigma_0,\sigma_1,\sigma_2$ we can assume that $\tau$ is a diagonal at a given point. The diagonality is preserved along $R$ (as in the Bianchi IX ansatz) by \cref{eqn: ODE tau}, and we have $\partial_R \frac{1}{2}\tau_{ii}^2=1$ for $i=1,2,3$. So each $\tau_{ii}$ is of the form $\sqrt{2 R+k_i}$ and can we assume that $k_1+k_2+k_3=0$ and $k_1\geq k_2 \geq k_3$.
The metric $g_\chi$ is 
\[\frac{1}{\tau_{11}\tau_{22}\tau_{33}}d R^2+\frac{\tau_{22}\tau_{33}}{\tau_{11}}\delta_1^2+\frac{\tau_{33}\tau_{11}}{\tau_{22}}\delta_2^2+\frac{\tau_{11}\tau_{22}}{\tau_{33}}\delta_3^2\]
If all $k_i=0$, then all $\tau_{ii}$ are equal and the metric is flat. If $k_1>0$ and $k_2=k_3<0$ then $g_{\chi}$ is the Eguchi-Hanson metric. In all other cases the metric is incomplete. Note that the Taub-NUT and Atiyah-Hitchin metric are not described by our set-up, since the $\SU(2)$ action is not tri-holomorphic on these spaces. Instead, the action rotates the three hyperk\"ahler two-forms.
\end{remark}
\subsection{From coherent tri-symplectic manifolds to \texorpdfstring{$\G2$}{G2} manifolds}
Finally, we use \cref{thm: Local characterization} to obtain a local construction of $\G2$ manifolds with $\T^2\times\SU(2)$-symmetry through \cite{MadsenSwann2012}*{Theorem 6.10}.

The last object that we need is an orthogonal (i.e., satisfies \cref{eqn: orthogonality curvature}) self-dual two-form $F_+\in \Omega^2(\chi,\R^2)$ on $\chi$ with integral periods. This condition assumes the existence of an anti-self-dual form $F_-\in \Omega^2(\chi,\R^2)$ such that $F_++ F_-$ is closed and defines an element in the image of $H^2(M,\Z^2)$. 

In the $G^{\SU(2)}$-invariant case the closedness condition can always be satisfied.

\begin{lemma}
\label{lemma: constr F_-}
For any $G^{\SU(2)}$-invariant $F_+ \in \Omega_+^2(\chi,\R^2)$, there is a $F_- \in \Omega_-^2(\chi,\R^2)$ such that $F_++F_-$ is closed.
\end{lemma}
\begin{proof}
Using the form that the self-dual two-forms $\{\sigma_i\}_{i=1}^3$ take in \cref{lemma: SD two-forms in alpha}, we can define the anti-self dual two-forms:
\begin{align*}
 {\sigma}^-_1 &= -\alpha_0 \wedge \alpha_1+\alpha_2 \wedge \alpha_3, \\
 {\sigma}^-_2 &= -\alpha_0 \wedge \alpha_2+\alpha_3 \wedge \alpha_1, \\
 {\sigma}^-_3 &= -\alpha_0 \wedge \alpha_3+\alpha_1 \wedge \alpha_2. 
 \end{align*}
 The vector of $2$-forms $\sigma^-:=({\sigma}^-_1,{\sigma}^-_2,{\sigma}^-_3)$ satisfies the same structure equation of $\sigma$: \cref{eqn: structure eqn sigma}. Indeed, this is evident by computing $d \sigma^-$ as before or by using a local diffeomorphism that preserves $\alpha_1,\alpha_2,\alpha_3$ and flips the sign of $\alpha_0$, i.e pulls back $\sigma$ to $\sigma^-$. It follows that their difference satisfies:
 \[d(\sigma - \sigma^-) = \partial_R (T) T^{-1} \delta_0 \wedge (\sigma - \sigma^-), \]
 which vanishes as $\sigma - \sigma^-= 2 \alpha_0 \wedge \alpha$ and $\alpha_0$ is proportional to $\delta_0$.

Since $F_+$ is self-dual, there is $a\colon \chi\to \R^3\otimes\R^2$ such that $F_+ = a\sigma = \sum_i a_i \sigma_i$. Because $F_+$ is $G^{\SU(2)}$-invariant, the same is true for $a$, which implies that $d a$ is a multiple of $\alpha_0$. Now define $F_-:=- a \sigma^-$ and observe
\[d (F_++ F_-) =\sum_{i=1}^3 2 d a_i \wedge \alpha_0 \wedge \alpha_i =0,\]
as required.
\end{proof}
\begin{remark}
	In a similar fashion, one can find all closed $G^{\SU(2)}$-invariant 2-forms $F^++F^-$ in terms of a system of ODEs.
\end{remark}

If the function $T$ is real-analytic the solutions of \cref{eqn: ODE tau} are real-analytic as well by the Cauchy-Kovalevskaya theorem. This observation, together with \cref{thm: Local characterization} and \cite{MadsenSwann2012}*{Theorem 6.10} implies the following theorem.

\begin{theorem}\label{thm: gibbons-hawking}
Simply connected $\G2$ manifolds with a free $G$-action are in bijection with solutions of \cref{eqn: ODE tau}, for any given $T:(a,b)\to \Sym_{2\times2}(\R)$ real-analytic function of positive-definite matrices, together with the real analytic two-form $F_+\in \Omega^2_+((a,b)\times G^{\SU(2)},\R^2)$ satisfying \cref{eqn: orthogonality curvature} and such that $F_+ +F_-$ is closed and with integral periods, for some real analytic anti-self-dual form $F_-$ in $\Omega^2(\chi;\R^2)$.
\end{theorem}

\section{\texorpdfstring{$\T^2$}{T2}-invariant associative submanifolds}\label{sec: section T^2-invariant associatives}
In this section, we study $\T^2\cong\T^2\times\Id_{\SU(2)}$-invariant associative submanifolds of a $\G2$ manifold $(M,\vphi)$, endowed with a structure-preserving, cohomogeneity two action of $\T^2\times \SU(2)$. We use the same notation and conventions of \cref{sec: T^2xSU(2) symmetry section}.

First, we give a characterization of $\T^2$-invariant associatives in terms of integral curves of a vector field in the $\T^2$-quotient. Since such a characterizing vector field is $\T^2\times\SU(2)$-invariant, the problem of finding associative submanifolds "splits" with the stratification constructed in \cref{thm: class stab groups}. Moreover, the multi-moment map $\mu:M\to\R^3$, defined in \cref{def: astphi-mmmap}, is a first integral of the ODE problem, i.e., it is constant on every $\T^2$-invariant associative.

In the principal part $M_P$ of the $\T^2\times\SU(2)$-action, we characterize $\T^2$-invariant associatives using the level sets of $\av{\mu}:M_P/G\to\R$. Indeed, $M_P/\T^2$ admits a $G^{\SU(2)}$-bundle structure, and $\T^2$-invariant associatives project to the level sets of $\av{\mu}$. Choosing a suitable connection on the $G^{\SU(2)}$-bundle, one can horizontal lift these level sets and reverse the procedure. We conclude our discussion on $M_P$ by making this characterization locally explicit. 

In the singular part of the $\T^2\times\SU(2)$-action, we use \cref{thm: class stab groups} to show that there exists a submersion from $\mathcal{S}_1$ to $S^2$ such that each fibre is a $\T^2$-invariant associative. Similarly, we show that $\mathcal{S}_2$ and $\mathcal{S}_3\cup \mathcal{S}_4$ are associatives. 

Putting together our discussion on the principal part and on the singular part of the $\T^2\times\SU(2)$-action, we deduce that there exists an easy geometrical condition that guarantees the existence of a $\T^2$-invariant associative fibration. Finally, we show that all $\T^2$-invariant associatives are smooth. 

In \cref{sec: section examples}, we will use the theory developed here to describe $\T^2$-invariant associatives in the FHN $\G2$ manifolds.

\subsection{\texorpdfstring{$\T^2$}{T2}-invariant associatives} As in \cref{sec: subsection multi-moment maps}, let $U_1$ and $U_2$ be the generators of $\LieT^2\oplus\{0\}\subset \LieT^2\oplus\LieSU(2)$. We now give a characterization of $\T^2$-invariant associatives as integral curves of $U_1\times U_2$. 

\begin{proposition}\label{prop: T^2-invariantAssociativesNoSU(2)}
	Let $L_0$ be a $\T^2$-invariant associative submanifold of $M\setminus \mathcal{S}\supseteq M_P$. Then $L_0/{\T^2}$ is an integral curve of the nowhere vanishing vector field $U_1\times U_2$ in $(M\setminus \mathcal{S})/{\T^2}$. Conversely, every integral curve of $U_1\times U_2$ in $(M\setminus \mathcal{S})/{\T^2}$ is the projection of a $\T^2$-invariant associative in $M\setminus \mathcal{S}$.
\end{proposition}
\begin{proof}

 Via the projection map, every $\T^2$-invariant submanifold $L_0$ of $M\setminus \mathcal{S}$ projects to a curve in $(M\setminus \mathcal{S})/{\T^2}$, and, conversely, every curve in $(M\setminus \mathcal{S})/{\T^2}$ can be lifted to a $\T^2$-invariant submanifold of $M\setminus \mathcal{S}$ by taking its preimage. This correspondence obviously extends to their tangent space.

If $L_0$ is also associative, it follows from \cref{prop: characterization associative planes} that its tangent space is spanned by $\{U_1, U_2, U_1\times U_2\}$. Since $U_1\times U_2$ is $\T^2$-invariant (\cref{cor: invariance gradients mmaps}) and orthogonal to $U_1,U_2$, we deduce that $L_0$ projects in $(M\setminus \mathcal{S})/{\T^2}$ to a curve with tangent space spanned by the nowhere vanishing vector field $U_1\times U_2$. Conversely, an integral curve of $U_1\times U_2$ in $(M\setminus \mathcal{S})/{\T^2}$ lifts to a $\T^2$-invariant submanifold of tangent space spanned by $\{U_1,U_2, U_1\times U_2\}$.
\end{proof}
We now state some general properties of $\T^2$-invariant associatives and integral curves of $U_1\times U_2$ that will play a crucial role later on.  

Since the flow of $U_1\times U_2$ commutes with the group action of $G$, we have the following. 
\begin{lemma}\label{lemma: flow U1xU2 preserves orbit type}
	The flow along $U_1\times U_2$ preserves the orbit type of $G$. Therefore, integral curves of $U_1\times U_2$ stay in the same stratum of the orbit type stratification, and hence of $\{\mathcal{S}_i\}$.
	\end{lemma}

In particular, we have proven that the problem of finding $\T^2$-invariant associatives decomposes with respect to the stratification, and, on $M\setminus \mathcal{S}$ it reduces to a problem of finding integral curves of a nowhere vanishing vector field.

 \begin{lemma}\label{lemma: mu is preserved by U_1xU_2}
 	The multi-moment map $\mu:M\to \R^3$ is preserved by the vector field $U_1\times U_2$. Therefore, $\mu$ is constant on every $\T^2$-invariant associative.
 \end{lemma} 
 \begin{proof}
 By definition of $\mu_i$ we have
 $
 d\mu_i(U_1\times U_2)=\ast\vphi(U_1,U_2,V_i,U_1\times U_2)
 $
 for every $i=1,2,3$. If $U_1, U_2$ are linearly independent, then $\{U_1,U_2, U_1\times U_2\}$ spans an associative plane and $\ast\vphi(U_1,U_2,V_i,U_1\times U_2)=0$ by \cref{prop: characterization associative planes}. Otherwise, the equation trivially holds.
 \end{proof}

 \subsection{Associatives in the principal set}\label{sec: subsection Associative in the principal set} In this subsection, we restrict our attention to the principal set $M_P$. Let $U_1,U_2,V_1,V_2,V_3$ be the generators of the $G$-action as in \cref{sec: subsection multi-moment maps}. Note that the action is assumed to be of cohomogeneity two, hence, the generators are everywhere linearly independent on $M_P$.
 
\begin{proposition}\label{prop: mu is a submersion}
	Let $\nu$ and $\mu$ be the multi-moment maps defined in \cref{def: phi-mmmap} and \cref{def: astphi-mmmap}, respectively, and restricted to $M_P$. Then the map $(\mu,\nu):M_P\to\R^3\times\R$ is a submersion. In particular, $\mu^{-1}(c)\cap M_P$ is a $4$-dimensional submanifold of $M_P$ for every $c$ in the image $\mu(M_P)$ and $(\av{\mu},\nu):M_P/G\to\R^2$ is a local diffeomorphism onto its image.
\end{proposition}

\begin{proof}
Given a fixed $x \in M_P$, it follows from \cref{cor: zero set of mu in barS} that $\mu(x)\neq 0$. Since $\mu$ is $\SU(2)$-equivariant and $\nu$ is $\SU(2)$-invariant, it suffices to show that $(\av{\mu}^2,\nu):M_P\to\R^2$ is a submersion at $x$.

As $\sum_{k=1}^3 \varphi(U_1,U_2,\mu_k V_k)=0$, there is an  $X\in T_x M$ such that $\sum_{k=1}^3\ast \varphi(U_1,U_2,\mu_k V_k,X)=1$. Observe that 
\[
\frac{1}{2} d \av{\mu}^2=\sum_{k=1}^3\mu_k \ast \vphi(U_1,U_2, V_k,\cdot), \]
which implies 
$d \av{\mu}^2(X) = 2$ and $d \av{\mu}^2(U_1 \times U_2)=0$.

Since $d(\av{\mu}^2,\nu)=(d\av{\mu^2},d\nu)$, we have proven that $d(\av{\mu}^2,\nu)(X)=(2,0)$. Obviously we also have that $d(\av{\mu}^2,\nu)(U_1\times U_2)=(0,\av{U_1\times U_2})$ and the statement follows.
\end{proof}
We now take a different perspective. Indeed, we argued in \cref{lemma: GSU(2) action on M/T^2 and GT2 action} that the action of $\SU(2)$ on $M$ induces on the quotient $M_P/ {\T^2}$ a principal bundle structure with structure group $G^{\SU(2)}$ and base space the surface $B=M_P/G$. Let $\mathcal{H}$ be a connection on $M_P/{\T^2}$ such that the $\SU(2)$-invariant $U_1\times U_2$ is horizontal at each point. A connection satisfying this property always exists: indeed, we showed in \cref{prop: T^2 invariance of mu and nu} that the one induced by the $\G2$-metric satisfies:
\[
g(U_1\times U_2,V_j)=\vphi(U_1,U_2,V_j)=0.
\]
\begin{remark}\label{rmk: metric induces connection}
	Note that an invariant metric on a principal bundle naturally induces an (Ehresmann) connection. Indeed, the horizontal distribution defined by $\mathcal{H}_p:=\mathcal{V}_p^\perp$ is clearly horizontal and equivariant.
\end{remark}
Using such a connection, integral curves of $U_1 \times U_2$ are horizontal lifts over such curves in $B$.
\begin{theorem}
\label{thm: associatives as level sets in B}
     Let $\mathcal{H}$ be a connection on the principal $G^{\SU(2)}$-bundle $M_P/ {\T^2} \to B$ such that $U_1\times U_2\in\mathcal{H}$. Let $\gamma$ be a curve in $M_P/{\T^2}$. The following are equivalent:
    \begin{enumerate}
        \item \label{item: T2 invariant assoc}
        The pre-image $\pi_{\T^2}^{-1}(\mathrm{im}\gamma)$ is a $\T^2$-invariant associative in $M_P$,
        \item \label{item: integral curve associative} $\gamma$ is an integral curve of $U_1 \times U_2$,
        \item \label{item: horizontal lift associative} $\gamma$ is the horizontal lift of a level set of $\av{\mu}$ on $B$. 
    \end{enumerate}
    Moreover, the correspondence between (\ref{item: T2 invariant assoc}) and (\ref{item: integral curve associative}) is 1-to-1, while for every integral curve of $U_1\times U_2$ in $B$ there is a $G^{\SU(2)}$-family of integral curves of $U_1\times U_2$ in $M_P/\T^2$.
\end{theorem}
\begin{proof}
The equivalence between (\ref{item: T2 invariant assoc}) and (\ref{item: integral curve associative}) has been established in \cref{prop: T^2-invariantAssociativesNoSU(2)}, while the equivalence between (\ref{item: integral curve associative}) and (\ref{item: horizontal lift associative}) can be deduced from the $G$-invariance of $U_1\times U_2$, the fact that it is assumed to be horizontal and \cref{prop: mu is a submersion}.
\end{proof}

	\subsection{Local description of associatives in the principal set} \label{sec: IC in a local trivialization}
We have seen that $M_P/ {\T^2}$ is a $G^{\SU(2)}$-principal bundle over the base $B$. 
In \cref{thm: associatives as level sets in B}, the integral curves of $U_1\times U_2$ in $M_P/ {\T^2}$ are described as horizontal lifts of curves in a surface. In the following, we will show how these horizontal lifts can be computed in a local trivialization of the principal bundle. 

\begin{lemma} \label{lemma: associative fibrations}
	Given $\mathcal{U}\subset B$ open, let $\mathcal{U}\times G^{\SU(2)} \to M_P/{\T^2}$ be a local trivialisation of the $G^{\SU(2)}$-bundle with $U_1\times U_2\in T\mathcal{U}\times\{0\}$. If $\bar{\mathcal{U}}\subset M_P$ and $p_{G^{\SU(2)}}\colon \bar{\mathcal{U}}\to G^{\SU(2)}$ are, repsectively, the induced local chart and the obvious projection coming from the trivialization, then the fibres of the submersion $(\av{\mu},p_{G^{\SU(2)}}): \bar{\mathcal{U}} \to \R^+\times G^{\SU(2)}$ are associative submanifolds.
\end{lemma}
\begin{proof}
	As $U_1\times U_2\in T\mathcal{U}\times\{0\}$, it follows that its integral curves will be constant on the $G^{\SU(2)}$ component of $\mathcal{U}\times G^{\SU(2)}$. Since $\av{\mu}$ is constant on the $G^{\SU(2)}$-component and since integral curves of $U_1\times U_2$ are contained in the level set of $\av{\mu}$ (\cref{thm: associatives as level sets in B}) we conclude.
\end{proof}

The aim is to find trivializations of $M_P/\T^2\to B$ where we can  apply \cref{lemma: associative fibrations}. Since $\mu$ is $G^{\SU(2)}$-equivariant, we can reduce the structure group of the $G^{\SU(2)}$-principal bundle. Indeed, given $v\in \R^3\setminus\{0\}$ and denoting by $\langle v \rangle$ the line spanned by $v$, then $Q_v:=\mu^{-1}(\langle v \rangle)$ is an $S^1$ reduction of the bundle $M_P/\T^2\to B$. 

\begin{proposition}\label{prop: existence flat connection on Qv}
Let $\mathcal{U}\subset B$ open. If $({\av{\mu}},\nu): \mathcal{U}\to\R^2$ is a diffeomorphism onto its image and the image is convex, then there exists a flat connection on $Q_v$ such that $U_1\times U_2$ is horizontal.
\end{proposition}
\begin{proof}
 Let $\theta\in \Omega^1(Q_v, \R)$ be any connection form on $Q_v$ for which $U_1\times U_2$ is horizontal. Then the curvature form $d \theta$ is a basic form, so there is a function $f\colon \mathcal{U} \to \R$ such that $d \theta=f d \nu \wedge d {\av{\mu}}$, where we are considering $({\av{\mu}},\nu)$ as coordinates on $\mathcal{U}\subset B$. The form $d {\av{\mu}}$ is basic and annihilates $U_1\times U_2$, hence, $\theta'=\theta+F d {\av{\mu}}$ is also a connection on $Q_v$ such that $U_1\times U_2$ is horizontal for every smooth function $F:\mathcal{U} \to \R$. The new connection $\theta'$ is flat if and only if 
 $ (\pd{\nu}{F} +f)d\nu \wedge d {\av{\mu}}=0$.
 Because the image is convex, $\pd{\nu}{F}=-f$ admits at least one solution, for instance, using the methods of characteristics. 
\end{proof}

\begin{theorem}\label{thm: construct trivialization}
 Let $\mathcal{U}\subset B$ open. If $({\av{\mu}},\nu): \mathcal{U}\to\R^2$ is a diffeomorphism onto its image and the image is convex, then there exists a trivialization $\mathcal{U}\times G^{\SU(2)} \to M_P/{\T^2}$ such that $U_1\times U_2\in T\mathcal{U}\times\{0\}$. As a consequence, the map $(\av{\mu},p_{G^{\SU(2)}})$ is a fibre bundle map whose fibres are associative submanifolds. Here, $p_{G^{\SU(2)}}$ is the projection to $G^{\SU(2)}$ coming from the trivialisation.
\end{theorem}
\begin{proof}
By \cref{prop: existence flat connection on Qv}, the bundle $Q_v$ admits a flat connection for which $U_1\times U_2$ is horizontal. Since $\mathcal{U}$ is diffeomorphic to a convex set (simply-connected), there is a trivialization $\mathcal{U}\times S^1 \to Q_v$ which induces this connection, i.e. the horizontal bundle is $T\mathcal{U}\times \{0\} \subset TQ_v$. Since $U_1\times U_2$ is horizontal the component in $S^1$ is constant along integral curves of $U_1\times U_2$. By equivariance, we get a trivialization $\mathcal{U}\times G^{\SU(2)}\to M_P/{\T^2}$ such that the component in $G^{\SU(2)}$ is constant along integral curves of $U_1\times U_2$. \end{proof}
Clearly, the condition on $(\av{\mu},\nu)$ in \cref{thm: construct trivialization} always holds locally.

\subsection{Associatives in the singular set}
In this subsection, we describe the $\T^2$-invariant associative submanifolds of $M$ that are contained in the singular set of the $\T^2\times\SU(2)$-action. In particular the following theorem holds.

\begin{theorem}[Associatives in the singular set]
\label{thm: associatives singular set} Let $\mathcal{S}_1, \mathcal{S}_2, \mathcal{S}_3$ and $\mathcal{S}_4$ be the strata as described in \cref{thm: class stab groups}. Then:
\begin{itemize}
    \item $\mathcal{S}_1$ admits an $\SU(2)$-equivariant submersion $F\colon \mathcal{S}_1\to S^2$ such that each (not necessarily connected) fibre is a $\T^2$-invariant totally geodesic associative. 
    \item every connected component of $\mathcal{S}_2$ is an associative $G$-orbit,
    \item  The set $\mathcal{S}_3\cup \mathcal{S}_4$ is {totally geodesic}, associative and the action of $G$ on $\mathcal{S}_3$ is of cohomogeneity one.
\end{itemize}
\end{theorem}
\begin{proof}
We first consider {$\mathcal{S}_1$}. For every $\underline{c}\in \R\times \R$ and $\underline{b} \in S^2$, consider the Killing vector field $W_{\underline{c},\underline{b}}:=c_1U_1+c_2U_2+b_1V_1+b_2V_2+b_3V_3$ and its zero set $L_{\underline{c},\underline{b}}\subset M\setminus \mathcal{S}$. Observe that every point of $\mathcal{S}_1$ lies in a unique $L_{\underline{c},\underline{b}}$, up to $L_{\underline{c},\underline{b}} = L_{-\underline{c},-\underline{b}}$. Indeed, $W_{\underline{c},\underline{b}}$ corresponds to the Lie algebra of $G_x\cong S^1$. Since $G_x$ is the quotient of a compact $1$-dimensional subgroup of $\T^2\times \SU(2)$, it follows that $\underline{c}\in \Q\times \Q$, (otherwise, $L_{\underline{c},\underline{b}}$ is empty). Let $H^+$ be a half plane in $\Q\times \Q$, determined by a line with irrational slope through the origin. This means that every element in $\Q\times \Q$ has a unique representative in $H^+$ under the action of $-1$. 
 In other words:
 \[\mathcal{S}_1=\bigcup_{(\underline{c},\underline{b}) \in H^+ \times S^2} L_{\underline{c},\underline{b}}\]
 and the union is disjoint. We define $F:\mathcal{S}_1\to S^2$ such that on each of $L_{\underline{c},\underline{b}}$ the value of $F$ is $\underline{b}$. To show that $F$ is equivariant, let $\xi_{c,b}$ be the Lie algebra element corresponding to the vector field $W_{c,b}$ and recall that
 \[L_{\underline{c},\underline{b}}=\{x \in M\mid \xi_{c,b} \in \mathfrak{g}_x\},\]
 where $\mathfrak{g}_x$ is the Lie algebra of $G_x$.
 The equivariance follows because, for every $g\in\SU(2)$ we have:
 \[\xi_{c,b} \in \mathfrak{g}_x \quad \Leftrightarrow \quad  \xi_{c,gb}= \mathrm{Ad}_g \xi_{c,b} \in  \mathrm{Ad}_g \mathfrak{g}_{x}=\mathfrak{g}_{gx}\]
 
 The space $L_{\underline{c},\underline{b}}$ is a totally geodesic submanifold since it is the zero set of a Killing vector field and, since the vector fields $U_1,U_2,U_1\times U_2$ commute with $W_{c,b}$, they are linearly independent and tangent to $L_{\underline{c},\underline{b}}$. 
 
 It remains to show that $F$ is a submersion. For a point $x\in \mathcal{S}_1$, a neighbourhood of the orbit $Gx$ in $\mathcal{S}_1$ is diffeomorphic to $\R\times G/{G_x}$. The vector field $U_1\times U_2$ is tangent to the $\R$ direction, so $F$ is invariant under the coordinate in $\R$ and descends to a $G$-equivariant map onto $G/{G_x}\cong S^2$, which is a $\T^2$-invariant submersion.

We now turn our attention to $\mathcal{S}_2$.
 By \cref{prop: dimensions S_i}, $\mathcal{S}_2$ is smooth, $3$-dimensional and, by \cref{thm: class stab groups}, associative. As it is $3$-dimensional, we deduce that every connected component is a $G$-orbit.
  
 Finally, we consider {$\mathcal{S}_3\cup \mathcal{S}_4$}. In \cref{prop: dimensions S_i},
 we have seen that $\mathcal{S}_3$ is smooth and $3$-dimensional and that $\mathcal{S}_4$ is smooth and $1$-dimensional. It follows from \cref{thm: class stab groups} (cfr. \cref{fig: repr how Sis intersect}) that $\mathcal{S}_3$ is dense in $\mathcal{S}_3\cup \mathcal{S}_4$ and it suffices to show that $\mathcal{S}_3\cup \mathcal{S}_4$ is smooth and that $\mathcal{S}_3$ is associative, totally geodesic and of cohomogeneity one. Clearly, $\mathcal{S}_3$ is open in $\mathcal{S}_3\cup \mathcal{S}_4$. Hence, it is enough to show smoothness at a point $x\in \mathcal{S}_4$. By \cref{thm: class stab groups}, the normal representation of $G_x$ on $\C^3$ splits into two invariant components $N=N_1\oplus N_2$ where $\dim_{\C}(N_1)=1,\dim_{\C}(N_2)=2$. The set of points with $3$-dimensional stabilizer is exactly $N_1$. So, by the slice theorem, there is a diffeomorphism of $G\times_{G_x} N$ to a neighbourhood $U\subset M$ of $Gx$ such that the subbundle $G\times_{G_x} N_1$ is mapped to $U \cap (\mathcal{S}_3\cup \mathcal{S}_4)$ and smoothness follows.  
 
 Being the vanishing locus of three Killing vector fields, $V_1,V_2,V_3$, it is clear that $\mathcal{S}_3$ is totally geodesic. Finally, it is associative because, at each point, the tangent space is the spanned by $U_1,U_2$ and $U_1\times U_2$.
  \end{proof}

  \begin{corollary}\label{cor: Associative fibrations corollary}
  If $(\av{\mu}, \nu) \colon B \to \R^2$ is a diffeomorphism onto its image and the image is convex, then $M$ admits a global $\T^2$-invariant associative fibration in the sense of \cref{def: calibrated fibration definition}. 
  \end{corollary}
 \begin{proof}
 Since $(\av{\mu}, \nu): B\to \R^2$ is a diffeomorphism onto its image and the image is convex, \cref{thm: construct trivialization} implies that there exists a smooth fibre bundle $\pi: M_P\to \R \times G^{\SU(2)}$ with $\T^2$-invariant associatives as fibres. Using \cref{thm: associatives singular set}, we conclude that the complement of $M_P$ is covered by possibly intersecting $\T^2$-invariant associatives.
 \end{proof}
 
 \subsection{Singularity analysis} In this last subsection, we show that every $\T^2$-invariant associative in a $\G2$ manifold with $\T^2\times\SU(2)$-symmetry needs to be smooth. 

\begin{theorem}\label{thm: regularity associatives}
	Every $\T^2$-invariant $\vphi$-calibrated integer rectifiable current in $M$ is a smooth submanifold. Moroever, if a $\T^2$-invariant $\vphi$-calibrated integer rectifiable current has support intersecting the singular set of the $\T^2\times\SU(2)$-action, then its support is contained in it. 
\end{theorem}
\begin{proof}
	As a first step, we observe that the local uniqueness and existence theorem (\cref{thm: local existence and uniqueness}) implies that $\T^2$-invariant $\vphi$-calibrated integer rectifiable currents are smooth away from $\mathcal{S}=\mathcal{S}_2\cup \mathcal{S}_4$. 
	
	Moreover, if $L$ is a $\T^2$-invariant $\vphi$-calibrated current with $\supp L\cap \mathcal{S}\neq\emptyset$, then its support is contained in the singular set of the $\T^2\times\SU(2)$-action. Indeed, if by contradiction $\supp L\cap M_P\neq\emptyset$, then $\restr{\mu}{\supp L}=c$ for some constant $c\neq0$, by \cref{cor: zero set of mu in barS}. However, once again by \cref{cor: zero set of mu in barS}, we have that $\restr{\mu}{\mathcal{S}}=0$ which is a contradiction as $\mu$ is constant on $L$. Hence, all $\T^2$-invariant currents with support in $M_P$ admit a local neighbourhood separated from the singular set of the $\T^2\times\SU(2)$-action and are smooth.
	
	We now consider $\T^2$-invariant associatives contained in $\mathcal{S}_1\cup \mathcal{S}_3\cup\mathcal{S}$. 	By \cref{thm: local existence and uniqueness}, we can distinguish two cases: $\supp{L}\subset \mathcal{S}_3\cup \mathcal{S}$ and $\supp{L}\subset \mathcal{S}_1\cup \mathcal{S}$. The smoothness of the second case was proven in \cref{thm: associatives singular set} so we restrict our attention to the first case. Given $x\in \mathcal{S}_1\cap \supp L\neq\emptyset$ we can associate a vector field $W_{\underline{c},\underline{b}}=c_1 U_1+c_2U_2+b_1V_1+b_2 V_2+b_3 V_3$ for $\underline c\in\R^2$ and $\underline b \in S^2$ on $M$, such that its zero set in $\mathcal{S}_1$ coincides with $\supp L\cap \mathcal{S}_1$ or one of its connected components (cfr. \cref{thm: associatives singular set}). We conclude that $\supp L$ is globally the zero set of a Killing vector field $W_{\underline{c},\underline{b}}$, which is a smooth totally geodesic submanifold. 
		\end{proof}
 
\begin{remark}
	The approach used to study the singularities in \cref{thm: regularity T3invariant} and \cref{thm: regularity SU(2)coassociatives} can be attempted for $\T^2$-invariant associatives as well. However, in this case, we could not rule out the existence of branched points.
\end{remark} 

\begin{remark}\label{rmk: extension associatives in manifolds with torsion}
	Note that, apart from \cref{sec: IC in a local trivialization} and \cref{cor: Associative fibrations corollary}, where we need $\nu$ to be defined, all the other results can be extended to manifolds with co-closed $\G2$-structures. 
\end{remark}

\section{\texorpdfstring{$\T^3$}{T3}--invariant and \texorpdfstring{$\SU(2)$}{SU2}--invariant coassociative submanifolds}\label{sec: T3-invariant/SU(2)-invariant coassociatives}
In this section, we study coassociative submanifolds of a $\G2$ manifold $(M,\vphi)$, endowed with a structure-preserving, cohomogeneity two action of $\T^2\times \SU(2)$. We use the same notation and conventions of \cref{sec: T^2xSU(2) symmetry subsection}. 

First, we consider coassociative submanifolds that are invariant under $\T^3\cong\T^2\times S^1< \T^2\times\SU(2)$, for some $S^1<\SU(2)$. Similarly to the $\T^2$-invariant case, we can characterize $\T^3$-invariant coassociatives in terms of integral curves of a vector field in the $\T^3$-quotient. Madsen and Swann \cite{MadsenSwann2018} found three first integrals of the ODE problem in the principal part of the $\T^3$-action, i.e., three constant quantities on every $\T^3$-invariant coassociative. Once again, these are components of the $\T^3$-multi-moment maps. For dimensional reasons, this means that $\T^3$-invariant coassociatives are the level sets of a function and from this we can prove that the same is true in $B=M_P/G$. Conversely, such level sets can be lifted to an $S^2$-family of $\T^3$-invariant coassociatives. Combining this result with the similar one for $\T^2$-invariant associatives, we deduce that there exists a parametrization of $B$ such that the coordinate lines correspond to $\T^2$-invariant associatives or $\T^3$-invariant coassociatives. Along the way, we show that $\T^3$-invariant coassociatives con only admit singularities modelled on the product of the Harvey--Lawson cone in $\C^3$ with a line. 

Afterwards, we consider $\SU(2)\cong\Id_{\T^2}\times\SU(2)$-invariant coassociatives. First of all, we need to assume that $\vphi$ vanishes when restricted to $\SU(2)$-orbits. Otherwise, it would be pointless discussing $\SU(2)$-invariant coassociatives (cfr. \cref{prop: characterization associative planes}). Most of the properties that were true for $\T^2$-invariant associatives remain true for $\SU(2)$-invariant coassociatives. The main difference is that $\SU(2)$-invariant coassociatives do not admit natural first integrals, but only $1$-forms on which $\SU(2)$-invariant coassociatives need to vanish.

In \cref{sec: section examples}, we will use the theory developed here to describe $\T^3$-invariant coassociatives and $\SU(2)$-invariant coassociatives in the FHN $\G2$ manifolds.

\subsection{\texorpdfstring{$\T^3$}{T3}-invariant coassociative submanifolds} Given any $S^1<\SU(2)$, we can consider a structure preserving $\T^3$-action on $M$ by $\T^2\times S^1<\T^2\times\SU(2)$. Moreover, up to passing to some quotient, we can assume that the action is effective.
We denote by $\overline{\mathcal{S}}$ the singular set of this action which satisfies: $\mathcal{S}_2\cup \mathcal{S}_4\subseteq\overline{\mathcal{S}}\subseteq \mathcal{S}_1\cup \mathcal{S}_2\cup \mathcal{S}_3\cup \mathcal{S}_4$. Madsen and Swann proved in \cite{MadsenSwann2018}*{Lemma 2.6} that the stabilizer of an effective $\T^3$-action on a $\G2$ manifold is either trivial, a circle or a two-torus.

In the notation of \cref{sec: subsection multi-moment maps}, we can assume that the generators of the $\T^3$ action are $U_1,U_2,V_1$ and, hence, the multi-moment maps associated to it are $\mu_1, \theta_1^1,\theta_1^2$ and $\nu$, which are maps in $C^{\infty}(M;\R)$ (as usual defined up to additive constants). Observe that \cref{eqn: U1xU2 in N} and \cref{thm: local existence and uniqueness} guarantee the local existence and uniqueness of $\T^3$-invariant associatives in $M\setminus \overline{\mathcal{S}}$.

 Similarly to the $\T^2$-invariant associative case, we can see $\T^3$-invariant coassociatives as integral curves of a vector field. 

\begin{proposition}\label{prop: T3invariant coassociatives as IC}
	Let $\Sigma_0$ be a $\T^3$-invariant coassociative submanifold of $M\setminus \overline{\mathcal{S}}$. Then $\Sigma_0/{\T^3}$ is an integral curve of the nowhere vanishing vector field $\nabla\mu_1$ in $(M\setminus \overline{\mathcal{S}})/{\T^3}$. Conversely, every integral curve of $\nabla\mu_1$ in $(M\setminus \overline{\mathcal{S}})/{\T^3}$ is the projection of a $\T^3$-invariant coassociative in $M\setminus \overline{\mathcal{S}}$.
\end{proposition}
\begin{proof}
The proof of this proposition is analogous to the one of \cref{prop: T^2-invariantAssociativesNoSU(2)}.
Observe that:
	\begin{align*}
		& \vphi(U_l,V_1,\nabla\mu_1)=g_\vphi (U_l\times V_1,\nabla\mu_1)=\ast\vphi(U_1,U_2,V_1,U_l\times V_1)=0, \hspace{20pt} l=1,2;\\
		& \vphi(U_1,U_2,\nabla\mu_1)=g_\vphi (U_1\times U_2,\nabla\mu_1)=\ast\vphi(U_1,U_2,V_1,U_1\times U_2)=0, 
	\end{align*}
	which ensure that $\{U_1,U_2,V_1,\nabla\mu_1\}$ is a coassociative subspace at each point of $M\setminus \overline{\mathcal{S}}$.
\end{proof}

In contrast to the associative case, $\nabla\mu_1$ does not commute with $\T^2\times\SU(2)$, hence, integral curves do not respect the stratification of \cref{sec: subsection stratification}. However, the following holds.
	
\begin{lemma}\label{lemma: mu_1 increasing along integral curve}
	Let $\gamma$ be an integral curve of $\nabla\mu_1$ in $M\setminus \overline{\mathcal{S}}$. Then the multi-moment map $\mu_1$ is strictly increasing along $\gamma$. 
\end{lemma}
\begin{proof}
	The lemma follows from the following standard computation:
	\[
	\frac{d}{dt} (\mu_1\circ\gamma)=d\mu_1(\dot\gamma)=g(\nabla\mu_1,\dot\gamma)=g(\nabla\mu_1,\nabla\mu_1)=\av{d\mu_1}^2>0.
	\]
	The strict inequality follows from \cref{prop: characterization associative planes} and \cref{eqn: U1xU2 in N}, which guarantees the existence of a vector $v$ such that $d\mu_1(v)>0$, i.e. the vector that together with the generators of the $\T^3$-action spans a coassociative plane. 
\end{proof}
We recall that $\T^3$-invariant coassociatives are the level sets of the following multi-moment maps.
\begin{proposition}[Madsen--Swann \cite{MadsenSwann2018}]\label{prop: T3-invariant coassociatives}
	The map $(\theta^1_1,\theta^2_1,\nu):M\setminus\overline{\mathcal{S}}\to\R^3$ is a submersion with fibres $\T^3$-invariant coassociative submanifolds. 
\end{proposition}

\begin{remark}
	In contrast to the $\T^2$-invariant associative case, where we showed that $M$ admits an associative fibration in the sense of \cref{def: calibrated fibration definition}, we can not argue in the same way in this case. Indeed, a priori we do not know if there exists a $\T^3$-invariant coassociative passing through each point of $\overline{\mathcal{S}}$.
\end{remark}

Using a completely different approach to the one employed in \cref{thm: regularity associatives}, we can study the singularities that a $\T^3$-invariant coassociative can admit. To this end, we need to describe the structure of the local model near the singular set $\overline{\mathcal{S}}$. 
This means that we only have to consider two cases, i.e., when the stabilizer is a circle or when it is a torus. We refer to these sets as $\overline{\mathcal{S}}_1$ and $\overline{\mathcal{S}}_2$, respectively. 

 \subsubsection{Blow-up analysis at $\overline{\mathcal{S}}_1$}\label{sec: Blow-up analysis at barS1} Let $p\in\overline{\mathcal{S}}_1$ and let $U_1\in\LieT^3$ be the generator of the $\T^3$-stabilizer at $p$. Let $U_2,U_3$ be a basis of the complement of $U_1$ in $\LieT^3$. We pick normal coordinates around $p$ using \cref{lemma: normal coordinates}. In these coordinates, under the blow-up procedure, the vector fields $U_1,U_2,U_3$, properly rescaled (cfr. \cref{lemma: def tildeX}), respectively converge to $\tilde{U}_1=U_1$ and $\tilde{U}_2=U_2(0),\tilde{U}_3=U_3(0)$ constant vector fields (cfr. \cref{lemma: normal coordinates}). If we write $\R^7$ as $\R^3\oplus \C^2$, where $\R^3$ is determined by $\tilde U_2, \tilde U_3,\tilde U_2\times_{\vphi_0} \tilde U_3,$ then $\tilde U_1$ generates a $\U(1)$-action on the $\C^2$-component preserving $\vphi_0$. Since this $\U(1)$ is a subgroup of $\G2$ and commutes with $\tilde{U}_2,\tilde U_3$ and $\tilde U_2\times_{\vphi_0} \tilde U_3$, it acts on $\C^2$ as a maximal torus of $\SU(2)$. We conclude that the integral curves of $\nabla^0\mu^0_1$ passing through $p$ generate, under the limit of the $\T^3$-action (cfr. \cref{rmk: Lie groups action limit in blow-up}), a multiplicity-1 plane. Here, $\nabla^0$ denotes the flat covariant derivative on $\R^7$ and $\mu_1^0$ is the multi-moment map defined by:
\[
d\mu_1^0=\ast\vphi_0(\tilde{U}_1,\tilde{U}_2,\tilde{U}_3,\cdot).
\]
 
\subsubsection{Blow-up analysis at $\overline{\mathcal{S}}_2$}\label{sec: Blow-up analysis at barS2} Given $p\in\overline{\mathcal{S}}_2$, we denote by $U_2,U_3$ the generators of the stabilizer of the $\T^3$-action at $p$ and by $U_1$ the generator of the complement in $\LieT^3$. Now, we pick normal coordinates at $p=0$, as above. In particular, we deduce from \cref{lemma: def tildeX} and \cref{lemma: normal coordinates} that, under blow-up, the properly rescaled vector fields $U_1,U_2,U_3$ converge to $\tilde{U}_1=U_1(0)$, constant vector field, and to $\tilde{U}_2=U_2$, $\tilde{U}_3=U_3$. We write $\R^7=\R\times\C^3$, where $\R$ is determined by the flow of $\tilde{U}_1$, and we observe that $\tilde{U}_2,\tilde{U}_3$ generate a $\T^2$, $\vphi_0$-preserving action that commutes with $\tilde{U}_1$. Hence, it acts only on the $\C^3$-component as a subgroup of $\SU(3)$. It is straightforward to see that integral curves of $\nabla^0\mu^0_1$ passing through $p$ generate, under the limit of the $\T^3$-action (cfr. \cref{rmk: Lie groups action limit in blow-up}), the multiplicity-1 cone:
$
\R\times N,
$
where $N$ is the Harvey--Lawson cone in $\C^3$.

\begin{theorem}\label{thm: regularity T3invariant}
	Let $\Sigma$ be a $\T^3$-invariant $\ast\vphi$-calibrated integer rectifiable current of $M$. Then $\Sigma$ is smooth at each point of $M$ where the stabilizer of the $\T^3$-action is $0$-dimensional or $1$-dimensional. Otherwise, the stabilizer is $2$-dimensional and $\Sigma$ has a tangent cone modelled on the product of the Harvey--Lawson cone in $\C^3$ with a line.
\end{theorem}
\begin{proof}
Let $\Sigma$ be a $\ast\vphi$-calibrated integer rectifiable current which is invariant under the $\T^3$-action. It is clear from the local existence and uniqueness theorem (\cref{thm: local existence and uniqueness}) that $\Sigma$ is smooth at each point where the stabilizer of the $\T^3$-action is $0$-dimensional. In particular, $\Sigma$ can exhibit singularities only at $\overline{\mathcal{S}}$.

	Note that $\Sigma$ can not be contained in $\overline{\mathcal{S}}$ and it corresponds to an integral curve $\gamma$ of $\nabla\mu_1$ in $M\setminus \overline{\mathcal{S}}$. Without loss of generality, we consider a connected component of $\Sigma$ in $M\setminus \overline{\mathcal{S}}$ so that $\gamma$ is connected.
	
	Let $p\in(\supp \Sigma)\cap \overline{\mathcal{S}}$ and let $B_2(0)$ be a neighbourhood of $p$, identified with $0$, as in \cref{lemma: normal coordinates}. Note that the restriction of $\Sigma$ to $B_2(0)\setminus \overline{\mathcal{S}}$ corresponds to a unique integral curve of $\nabla\mu_1$ up to picking $B_2(0)$ small enough. Otherwise, $\restr{\mu_1}{\supp \Sigma}$ would have an interior maximum or a minimum contradicting \cref{lemma: mu_1 increasing along integral curve}. In particular, the support of the integral curve can not be a loop passing through $p$. (This means that $\gamma_1$ as in \cref{fig:blow-up scheme} can not be an integral curve of $\nabla\mu_1$).
	
	We now want to show that, under a suitable blow-up, $\gamma$ converges to an integral curve of $\nabla^0\mu^0_1$ passing through zero. We can then conclude by the analysis of the local models (cfr. \cref{sec: Blow-up analysis at barS1}, \cref{sec: Blow-up analysis at barS2}) and by \cref{thm: regularity calibrated currents}.
	
	Since $0\in \overline{\Im\gamma}$, we can choose a sequence of points of $\Im\gamma$: $x_k\in C_k:=S_{1/k}(0)=\{x\in B_2(0):\av{x}_{\R^7}=\frac 1k\}$. In particular, $kx_k\in S_1(0)$ will converge, up to passing to a subsequence, to some $\overline{x}\in S_1(0)$. 
	We denote by $\gamma^x_t$ the integral curve of $\widetilde{(\nabla \mu_1)^t}$ with initial value $x$. Since for $k\to\infty$ we have that $kx_k\to \bar x$ and $\widetilde{(\nabla \mu_1)^t}\to\nabla^0 \mu^0_1$ because of \cref{lemma: rescaled multimoment}, it follows from the theory of ODEs that $\gamma^{kx_k}_{1/k}$ converges to $\gamma^{\overline{x}}_0$ integral curve of $\nabla^0 \mu^0_1$ of initial value $\overline{x}$. From the choice of $x_k$ and \cref{lemma: rescaled multimoment}, we deduce that $\{\gamma^{kx_k}_{1/k}\}_{k=1}^{\infty}$ is a blow-up of $\gamma$ and we can conclude. 
	\end{proof}

 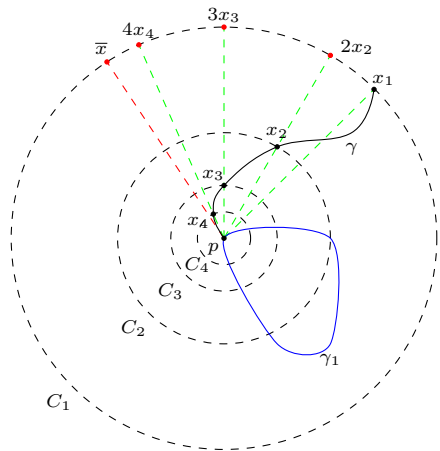
\begin{figure}
     \centering
    
 \begin{tikzpicture}[scale=1.4]
\draw[color=black, dashed] (0,0) circle [radius=2];
\draw[color=black, dashed] (0,0) circle [radius=1];
\draw[color=black, dashed] (0,0) circle [radius=1/2];
\draw[color=black, dashed] (0,0) circle [radius=1/4];

\draw[color=red, dashed] (0,0)-- (-1.1,1.670);
\draw[color=green, dashed] (0,0)-- (-0.8,1.832);
\draw[color=green, dashed] (0,0)-- (0,2);
\draw[color=green, dashed] (0,0)-- (1,1.732);
\draw[color=green, dashed] (0,0)-- (1.41,1.41);

\draw plot [smooth, tension=0.7] coordinates {(0,0) (-0.1,0.229) (0,0.5) (0.5,0.866) (1.2,1) (1.41,1.41)};
\draw [blue] plot [smooth cycle] coordinates {(0,0) (1,0) (1,-1) (0.5,-1)};

\filldraw[black] (0,0) circle [radius=0.02];
\filldraw[black] (-0.1,0.229) circle [radius=0.02];
\filldraw[black] (0,0.5) circle [radius=0.02];
\filldraw[black] (0.5,0.866) circle [radius=0.02];
\filldraw[black] (1.41,1.41) circle [radius=0.02];
\filldraw[red] (-1.1,1.670) circle [radius=0.02];
\filldraw[red] (-0.8,1.832) circle [radius=0.02];
\filldraw[red] (1,1.732) circle [radius=0.02];
\filldraw[red] (0,2) circle [radius=0.02];

\node at (-0.1,-0.1) [font={\tiny}]{$p$};
\node at (-0.24,0.15) [font={\tiny}]{$x_4$};
\node at (-0.1,0.6) [font={\tiny}]{$x_3$};
\node at (0.5,1) [font={\tiny}]{$x_2$};
\node at (1.5,1.5) [font={\tiny}]{$x_1$};
\node at (1.25,1.81) [font={\tiny}]{$2x_2$};
\node at (0,1.95) [font={\tiny}, above]{$3x_3$};
\node at (-0.8,1.8) [font={\tiny}, above]{$4x_4$};
\node at (-1,1.8) [font={\tiny},black,left]{$\overline{x}$};
\node at (1.2,1) [font={\tiny},black,below]{$\gamma$};
\node at (1,-1) [font={\tiny},black,below]{$\gamma_1$};
\node at (-0.25,-0.25) [font={\tiny},black]{$C_4$};
\node at (-0.5,-0.5) [font={\tiny},black]{$C_3$};
\node at (-0.85,-0.85) [font={\tiny},black]{$C_2$};
\node at (-1.55,-1.55) [font={\tiny},black]{$C_1$};
\end{tikzpicture}
       
     \caption{Blow-up procedure of \cref{thm: regularity T3invariant}}
     \label{fig:blow-up scheme}
 \end{figure}

\begin{remark}
 In \cref{sec: the Bryant--Salamon case}, we will see that there are examples of singular $\T^3$-invariant coassociatives. 
\end{remark}

\begin{remark}
	Observe that we have not used the fact that $\T^3$ is a subgroup of $\T^2\times\SU(2)$. In particular, \cref{thm: regularity T3invariant} holds in $\G2$-manifolds with a structure-preserving $\T^3$-action.
\end{remark}

On $B:= M_P/G$ the $\T^3$-invariant coassociatives correspond to the level sets of $\nu$.
\begin{theorem}\label{thm: T3-invariant coassociatives in the quotient}
	Let $\Sigma_0$ be a $\T^3$-invariant coassociative submanifold of $M_P$. Then the projection of $\Sigma_0$ to $B$ is contained in a level set of $\nu$. Conversely, every level set of $\nu$ on $B$ can be lifted to an $S^2$-family of  $\T^3$-invariant coassociatives.
\end{theorem}
\begin{proof}
	If we consider the projection of $\Sigma_0$ to $M_P/\T^2$, we obtain a surface $\Sigma_0/\T^2$ which is invariant under the action of an $S^1< G^{\SU(2)}$. So, projecting it to $B$ reduces the dimension to one and we obtain a curve in $B$. We conclude from \cref{prop: T3-invariant coassociatives} and dimensional reasons that $\Sigma_0$ is contained in a level set of $\nu$. 
	
	Conversely, given a level set of $\nu$ on $B$ and a point $p$ in it, we can construct, using \cref{prop: T3-invariant coassociatives}, a $\T^3$-invariant coassociative from every point of $M_P/\T^2$ in the fibre over $p$. Indeed, such a point determines a value of $(\theta^1_1,\theta^2_1,\nu)$. Since two points in the same $S^1$-orbit determine the same $\T^3$-invariant coassociative we conclude.
\end{proof}

As a consequence of this discussion we deduce that $B$ has a nice parametrization determined by associative and coassociative submanifolds, which are $\T^2$-invariant and $\T^3$-invariant respectively.
\begin{corollary}[Associative/coassociative parametrization of the quotient]\label{cor: Associative/coassociative parametrization} 
	Consider the local orthogonal parametrization of $B:=M_P/G$ given by $(\av{\mu},\nu)$. Then the coordinate lines correspond to $\T^2$-invariant associative submanifolds and $\T^3$-invariant coassociative submanifolds, respectively.
\end{corollary}
\begin{proof}
	The proof follows immediately from \cref{thm: associatives as level sets in B} and \cref{thm: T3-invariant coassociatives in the quotient}.
\end{proof}

\begin{remark}\label{rmk: T3-invariant coassociatives extension G2 manifolds with torsion}
	Note that, apart from \cref{lemma: mu_1 increasing along integral curve}, \cref{thm: regularity T3invariant} and \cref{cor: Associative/coassociative parametrization} where we need $\mu$ to be defined, all the other results of this section so far can be extended to manifolds with closed $\G2$-structures. Indeed, this can be done by reading $\ast\vphi(U_1,U_2,V_1,\cdot)^{\musSharp}$ instead of $\nabla\mu_1$. 
\end{remark}

\begin{figure}
\centering
\begin{tikzpicture}[scale=0.5]
\draw[color=blue, thick] (-3,-3)-- (3,-3);
\draw[color=black, <-] (0,-2.5)-- (0,-1);
\draw[color=black, ->] (3.7,1.67)-- (5.2,1.67);
\draw[color=orange, thick] (5.5,-1)-- (5.5,3.8);
\draw [thick, dashed] plot [smooth cycle, tension=0.7] coordinates { (-3,3) (-2,3.8) (-1,3) (0,3.5) (1,2.5) (2,3.4) (3,3) (3,0) (2,-1) (1,0) (0,-0.5) (-1,0.5) (-2,-0.5) (-3,0)};

\draw[blue,dashed] (-2,3.8)--(-2,-0.5);
\draw[blue,dashed] (-1,3)--(-1,0.5);
\draw[blue,dashed] (0,3.5)--(0,-0.5);
\draw[blue,dashed] (1,2.5)--(1,0);
\draw[blue,dashed] (2,3.4)--(2,-1);
\node at (4.3,2) {$\nu$};
\node at (0.2,-1.7) [right]{$\av{\mu}$};
\node at (-3,3.9) []{$B$};

\draw [dashed,orange] plot [smooth, tension=0.7] coordinates {(-3,2.8) (-2,2.4) (-1,2.5) (0,2.2) (1,1.9) (2,2.6) (3,2.7)};
\draw [dashed,orange] plot [smooth, tension=0.7] coordinates {(-3.15,1.670) (-2,1.5) (-1,1.8) (0,1.4) (1,1.3) (2,1.5) (3.15,1.670)};
\draw [dashed,orange] plot [smooth, tension=0.7] coordinates {(-3.08,0.4) (-2, 0.3) (-1,1.1) (0,0.3) (1,0.7) (2,0.5) (3.08,0.4)};
\end{tikzpicture}
\caption{Associative/coassociative parametrization of $B$}
\end{figure}
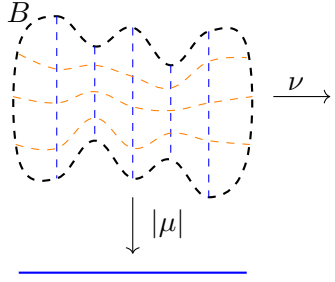
\subsection{\texorpdfstring{$\SU(2)$}{SU2}-invariant coassociative submanifolds} For the sake of brevity we omit the proofs, which are analogous to the other cases. In order to guarantee the existence of $\SU(2)$-invariant coassociatives, we need to assume that $\vphi(V_1,V_2,V_3)\equiv 0$ from now on. Actually, it is enough to have that it vanishes at a point. Indeed, Cartan's formula, together with $[\mL_X, i_Y]=i_{[X,Y]}$, implies that $\vphi(V_1,V_2,V_3)$ is a constant function. A sufficient condition, but not necessary as shown in \cref{sec: SU(2)coassociatives in BS}, is that the $\SU(2)/\Gamma_2$ action has a singular orbit. We denote the singular set of this action by $\tilde{\mathcal{S}}$.

\begin{proposition}\label{prop: SU(2)invariant coass as curves}
	Let $\Sigma_0$ be a $\SU(2)$-invariant coassociative submanifold of $M\setminus \tilde{\mathcal{S}}$. Then $\Sigma_0/{\SU(2)}$ is an integral curve of the nowhere vanishing vector field $\nabla\eta$ in $(M\setminus \tilde{\mathcal{S}})/{\SU(2)}$. Conversely, every integral curve of $\nabla\eta$ in $(M\setminus \tilde{\mathcal{S}})/{\SU(2)}$ is the projection of a $\SU(2)$-invariant coassociative in $M\setminus \tilde{\mathcal{S}}$.
\end{proposition}
\begin{lemma}\label{lemma: eta increasing along IC}
	Let $\gamma$ be an integral curve of $\nabla\eta$ in $M\setminus \tilde{\mathcal{S}}$. Then the multi-moment map $\eta$ is strictly increasing along $\gamma$ . 
\end{lemma}
\begin{proof}
	The proof is analogous to the one of \cref{lemma: mu_1 increasing along integral curve}. The existence of the vector $v$ such that $d\eta(v)>0$ is guaranteed once again by \cref{prop: characterization associative planes} and by the assumption: $\vphi(V_1,V_1,V_3)\equiv 0$.
\end{proof}

\begin{proposition}\label{prop: SU(2)-invariant coassociatives split with stratification}
	The flow of $\nabla\eta$ preserves the orbit type of $G$. Hence, the integral curves of $\nabla\eta$ stay in the same stratum of the stratification described in \cref{thm: class stab groups}.
	\end{proposition}
	
By \cref{lemma: GSU(2) action on M/T^2 and GT2 action}, the action of $\T^2$ on $M$ induces on the quotient $M_P/(\SU(2)/\Gamma_2)$ a $G^{\T^2}$ principal bundle structure with base space $B$. Let $\mathcal{H}$ be a connection on $M_P/(\SU(2)/\Gamma_2)$ such that the $\T^2$-invariant vector field $\nabla\eta$ is horizontal. For instance, the connection induced by the metric $g_\vphi$ satisfies this property:
$g(U_i,\nabla \eta)=\ast\vphi(U_i, V_1,V_2,V_3)=0$
for $i=1,2$ (cfr. \cref{rmk: metric induces connection}).
As in \cref{thm: associatives as level sets in B}, we deduce the following proposition. 
\begin{theorem}
\label{thm: co-associatives as level sets in B} 
     Let $\mathcal{H}$ be a connection on the principal $G^{\T^2}$-bundle $M_P/ {\SU(2)} \to B$ such that $\nabla \eta \in\mathcal{H}$. 
     Let $\gamma$ be a curve in $M_P/({\SU(2)/\Gamma_2})$. The following are equivalent:
    \begin{enumerate}
        \item \label{item: SU(2) invariant coassoc}
        The pre-image $\pi_{\SU(2)}^{-1}(\mathrm{im}\gamma)$ is a $\SU(2)$ invariant co-associative in $M_P$,
        \item \label{item: integral curve coassociative} $\gamma$ is an integral curve of $\nabla \eta$,
        \item \label{item: horizontal lift coassociative} $\gamma$ is the horizontal lift of an integral curve of $\nabla \eta$ in $B$.
    \end{enumerate}
    Moreover, the correspondence between (\ref{item: SU(2) invariant coassoc}) and (\ref{item: integral curve coassociative}) is 1-to-1, while for every integral curve of $\nabla\eta$ in $B$ there is a $\T^2$-family of integral curves of $\nabla\eta$ on $M_P/({\SU(2)/\Gamma_2})$.
\end{theorem}

\begin{remark}
	Note that, we can not conclude that we have an $\SU(2)$-invariant coassociative fibration in the sense of \cref{def: calibrated fibration definition}. Indeed, \cref{thm: co-associatives as level sets in B} only implies that $M_P$ admits a foliation of coassociative leaves.
\end{remark}

In contrast to the other cases, the obvious 1-forms that would give constant quantities on $\SU(2)$-invariant coassociatives are not closed. These are defined as:
\begin{align}\label{eqn: quasi multimoment maps}
\omega_1:=\vphi(V_2,V_3,\cdot),\hspace{10pt} \omega_2:=\vphi(V_3,V_1,\cdot), \hspace{10pt}\omega_3:=\vphi(V_1,V_2,\cdot).
\end{align}
\begin{remark}
	These 1-forms can be put in the context of weak homotopy moment-maps (see \cite{Herman2018} and references therein). Moreover, since $i_{U_l}\omega_i=-\theta^l_i$ the $\omega_i$s do not descend to the quotients: $M_P/(\SU(2)/\Gamma_2), M_P/\T^2$ and $B$.
\end{remark}

\begin{proposition}\label{prop: SU(2)coassociatives as level 1forms}
	A $4$-dimensional submanifold, $\Sigma_0$, is an $\SU(2)$-invariant coassociative submanifold of $M\setminus \tilde{\mathcal{S}}$ if and only if  $\restr{\omega^i}{\Sigma_0}= 0$ for all $i=1,2,3$.  
\end{proposition}

\begin{remark}
	The previous proposition does not use the additional $\T^2$-action. In particular, we re-obtain the characterizing ODEs for the $\SU(2)$-invariant coassociative submanifolds on the Bryant--Salamon manifold $\Lambda_-^2(S^4)$ and $\Lambda_-^2(\C\P^2)$ computed in \cite{KarigiannisLotay2021}.
\end{remark}

In a similar fashion to \cref{thm: regularity T3invariant}, one can obtain the following regularity result on $\SU(2)$-invariant coassociative submanifolds. 

\begin{theorem}\label{thm: regularity SU(2)coassociatives}
	Every $\SU(2)$-invariant $\ast\vphi$-calibrated integer rectifiable current in $M$ is a smooth submanifold.
\end{theorem}

\begin{remark}
The existence of the $\T^2$-action is crucial for \cref{thm: regularity SU(2)coassociatives}. Indeed, Karigiannis and Lotay constructed in \cite{KarigiannisLotay2021} examples of asymptotically singular $\SU(2)$-invariant coassociatives on $\Lambda^2_- (S^4)$ and on $\Lambda^2_- (\C\P^2)$.
\end{remark}

\begin{remark}\label{rmk: SU(2)-invariant coassociatives extension G2 manifolds with torsion}
	Note that, apart from \cref{lemma: eta increasing along IC} and \cref{thm: regularity SU(2)coassociatives} where we need $\eta$ to be defined, all the other results can be extended to manifolds with closed $\G2$-structures. Indeed, this can be done by reading $\ast\vphi(V_1,V_2,V_3,\cdot)^{\musSharp}$ in place of $\nabla\eta$.
\end{remark}

\section{Examples}\label{sec: section examples}
In this final section, we consider the $\G2$ manifolds constructed by Foscolo--Haskins--Nordstr\"om in \cite{FoscoloHaskinsNordstrom2021b} and the Bryant--Salamon $\G2$ manifolds of topology $S^3\times\R^4$. On these spaces we use the general theory developed in \cref{sec: section T^2-invariant associatives,sec: T3-invariant/SU(2)-invariant coassociatives} to study calibrated submanifolds in them.

In particular, fixed a $\T^2\times\SU(2)<\SU(2)\times\SU(2)\times\U(1)$, we compute in each FHN manifold the relative stratification and multi-moment maps. Then we explicitly construct the submersion $F:\mathcal{S}_1\to S^2$ given in \cref{thm: associatives singular set} and describe the quotient $M_P/G$, together with the relevant multi-moment maps. In this way, we have described all $\T^2$-invariant associatives and $\T^3$-invariant coassociatives in the FHN manifolds. By inspection, one can see that $\SU(2)$-invariant coassociative are trivial.

In reality, our discussion does not rely on the completeness of the FHN manifolds, and is carried out in the non-complete setting.

Afterwards we specialize our discussion to the Bryant--Salamon manifolds of topology $S^3\times\R^4$, which are explicit examples of FHN manifolds. Finally, we observe that certain possibly twisted vector subbundles of the trivial bundle $S^3\times\R^4\to S^3$ are associative submanifolds with respect to the Bryant--Salamon $\G2$-structure.

\subsection{The Foscolo--Haskins--Nordstr\"om manifolds}\label{sec: the FHN calibrated submanifolds} The FHN manifolds, described in \cref{sec: FHN manifolds}, admit the required $\T^2\times \SU(2)$-symmetry. Indeed, the action of $(\lambda_1,\lambda_2,\gamma)\in\U(1)\times\U(1)\times\SU(2)$ on $([p,q],t)\in(\SU(2)\times\SU(2))/K_0\times I$, given as follows:
\begin{align}\label{eqn: T2xSU(2)-action on FHN}
(\lambda_1,\lambda_2,\gamma)\cdot ([p,q],t)=([\lambda_1 p \overline{\lambda}_2,\gamma q \overline{\lambda}_2],t),
\end{align}
is structure preserving (cfr. \cref{eq: SU(2)^2xU(1) FHN action}), where the two $\U(1)$s are generated by quaternionic multiplication by $i$.

\begin{remark}\label{rmk: other T2SU(2) action on FHN}
    Obviously, there is another action of $(\lambda_1,\lambda_2,\gamma) \in \T^2\times\SU(2)$ on $([p,q],t)\in(\SU(2)\times\SU(2))/K_0\times I$:
    \[
    (\lambda_1,\lambda_2,\gamma)\cdot ([p,q],t)=([\gamma p \overline{\lambda}_2,\lambda_1 q \overline{\lambda}_2],t).
    \]
    The discussion is analogous to the one for \cref{eqn: T2xSU(2)-action on FHN} and we leave it to the reader.
\end{remark}
\subsubsection{The stratification}\label{sec: stratification FHN} We first deal with the set: $(\SU(2)\times\SU(2))/K_0\times \textup{Int}(I)$. If $K_0$ is trivial, it is straightforward to see that the principal stabilizer of the $\T^2\times\SU(2)$-action is generated by $(-1_{\T^2},-1_{\SU(2)})$. On the other hand, if $K_0=K_{m,n}\cap K_{2,-2}$ the principal stabilizer is a discrete subgroup of $\T^2\times\SU(2)$ with $\Gamma_1\neq0$. In both cases, $G^{\SU(2)}=\SO(3)$ and the singular set of the $\T^2\times\SU(2)$-action is given by:
\begin{align*}
  &\mathcal{S}_+= \left\{([p,q],t)\in (\SU(2)\times\SU(2))/K_0 \times \textup{Int}(I): p\in\C\times\{0\}\subset\H\right\},\\
  &\mathcal{S}_-= \left\{([p,q],t)\in (\SU(2)\times\SU(2))/K_0 \times \textup{Int}(I): p\in\{0\}\times\C\subset\H\right\},
\end{align*}
with $1$-dimensional stabilizer. If $K_0$ is trivial, the stabilizer at $([p,q],t)$ is either the circle $\{(\lambda,\lambda,q\lambda\overline{q})\}$ or $\{(\lambda,\overline{\lambda},q\overline{\lambda}\overline{q})\}$, depending on whether $([p,q],t)$ is in $\mathcal{S}_+$ or $\mathcal{S}_-$.

To understand the stratification on $(\SU(2)\times\SU(2))/K$ we need to distinguish three cases:

\textbf{Case 1 (\texorpdfstring{$K=\Delta\SU(2)$}{K-DSU2}).} If we identify $\SU(2)\times\SU(2)/\Delta\SU(2)$ with $S^3$ via $[(p,q)]\mapsto p\overline{q}$, then the action of $\T^2\times\SU(2)$ becomes, for every $p\in S^3\cong\Sp(1)$:
\[
(\lambda_1,\lambda_2,\gamma)\cdot p=\lambda_1 p \overline{\gamma}.
\]
We deduce that the stabilizer is always $2$-dimensional and it is the two torus: $\{(\lambda_1,\lambda_2, \overline{p}\lambda_1 p)\}$.

\textbf{Case 2 (\texorpdfstring{$K=\{1_{\SU(2)}\}\times\SU(2)$}{K=1SU2xSU2}).} 
Under the identification of $(\SU(2)\times\SU(2))/K$ with $S^3$ given by $[(p,q)]\mapsto p$, the $\T^2\times\SU(2)$ action becomes:
\[
(\lambda_1,\lambda_2,\gamma)\cdot p=\lambda_1 p \overline{\lambda_2},
\]
where $p\in S^3\cong\Sp(1)$. Hence, the stabilizer is the $\Z_2\times\SU(2)$ given by $\{\pm 1_{\T^2},\gamma\}$ if $p\notin (\C\times\{0\}\cup \{0\}\times\C)\subset\Sp(1)$, otherwise it is the $4$-dimensional $\SU(2)\times\U(1)$ given by $\{(\lambda,\overline{\lambda},\gamma)\}$ or $\{(\lambda,{\lambda},\gamma)\}$.

\textbf{Case 3 (\texorpdfstring{$K=K_{m,n}$}{K=Kmn}).} Using the isomorphism for $K_{m,n}\cong\U(1)$ of \cref{eqn: isom U(1) Kmn}, we have that two elements of $\SU(2)\times\SU(2)$ are in the same equivalence class if and only they they are equal up to right multiplication of $(e^{-in\theta},e^{im\theta})$ for some $\theta\in[0,2\pi)$. It is straightforward to verify that the stabilizer at $[(p,q)]$ is $1$-dimensional if $p\notin\C\times\{0\}\cup \{0\}\times\C\subset\Sp(1)$. Otherwise, it is $2$-dimensional.

\subsubsection{The multi-moment maps}\label{sec: multimoment maps FHN} In this subsection we compute the multi-moment maps on $(\SU(2)\times\SU(2))/K_0\times \textup{Int}(I)$ and hence, by continuity, on the whole space. In this subsection, $i,j,k$ will denote the standard basis of $\im\,\H$ such that $i\cdot j=k$.

Consider the Hopf fibration map $S^3\subset \H\to S^2\subset \im\,\H$ that maps $p\to \overline pip$. Taking two copies of the Hopf fibration, together with the identity on $\textup{Int}(I)$, yields the quotient map to the $\T^2$-quotient:
\begin{align*}
	\pi_{\T^2}: (\SU(2)\times \SU(2))/K_0\times \textup{Int}(I)&\to S^2\times S^2\times \textup{Int}(I)\\
	(p,q,t)&\mapsto (v,w,t),
\end{align*}
where $v= q\overline p ip\overline q =v_1i+v_2j+v_3k$ and $w= q i \overline q =w_1i+w_2j+w_3k.$

If $h:=\overline pip=h_1 i+h_2 j+h_3 k$, $g_1:=\overline q iq=g_{1,1} i+g_{1,2} j+g_{1,3} k$, $g_2:=\overline q jq=g_{2,1} i+g_{2,2} j+g_{2,3} k$ and $g_3:=\overline q kq=g_{3,1} i+g_{3,2} j+g_{3,3} k$, then the Killing vector fields of the $\T^2\times\SU(2)$-action satisfying \cref{eqn: bracket relation between generators action} are:
\begin{align*}
U_1(p,q,r)&=(ip,0,0)=(p\overline pip,0,0)=-\sum_{m=1}^3 h_m E_m(p,q,r),\\
U_2(p,q,r)&=(-pi,-qi,0)=E_1+F_1,
\end{align*}
\begin{align*}
V_1(p,q,r)&=-\frac{1}{2}(0,-iq,0)=-\frac12 (0,q\overline{q}iq,0)=\frac{1}{2}\sum_{m=1}^3 g_{1,m} F_m,\\
V_2(p,q,r)&=-\frac{1}{2}(0,-jq,0)=-\frac12 (0,q\overline{q}jq,0)=\frac{1}{2}\sum_{m=1}^3 g_{2,m} F_m,\\
V_3(p,q,r)&=-\frac{1}{2}(0,-kq,0)=-\frac12 (0,q\overline{q}kq,0)=\frac{1}{2}\sum_{m=1}^3 g_{3,m} F_m,
\end{align*}
where $E_m,F_m$ form the standard orthonormal left invariant frame of $\SU(2)\times\SU(2)$ as defined in \cref{sec: appendix G2 structure FHN}. 

A straightforward computation gives the multi-moment maps in the quotient:
\begin{align}
\label{eqn: multimoment FHN}
\begin{split}
    \nu&=-4(b-c_1)\langle v,w\rangle_{\R^3},\qquad \hspace{1.6 cm} \mu=-4\dot a \dot b v \times_{\R^3} w,\\
    \theta^1&=2av-2(a-b)\langle v,w\rangle_{\R^3} w,\qquad \hspace{0.5 cm}
\theta^2=-2(b+c_2) w, \\
    \eta&=\textup{Primitive of}\left(\frac{2ba^2+c_2(b^2+2a^2+c_1 c_2)}{\sqrt{-\Lambda}}\right),
    \end{split}
\end{align}
where $\Lambda$ is as defined in \cref{eqn: biglambda}. Note that we used the following identities:
\begin{align*}
    h_1=\langle v,w\rangle_{\R^3},\quad
    \langle h,g_m\rangle_{\R^3}=v_m,\quad
    g_{m,1}=w_m,\quad 
    (h\times g_m)_1=(v\times w)_m,
\end{align*}
for every $m=1,2,3$.
\subsubsection{Associatives in the singular set}
\label{subsec: assoc in the singular set} As a first step, we deal with $(\SU(2)\times\SU(2))/K_0\times \textup{Int}(I)$. Observe that the images of $\mathcal{S}_+$ and $\mathcal{S}_-$ under the $\T^2$-projection map $\pi_{\T^2}$ are:
\[\mathcal{O}_+=\{(v,v,t)\in S^2\times S^2\times\textup{Int}(I)\}, \quad \mathcal{O}_-=\{(v,-v,t)\in S^2\times S^2\times\textup{Int}(I)\}.\]

As argued in \cref{lemma: GSU(2) action on M/T^2 and GT2 action}, the action of $G^{\SU(2)}$ descends to $(M\setminus \mathcal{S})/\T^2$ and $G^{\SU(2)}=\SO(3)$ acts diagonally on $S^2\times S^2$. This $\SO(3)$-action is of cohomogeneity one and the singular orbits are $\mathcal{O}_+$ and $\mathcal{O}_-$ which have stabilizer diffeomorphic to $S^1$.

The proof of \cref{thm: associatives singular set} contains the construction of a fibration $\mathcal{S}_1\to S^2$ with associative fibres. These are zero sets of Killing vector fields. For $\mathcal{S}_+\cup \mathcal{S}_-$, the fibration can be described explicitly as follows.

Let $u \colon (\SU(2)\times\SU(2))/K_0\times \textup{Int}(I)\to S^2\times S^2$ be the composition of $\pi_{\T^2}$ with the projection $p:S^2\times S^2 \times \textup{Int}(I)\to S^2\times S^2$. Then $u$ maps $\mathcal{S}_+\cup \mathcal{S}_-$ to $p(\mathcal{O}_+)\cup p(\mathcal{O}_-)$ and the fibres are associative.
\begin{proposition}
\label{prop: fibration singular set FHN}
 The map $u\colon \mathcal{S}_+\cup \mathcal{S}_-\to p(\mathcal{O}_+) \cup p(\mathcal{O}_-)\cong S^2\cup S^2$ is a submersion with totally geodesic $\T^2$-invariant associative fibres of topology {$\T^2 \times \textup{Int}(I)$}. 
\end{proposition}
\begin{proof}
 By $\SU(2)$-equivariance, it suffices to show the statement for a single fibre in each of $\mathcal{O}_+$ and $\mathcal{O}_-$. We restrict ourselves to the fibre over the point $\{(i,i)\} \in \mathcal{O}_+\subset \Im\,\H\times \Im\,\H$, as the $\mathcal{O}_-$ case is analogous. 
 
Note that
 \[
 u^{-1}(\{(i,i)\})=\{([p,q],t):p,q\in(\C\times\{0\})\cap\Sp(1), t\in\textup{Int}(I)\},
 \]
 which is the fixed set of the involution $(i,i,i)\in U(1)\times U(1)\times \Sp(1)$ acting on $(\SU(2)\times\SU(2))/K_0\times \textup{Int}(I)$ as in \cref{eqn: T2xSU(2)-action on FHN}. So $ u^{-1}(\{(i,i)\})$ is a connected component of the fixed set of $(i,i,i)$, which is therefore totally geodesic and associative.
\end{proof} 

We now consider the singular orbit $\SU(2)\times\SU(2)/K$. If $K=\Delta\SU(2)$ or $K=\{1\}\times\SU(2)$, then $\SU(2)\times\SU(2)/K$ is an associative submanifold because it is either $\mathcal{S}_2$ or $\mathcal{S}_3\cup \mathcal{S}_4$. For $K=K_{m,n}$, the singular orbit, $\SU(2)\times\SU(2)/K_{m,n}$, is diffeomorphic to $S^3\times S^2$ and it admits a submersion onto $S^2$:
\begin{align*}
F:(\SU(2)\times\SU(2))/K_{m,n}\to S^2\quad [(p,q)]\mapsto qi\overline{q},
\end{align*}
with fibres that are $\T^2$-invariant associative submanifolds, of topology the lens space: $L(m;-n,n)$. 

In order to prove the previous claim, we observe that, by $\SU(2)$-equivariance, it is enough to show that $F^{-1}(\{i\})=\{[p,q]: q\in(\C\times\{0\})\cap \Sp(1)\}$ has the desired properties. By inspection, it is straightforward to deduce that it is $\T^2$-invariant and of the given topology. Associativity of $F^{-1}(\{i\})$ follows because it is a connected component of the set with $2$-dimensional stabilizer with respect to the action of \cref{rmk: other T2SU(2) action on FHN}. Moreover, there are two additional $\T^2$-invariant associative submanifolds in $\SU(2)\times\SU(2)/K_{m,n}$: the two components of $\mathcal{S}_2$ described in the stratification discussion of \cref{sec: stratification FHN}, which have topology $L(n;m,-m)$. 

Finally, note that for all possible $K$, the associative submanifolds of \cref{prop: fibration singular set FHN} extend smoothly to associatives of topology $S^1\times\R^2$ because of \cref{thm: regularity associatives}.

\subsubsection{Associatives in the principal set} On the principal set
\[
M_P=\left((\SU(2)\times\SU(2))\times \textup{Int}(I)\right) \setminus \left(\mathcal{S}_+\cup \mathcal{S}_- \right),
\]
we are able to give an an explicit parametrization of the $G^{\SU(2)}$-bundle described in \cref{sec: subsection Associative in the principal set}.

Consider the maps:
\[
\Psi \colon \SO(3)\times (0,\pi)\to S^2\times S^2,\quad (g,\theta)\mapsto (g_1, (g_1\cos \theta -g_2 \sin \theta))
\]
where $g_1, g_2$ and $g_3$ are the column vectors of $g$, and: 
\begin{align*} \label{eqn: map A}
 A\colon S^2\times S^2\setminus (p(\mathcal{O}_+ \cup \mathcal{O}_-))\to \SO(3),\quad (v,w)\mapsto \left(\left(v,\frac{1}{\sin \theta}(\cos\theta v-w),-\frac{1}{\sin (\theta)}v\times w\right)\right),
\end{align*}
where $\theta\in(0,\pi)$ is defined by $\langle v,w\rangle_{\R^3} =\cos\theta$. We recall that $p:S^2\times S^2\times \textup{Int}(I)\to S^2\times S^2$ is the obvious projection and that $p(\mathcal{O}_\pm)=\{(v,\pm v)\}$.

It is easy to see that the map $(A,\theta)$ is the inverse of $\Psi$, and $\Psi$ is a diffeomorphism that is equivariant with respect to the action of $\SO(3)$ on both spaces, where $\SO(3)$ acts on $\SO(3)\times (0,\pi)$ by left multiplication on the $\SO(3)$ factor.
The singular orbits $\mathcal{O}_+$ and $\mathcal{O}_-$ are the images of $\{0\}\times \SO(3)$ and $\{\pi\}\times \SO(3)$ if $\Psi$ is extended to $\SO(3)\times [0,\pi]$.

By taking the identity on the component $\textup{Int}(I)$ we get the equivariant diffeomorphism, which we also denote by $\Psi$:
\[
\Psi\colon \SO(3)\times (0,\pi)\times \textup{Int}(I)\to M_P/{\T^2}=(S^2\times S^2\times \textup{Int}(I))\setminus (\mathcal{O}_+\cup \mathcal{O}_-). 
\]
This means that the base space of the $G^{\SU(2)}$-bundle described in \cref{sec: subsection Associative in the principal set} is diffeomorphic to $B=(0,\pi)\times \textup{Int}(I)$ and $\Psi$ is a global trivialization of $M_P/{\T^2} \to B$.
 With respect to this trivialization, we have:
 \begin{align*}
 \av{\mu}=4\dot a \dot b \sin\theta,\quad
    \nu= -4(b-c_1)\cos\theta.
     \end{align*}
In order to apply the machinery of \cref{sec: IC in a local trivialization}, we need the following lemma. In our case, we will have $\alpha = (\av{\mu},\nu)$, $u=4\dot a \dot b$ and $v = \pm 4(b-c_1)$, depending on its sign.

\begin{lemma}
 Let $u,v$ be two functions from an interval, $\textup{Int}(I)$, to $\R^+$. If $\dot{u},\dot{v}$ are both positive or both negative everywhere, then $\alpha(\theta,t)=(u(t) \sin(\theta),v(t) \cos(\theta))$ defines a diffeomorphism from $(0,\pi)\times \textup{Int}(I)$ onto its image in $\R\times \R^+$. Moreover, let $v_-$ is the infimum of $v$ over $I$. Then $(u(t) \cos(\theta))^{-1}(c)$ is connected if $c>u_-$ and has two connected components otherwise. In particular, the map $\alpha$ is a diffeomorphism onto its image and the image is convex if and only if $u_- = 0$.
\end{lemma}
\begin{proof}
The determinant of the Jacobian vanishes if and only if $\dot{u} {v}\sin^2(\theta) + \cos^2( \theta){u}\dot{v}=0$, which never happens because $\dot{u}v$ and $\dot{v}u$ have the same sign. So, $\alpha$ is a local diffeomorphism and it remains to show that it is injective.
 For a fixed value, $t_0$, of $t$ the function $\alpha(\theta,t_0)$ traces out a half ellipse centred at the origin with semi-axes $u(t_0),v(t_0)$. If $t_1$ is another fixed value for $t$, then the ellipses $\alpha(\theta,t_0)$ and $\alpha(\theta,t_1)$ intersect if $u(t_0)-u(t_1)$ and $v(t_0)-v(t_1)$ have different signs. But this is impossible because $\dot{u}$ and $\dot{v}$ have the same sign.
 Denote by $u_\pm$ the supremum and the infimum of $u$, and by $v_{\pm}$ the supremum and infimum of $v$. The image of $\alpha$ is the half ellipse with semi-axes $(u_+,v_+)$ minus the smaller ellipse with semi-axes $(u_-,v_-)$ (see \cref{fig: image of alpha}), which implies the last statement.
\end{proof}
\begin{figure}
    \centering
    \begin{tikzpicture}[scale=0.8]
    \node at (-1.3,0.1) [above]{$(\av{\mu},\nu)$};
    \node at (-2.2,0) [left]{$(0,\pi)\times \textup{Int}(I)$};
    \draw[thick,black,->] (-2,0) -- (-0.5,0);
    \draw[thick, blue, dashed,rotate=-90] (2,0) arc (0:180: 2 and 2.6);
    \draw[thick, blue, dashed,rotate=-90] (1,0) arc (0:180: 1 and 0.8);
    \draw[thick, blue, dashed,rotate=-90] (1,0) -- (2,0);
    \draw[thick, blue, dashed,rotate=-90] (-1,0) -- (-2,0);
    \node at (2.6,0) [right]{$u_+$};
    \node at (0.8,0) [left]{$u_-$};
    \node at (0,2) [left]{$v_+$};
    \node at (0,1) [left]{$v_-$};
    \node at (0,-2) [left]{$-v_+$};
    \node at (0,-1) [left]{$-v_-$};
    \filldraw[blue] (2.6,0) circle [radius=0.03];
    \filldraw[blue] (0.8,0) circle [radius=0.03];
    \filldraw[blue] (0,2) circle [radius=0.03];
    \filldraw[blue] (0,1) circle [radius=0.03];
    \filldraw[blue] (0,-2) circle [radius=0.03];
    \filldraw[blue] (0,-1) circle [radius=0.03];
\end{tikzpicture}
    \caption{Image of $\alpha$}
    \label{fig: image of alpha}
\end{figure}
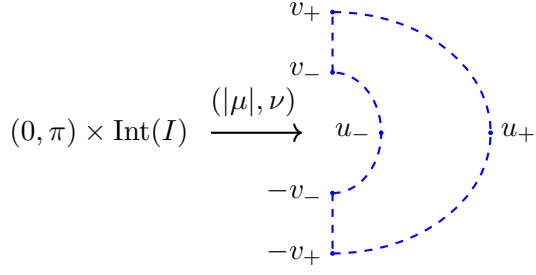

In particular, {if the infimum of $\dot{a}\dot{b}$ is zero}, we get a global fibration in the sense of \cref{def: calibrated fibration definition} by \cref{cor: Associative fibrations corollary}. Note that this is always the case, when the $\G2$-structure defined by Foscolo--Haskins--Nordstr\"om extends to the singular orbit $\SU(2)\times\SU(2)/K$ (cfr. \cref{sec: extension FHN to singular orbit}). 

On the other hand, if the infimum of $\dot{a}\dot{b}$ is not zero, we can still describe the $\T^2$-invariant associatives splitting $B\cong(0,\pi)\times \textup{Int}(I)$ into $(0,\pi/2)\times\textup{Int}(I)$ and $(\pi/2,\pi)\times\textup{Int}(I)$. 

We summarize everything in the following theorem. 
\begin{theorem}[$\T^2$-invariant associatives in FHN manifolds]\label{thm: associative manifolds FHN}
    Consider the stratification, as given in \cref{sec: T^2xSU(2) symmetry subsection}, of the FHN manifolds into $M_P\cup \mathcal{S}_1\cup \mathcal{S}_2\cup \mathcal{S}_3\cup \mathcal{S}_4$ with respect to the $\T^2\times\SU(2)$-action. 
    
    We first consider the subset $((\SU(2)\times\SU(2))/K_0)\times \textup{Int}(I)$, which does not intersect $\mathcal{S}_2,\mathcal{S}_3,\mathcal{S}_4$. Then each stratum decomposes into $\T^2$-invariant associatives in the following way:
\begin{itemize}
    \item 
$M_P$ is fibred by $\T^2$-invariant associatives which are horizontal lifts of level sets of $\av{\mu}=4\dot{a}\dot{b}\sin\theta$ in $B\cong(0,\pi)\times \textup{Int}(I)$, where $\theta$ is determined by $\cos\theta=\la v,w\ra$ and $v,w$ are images of the Hopf maps: $(v=q\bar p ip\bar q,\hspace{2pt} w=qi\bar q)\in S^2\times S^2$. The topology of these associatives is $\T^2\times \R$. If the $\G2$-structure extends smoothly to $(\SU(2)\times\SU(2))/K$, these associatives do not intersect $(\SU(2)\times\SU(2))/K$.

\item As in \cref{prop: fibration singular set FHN}, $\mathcal{S}_1$ admits a submersion over $S^2\cup S^2$ with totally geodesic $\T^2$-invariant associative fibres of topology {$\T^2\times \R$}. If the $\G2$-structure extends smoothly to $(\SU(2)\times\SU(2))/K$, these associatives extend smoothly to associatives of topology $S^1\times\R^2$ in $M$.
\end{itemize}

When the $\G2$-structure extends to $\SU(2)\times\SU(2)/K$, we distinguish two cases:
\begin{itemize}
    \item If $K=\Delta\SU(2)$ or $K=\Id_{\{\SU(2)\}}\times \SU(2)$, then $\SU(2)\times\SU(2)/K$ is a $\T^2$-invariant associative of topology $S^3$ as it is $\mathcal{S}_2$ or $\mathcal{S}_3\cup \mathcal{S}_4$.
    \item If $K=K_{m,n}$, the set consists of $\mathcal{S}_1$ and $\mathcal{S}_2$. There exists a submersion over $S^2$ with $\T^2$-invariant associative fibres of topology $L(n:m,-n)$. Moreover, there are two additional $\T^2$-invariant associatives corresponding to the two connected components of $\mathcal{S}_2$.
\end{itemize}

\end{theorem}

\subsubsection{\texorpdfstring{$\T^3$}{T3}-invariant coassociatives} Let $\T^3$ be the torus generated by $V_1,U_1,U_2$. It is straightforward to see that the singular set of this action, $\overline{\mathcal{S}}$, restricted to $((\SU(2)\times\SU(2))/K_0)\times \textup{Int}(I)$ is: 
\begin{align*}
\overline S_P=\left\{([p,q],t)\in (\SU(2)\times\SU(2)/K_0)\times \textup{Int}(I): p,q\in (\C\times\{0\}\cup \{0\}\times \C)\subset \Sp(1) \right\},
\end{align*} 
which is contained in $\subset \mathcal{S}_+\cup \mathcal{S}_-$.
On $\overline{\mathcal{S}}_{P}$ the stabilizer is $1$-dimensional and it is mapped, via $\pi_{\T^2}$ to $\{(\pm i,\pm i,t),(\pm i,\mp i,t)\}$. 

On $\SU(2)\times\SU(2)/K$, with $K=\Delta\SU(2)$ or $K=\{1\}\times\SU(2)$, the stabilizer is everywhere $1$-dimensional apart from the intersection of $\SU(2)\times\SU(2)/K$ with the closure of $\overline S_P$, where the stabilizer is $2$-dimensional. If $K=K_{m,n}$, the stabilizer at $[(p,q)]\in (\SU(2)\times\SU(2))/K_{m,n}$ is $2$-dimensional if $p$ and $q$ are in $\C\times\{0\}\cup\{0\}\times\C$, it is $1$-dimensional if $p$ or $q$ is in $\C\times\{0\}\cup\{0\}\times\C$ and it is $0$-dimensional otherwise.

By \cref{prop: T3-invariant coassociatives}, the $\T^3$-invariant coassociatives, in $M\setminus \overline{\mathcal{S}}$, are the level sets of the map $(\theta^1_1,\theta^2_1,\nu)$:
\[
([p,q],t)\mapsto \left(2av_1-2(a-b)\langle v,w \rangle_{\R^3}w_1, -2(b+c_2)w_1, -4(b-c_1)\langle v,w\rangle_{\R^3} \right),
\]
where $v,w$ are as above. 

We now characterize the $\T^3$-invariant coassociatives intersecting the $1$-dimensional and the $2$-dimensional stabilizer. 

Given $([p,q],t_0)\in \overline{\mathcal{S}}_{P}$, it is mapped via $(\theta^1_1,\theta^2_1,\nu)$ to 
$ (\epsilon_1 2  b(t_0),\epsilon_2 2(b(t_0)+c_2),\epsilon_3 4 (b(t_0)-c_1))$, where $\epsilon_i\in \{0,1\}$ take one of four possibilities for which $\epsilon_1 \epsilon_2 \epsilon_3=1$, depending whether $p$ and $q$ are in $\C\times\{0\}$ or $\{0\}\times \C$. We now turn our attention to $\SU(2)\times\SU(2)/K$. 

\textbf{Case 1 (\texorpdfstring{$K=\Delta\SU(2)$}{K=DSU2}).}
If $K=\Delta\SU(2)$, a $\T^3$-invariant coassociative intersects the set with $1$-dimensional stabilizer in $\SU(2)\times\SU(2)/K$, if and only if it is the preimage of 
$(x,0,0)$ for $x\in (-2c_1,2c_1)$. It intersects the set with $2$-dimensional stabilizer, and hence singular by \cref{thm: regularity T3invariant}, if and only if $x=\pm 2c_1$. 

\textbf{Case 2 (\texorpdfstring{$K=\{1_{\SU(2)}\}\times\SU(2)$}{K=1SU2xSU2}).} In this case, the $\T^3$-invariant coassociatives corresponding to the preimages of $(0,0,x)$, for $x\in [-4c_1,4c_1]$, are the ones intersecting $\SU(2)\times\SU(2)/K$. Among them, the one intersecting the set with $2$-dimensional stabilizer are the preimages of $(0,0,\pm 4c_1)$. 

\textbf{Case 3 (\texorpdfstring{$K=K_{m,n}$}{K=Kmn}).} When $K=K_{m,n}$, the coassociatives intersecting the set with $0$-dimensional stabilizer in $\SU(2)\times\SU(2)/K$ are the the level sets of points in: 
\[
\left\{(2mnr_0^3 xy,-2n(m+n)r_0^3y,-4m(m+n)r_0^3x): x,y\in (-1,1)\right\};
\]
they intersect the set with $1$-dimensional stabilizer they are the level set of points in: 
\[
\left\{(2mnr_0^3 xy,-2n(m+n)r_0^3y,-4m(m+n)r_0^3x): x=\pm 1,y\in (-1,1)\hspace{5pt} \text{or} \hspace{5pt} y=\pm 1,x\in (-1,1) \right\};
\]
and they are singular if they are the preimage of:
\begin{align*}
(\pm 2mn r_0^3,-2n(m+n)r_0^3,\mp 4{m}({m+n}){r_0}^3)\text{ or } (\pm 2mn r_0^3,+2n(m+n)r_0^3,\pm 4{m}({m+n}){r_0}^3).
\end{align*}

In particular, from this discussion one could characterize the $\T^3$-invariant coassociatives of different topology (see \cref{sec: T^3-invariant coassociatives BS} for an explicit example). Note that, the only topological possibilities are the $\T^3\times\R$, $\T^2\times\R^2$ and the singular ones $\T^2\times\R\times \R^+$.

\subsubsection{\texorpdfstring{$\SU(2)$}{SU2}-invariant coassociatives} Finally, we study $\SU(2)$-invariant coassociatives. Similarly to \cref{sec: multimoment maps FHN}, we can compute
$
\vphi(V_1,V_2,V_3)=c_2.
$
Hence, there are $\SU(2)$-invariant coassociatives if and only if $c_2=0$. If this is the case, the coassociative submanifolds are of the form:
\[
\{([p_0,q],t)\in ((\SU(2)\times\SU(2))/K_0)\times \textup{Int}(I): q\in \SU(2), t\in  \textup{Int}(I)\},
\]
for every fixed $p_0\in \SU(2)$. As we assumed $c_2=0$, the only possibility to extend the $\G2$-structure to $\SU(2)\times\SU(2)/K$ is for $K$ equal to $\{1\}\times\SU(2)$. In this situation, the resulting $\SU(2)$-invariant coassociatives extend to smooth $\R^4$s.

\subsection{The Bryant--Salamon manifold}\label{sec: the Bryant--Salamon case}
As an explicit special case of \cref{sec: the FHN calibrated submanifolds}, we consider the Bryant--Salamon manifolds of topology $S^3\times\R^4=\{(x,a)\in\H^2:\av{\av{x}}=1\}$. Up to an element of the automorphism group, we can restrict ourselves to the following actions of $\T^2\times\SU(2)$:
\begin{enumerate}
    \item \label{case 0} $(\lambda_1,\lambda_2,\gamma)(x,a)\mapsto (\lambda_1 x \bar{\gamma},\lambda_2 a \bar{\gamma})$,
    \item\label{case 1} $(\lambda_1,\lambda_2,\gamma)(x,a)\mapsto ({\lambda_1}x{\overline{\lambda_2}},\gamma a {\overline{\lambda_2}})$,
    \item \label{case 2} $(\lambda_1,\lambda_2,\gamma)(x,a)\mapsto (\gamma x{\overline{\lambda_2}},{\lambda_1}a{\overline{\lambda_2}})$,
\end{enumerate}
where  $(\lambda_1,\lambda_2,\gamma)\in \U(1)\times\U(1)\times\Sp(1)$ and the $\U(1)$s are generated by quaternionic multiplication by $i$. Note that Case (\ref{case 0}) can be reduced to the discussion in \cref{sec: the FHN calibrated submanifolds}, picking $K=\Delta\SU(2)$. The same holds for Case (\ref{case 1}) and Case (\ref{case 2}) picking $K=\{1\}\times\SU(2)$. However, to be more explicit, we fix the description of the Bryant--Salamon manifold as in \cref{eqn: Bryant--Salamon as FHN} and we adjust the arguments of \cref{sec: the FHN calibrated submanifolds} accordingly. 

\subsubsection{The stratification} We first notice that the principal stabilizer is generated by $(-1,-1)\in \T^2\times\SU(2)$ for all cases, hence $G^{\SU(2)}=\SO(3)$.

The stratification for Case (\ref{case 0}) is:
\begin{align*}
     M_P&=(S^3\times \H^*)\setminus \mathcal{S}_1, \quad
     \mathcal{S}_1=\{(x,a)\in S^3\times \H^*: \overline{x}a\in \C\times\{0\} \cup \{0\}\times\C\}, \\
     \mathcal{S}_2&=\{(x,0)\in \H^2\}, \quad
     \mathcal{S}_3=\emptyset, \quad
     \mathcal{S}_4=\emptyset,
\end{align*}
for Case (\ref{case 1}) it is:
\begin{align*}
     M_P&=(S^3\times\H^*)\setminus \mathcal{S}_1, \quad
     \mathcal{S}_1=\{(x,a)\in \H^2: x\in \U(1)\times\{0\} \cup \{0\}\times\U(1)\}, \\ 
     \mathcal{S}_2&=\emptyset, \quad
     \mathcal{S}_3=\{(x,0)\in \H^2\} \setminus \mathcal{S}_1, \quad
     \mathcal{S}_4=\{(x,0)\in \H^2\} \cap \mathcal{S}_1,
\end{align*}
finally, for Case (\ref{case 2}) it is:
\begin{align*}
    M_P&=(S^3\times\H^*\setminus \mathcal{S}_1), \quad
    \mathcal{S}_1=\{(x,a)\in \H^2: a\in \U(1)\times\{0\} \cup \{0\}\times\U(1)\}, \\
    \mathcal{S}_2&=\{(x,0)\in \H^2\} \quad 
    \mathcal{S}_3=\mathcal{S}_4=\emptyset.
\end{align*}
\subsubsection{The multi-moment maps} Before computing the multi-moment maps, we write the explicit form of the projection to the $\T^2$-quotient: $\pi_{\T^2}$. Identifying $\H^*$ with $S^3\times \R^+$ using the standard map: $a\mapsto (a/\av{a},\av{a})$, the projections take the following form in $S^3\times S^3\times\R^+$:
\begin{align*}
\pi_{\T^2}: S^3\times S^3\times \R^+\to S^2\times S^2\times \R^+ \quad
	(p,q,r)\mapsto (v,w,r),
\end{align*}
where, for Case (\ref{case 0}) $v=\overline{p}i{p},w=\overline{q} i{q}$, for Case (\ref{case 1}) $v=q\Bar{p}ip\Bar{q}, w=qi\Bar{q}$ and, for Case (\ref{case 2}), $v=pi\Bar{p},w=p\Bar{q}iq\Bar{p}$. The multi-moment maps, which pass to the $\T^2$-quotients, are:
\\
\begin{center}
\begin{tabular}{l| l l l}
   \hspace{0.7cm} & \textup{Case (\ref{case 0})} \hspace{3cm} & \textup{Case (\ref{case 1})} \hspace{3cm} & \textup{Case (\ref{case 2})} \hspace{3cm}\\
 \hline
   $\nu$ & $2\sqrt3 r^2\langle v,w\rangle_{\R^3}$ & $-\frac{\sqrt3}{2}(3c+4r^2)\langle v,w\rangle_{\R^3}$ & $-2\sqrt3 r^2\langle v,w\rangle_{\R^3}$\\
   $\theta^1$ &  $\frac{\sqrt3}{{4}}(3 c+ 4 r^2)v$ & $\sqrt3 r^2 v$&$\frac{\sqrt3}{4}(3c+4r^2)v$\\
   $\theta^2$ & $-\sqrt 3 r^2 w$ & $-\sqrt3 r^2w$ &$-\sqrt3 r^2w$\\
   $\theta^3$ & $-3r^2(c+r^2)^{1/3} v \times_{\R^3} w$ & $-3r^2(c+r^2)^{1/3} v\times_{\R^3} w$ & $3r^2(c+r^2)^{1/3} v\times_{\R^3} w$
\end{tabular}.
\end{center}

\subsubsection{\texorpdfstring{$\T^2$}{T2}-invariant associatives}
The description of the $\T^2$-invariant associatives follows exactly as in the FHN manifolds. For instance, we obtain the following result for Case (\ref{case 0}). 
\begin{theorem}[$\T^2$-invariant associatives in Bryant--Salamon manifolds]
\label{thm: Associatives in the Bryant--Salamon}
Consider the stratification, as given in \cref{sec: T^2xSU(2) symmetry subsection}, of the Bryant--Salamon space into $M_P\cup \mathcal{S}_1\cup \mathcal{S}_2\cup \mathcal{S}_3\cup \mathcal{S}_4$ with respect to the $\T^2\times\SU(2)$-action of Case (\ref{case 0}). Then each stratum decomposes into $\T^2$-invariant associatives in the following way:
\begin{itemize}
    \item 
$M_P$ is fibred by $\T^2$-invariant associatives which are horizontal lifts of level sets of $\av{\mu}=3r^2(c+r^2)^{1/3}\sin\theta$ in $B\cong(0,\pi)\times \R^+$, where $\theta$ is determined by $\cos\theta=\la v,w\ra$ and $v,w$ are images of the Hopf maps: $(v=pi\bar p,w=qi \bar q)\in S^2\times S^2$. The topology of these associatives is $\T^2\times \R$ and they do not intersect the zero section.

\item $\mathcal{S}_1$ admits a fibration over $S^2\cup S^2$ with totally geodesic $\T^2$-invariant associative fibres of topology {$\T^2\times \R$}. These associatives extend smoothly to associatives of topology $S^1\times\R^2$ in $M$.

\item $\mathcal{S}_2$ is the zero section, which is an associative totally geodesic group orbit of topology $S^3$.

\item $\mathcal{S}_3=\mathcal{S}_4=\emptyset$.

\end{itemize}

\end{theorem}

\subsubsection{\texorpdfstring{$\T^3$}{T3}-invariant coassociatives}\label{sec: T^3-invariant coassociatives BS}
Up to an element of the autormorphism group, we can choose, for all the three cases, the torus $\T^3$ acting on $(x,a)\in S^3\times\R^4$ as follows:
$$(\lambda_1,\lambda_2,\lambda_3)(x,a)\mapsto (\lambda_1 x \bar{\lambda}_3,\lambda_2 a \bar{\lambda}_3),$$
where all the $\lambda_i$s are generated by multiplication by $i$. 

It is straightforward to see that the singular set of this action, $\overline{\mathcal{S}}$, is given by the zero section and the following subset:
\begin{align*}
\overline S_P=\left\{(x,a)\in S^3\times \H: x,a\in (\C\times\{0\}\cup \{0\}\times \C)\subset \C\times\C  \right\},
\end{align*} 
In the singular set, the stabilizer is everywhere $1$-dimensional apart from the points in:
\begin{align*}
\left\{(x,0)\in S^3\times \H: x\in (\C\times\{0\}\cup \{0\}\times \C)\subset \C\times\C  \right\},
\end{align*} 
where the stabilizer is $2$-dimensional.

By \cref{prop: T3-invariant coassociatives}, the $\T^3$-invariant coassociatives are given by the level sets of the map $(\theta^1_1,\theta^2_1,\nu)$,
which is explicitly given by:
\[
(p,q,r)\mapsto \left(\frac{\sqrt 3}{4}(3c+4r^2) v_1,-\sqrt 3 r^2 w_1,2\sqrt3 r^2\langle v,w\rangle_{\R^3}\right),
\]
where $v,w\in S^2\subset\R^3$ are defined accordingly to (\ref{case 0}). By \cref{thm: regularity T3invariant}, the $\T^3$-invariant coassociatives are smooth topological $\T^3\times \R$, apart from the ones intersecting the points with one or $2$-dimensional stabilizer, which are smooth $\T^2\times\R^2$s and $\T^2\times\R\times\R^+$ cones, respectively. The intersection with the $2$-dimensional stabilizer occurs only to the preimages of $\{(\pm\frac{3\sqrt 3}{4}  c,0,0) \}$. The $\T^3$-invariant coassociatives intersecting the $1$-dimensional stabilizer are the ones corresponding to the fibres of the following set: $\{(x,0,0):x\in (-\frac{3\sqrt 3c}{4},\frac{3\sqrt 3c}{4})\}\cup A$, where $A$ is:
\begin{align*}
    \left\{\left(\pm \left(\frac{3\sqrt3 c}{4}+ a\right),-a,\pm 2a\right): a\in \R^+\right\}\cup\left\{\left(\pm \left(\frac{3\sqrt3 c}{4}+ a\right),+a,\mp 2a\right): a\in \R^+\right\}
\end{align*}

\begin{figure}
\centering
\includegraphics[width=0.5\linewidth]{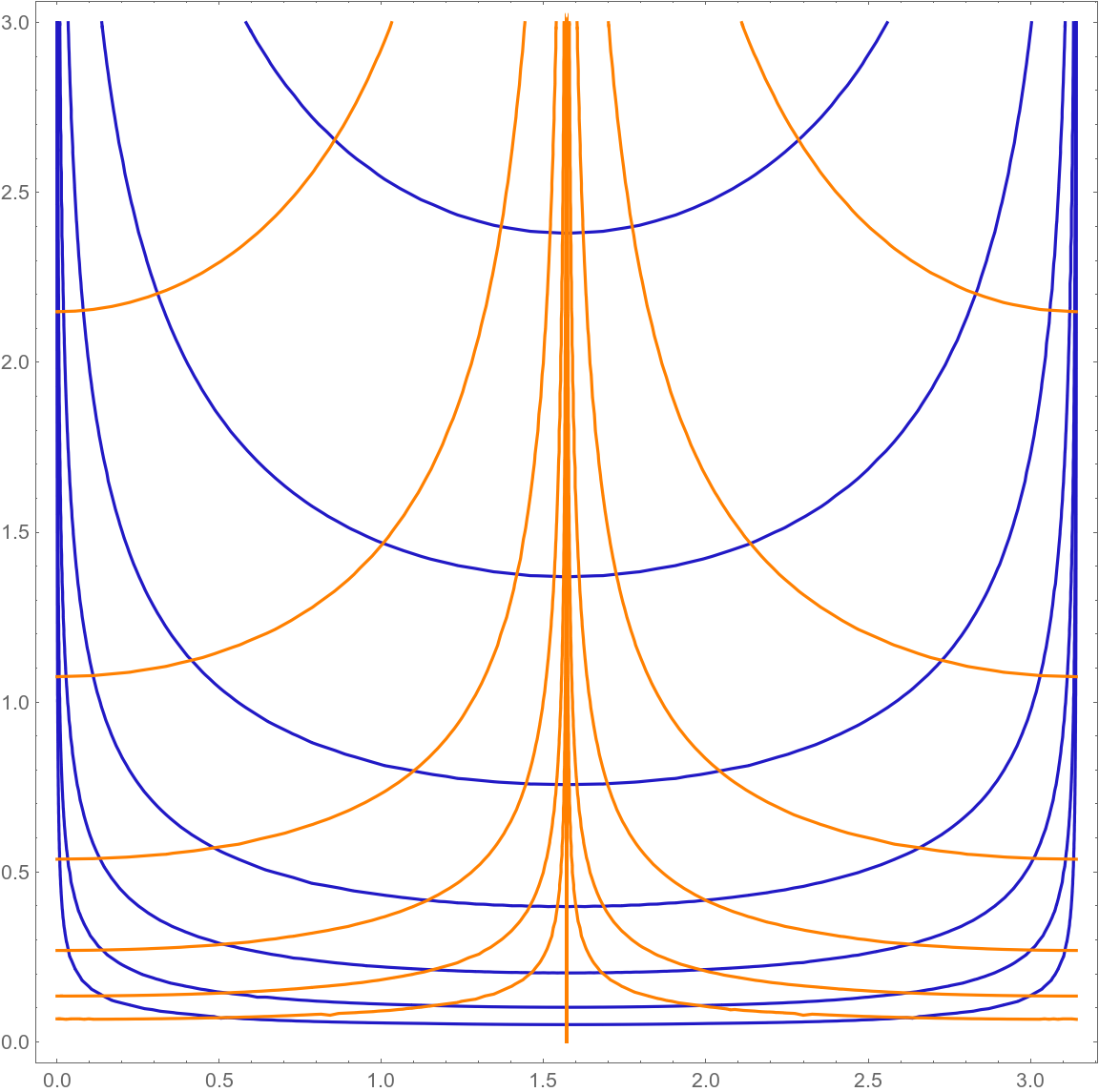}
\caption{Blue: The level sets of $\av{\mu}=3r^2(c+r^2)^{1/3}\sin \theta$ in $B=(0,\pi)\times \R^+$. Every level set represents an $\SU(2)$-family of $\T^2$-invariant associatives in $M_P$. Orange: The level sets of $\nu = 2 \sqrt{3} r^2 \cos\theta$. Every level set represents an $S^2$ family of $\T^3$-invariant coassociatives in $M_P$. The vertical line represents the ones intersecting the zero section, two of these $\T^3$-invariant coassociatives are singular.}
\label{figure levelsets BS}
\end{figure}

\subsubsection{\texorpdfstring{$\SU(2)$}{SU2}-invariant coassociatives}\label{sec: SU(2)coassociatives in BS} One can compute $\vphi_c(V_1,V_2,V_3)$ for Case (\ref{case 0}), Case (\ref{case 1}) and (\ref{case 2}). This vanishes only when $c=0$ in Case (\ref{case 0}) and Case (\ref{case 2}), while for Case (\ref{case 1}) it is always vanishing. We deduce that $\SU(2)$-invariant coassociatives are given by fibres of the standard projection to $S^3$ (cfr. \cite{KarigiannisLotay2021}*{Section 4}).

\subsubsection{Another family of associative submanifolds}
In this subsection, we consider the Bryant--Salamon manifold as described in \cite{KarigiannisLotay2021}*{Section 3}. The associatives fibres of $\mathcal{S}_1 \to S^2$ in \cref{thm: Associatives in the Bryant--Salamon} are products of a plane in $\R^4$ times a geodesic in $S^3$. In general, one can take any $2$-dimensional vector subspace $W \subset \R^4$, with an orthonormal basis $w_1,w_2$, and observe that $w_1 \times w_2$ is tangent to $S^3$. For every $p\in S^3$, we can consider $\gamma_{W,p}$ to be the unit length geodesic starting at $p$ with velocity $w_1\times w_2$, and observe that $\gamma_{W,p}\times W$ is an associative submanifold. These examples are not only part of the family of $\T^2$-invariant associative submanifolds, but also of the following family, where each associative contains an affine plane $\bar{W}:=W + x$ in $\R^4$. Here, $W$ is a $2$-dimensional vector subspace of $\R^4$ and $x$ is in the Euclidean perpendicular subspace $ W^\perp$. The orthogonal complement $W^\perp$ carries a unique positive complex structure, so we can define the curve contained in it:
\[\delta_{W,x}(t)=e^{-i\frac{t}{2}}x.\]
\begin{proposition}\label{prop: affine plane}
Let $p$ be a point in $S^3$, $\bar{W}=W + x$ be an affine plane with $x\in W^\perp$. The unique associative containing $\{p\}\times\bar{W}$ is 
\[N:=\{(\gamma_{W,p}(t),y+\delta_{W,x}(t))\in S^3\times \R^4 \mid y\in W, t \in \R\}.\]
\end{proposition}
\begin{proof}
As the uniqueness follows immediately from the local existence and uniqueness theorem, we only need to prove that $N$ is an associative submanifold.
We use the parametrisation of $S^3 \times \R^4$ as in \cite{KarigiannisLotay2021}*{Section 3}.
By applying elements of the automorphism group $\SU(2)^3$, we can assume without loss of generality that $W=\{a_2=a_3=0\}$. Moreover, we choose a left-invariant frame $\{E_1,E_2,E_3\}$ on $S^3$ such that the tangent space of $N$ is spanned by $\{\partial_{a_0},\partial_{a_1},e_1-(a_3\partial_{a_2}-a_2\partial_{a_3})/2\}$ at any point of $N$. We conclude as $\ast\varphi(e_1-(a_3\partial_{a_2}-a_2\partial_{a_3})/2,\partial_{a_0},\partial_{a_1},\cdot)=0$ at any point of $N$.
\end{proof}
In particular, \cref{prop: affine plane} extends the description of possibly twisted calibrated subbundles in manifolds of exceptional holonomy which was started by Karigiannis, Leung and Min-Oo in \cites{KarigiannisLeung2012,KarigiannisMinoo2005}.

\appendix

\section{Differentiable transformation groups}\label{app: differentiable transformation groups}
In this appendix, we provide a short introduction to the theory of differentiable transformation groups, i.e. the theory of Lie groups smoothly acting on smooth manifolds. In particular, we will fix the notation and state (without proof) three fundamental results: the slice theorem, the orbit type stratification theorem and the principal orbit type theorem.

Let $G$ be a compact connected Lie group of Lie algebra $\mathfrak{g}$ and let $M$ be a manifold.

\begin{definition}
	A Lie group action of $G$ on $M$ is a Lie group homomorphism:
	\begin{align*}
		G &\to \Diff(M)\\
		g&\mapsto f_g
	\end{align*}
	This homomorphism induces the smooth action map:
	\begin{align*}
		G\times M &\to M\\
		(g,m)& \mapsto f_g(m).
	\end{align*}
	It is costumary to write $g\cdot m$ (or $gm$) instead of $f_g(m)$.
\end{definition}

\begin{definition}
	An action of a Lie algebra $\mathfrak{g}$ on $M$ is a Lie algebra homomorphism:
	\begin{align*}
		\mathfrak{g} &\to \Gamma(TM) \\
			\xi&\mapsto \xi_M
	\end{align*}
	where the space of vector fields is endowed with the usual Lie-bracket structure.
\end{definition}

A Lie group action of $G$ on $M$ induces a Lie algbera action of $\mathfrak g$ on $M$, by mapping $\xi\in \mathfrak{g}$ to the the vector field:
\[
\xi_M(m):=\frac{d}{dt}\bigg|_0 \exp(-t\xi)\cdot m,
\]
where $\exp:\mathfrak{g}\to G$ is the usual exponential map for Lie groups. We will often identify $\xi\in \mathfrak{g}$ with the corresponding vector field (and similarly we will think of $\mathfrak{g}\subset \Gamma(TM))$. All such vector fields are called generating vector fields. Conversely, every Lie algebra action induces a (local) Lie group action.

For any $m\in M$, we can construct an (embedded, closed) submanifold of $M$,  called the orbit of $m$, which is defined by:
\[
Gm:=\{g\cdot m\in M: g\in G\}.
\]
We can also construct a (compact) Lie subgroup of $G$, called the stabilizer of $m$, which is defined by:
\[
G_m:=\{g\in G :g\cdot m=m\}.
\]

We denote the orbit space of the action by $M/G:=\{Gm: m\in M\}$.

\begin{remark}
	As for standard group actions, a Lie group action is free if all stabilizer subgroups are trivial. It is effective if the Lie group homomorphism $G\to \Diff(M)$ is injective. Finally, it it transitive if $Gm=M$ for some $m\in M$.
	\end{remark}

We can now state the slice theorem, which locally describe the geometry of $M$ near a fixed orbit.
\begin{theorem}[Slice theorem \cite{MontgomeryYang1957}]
Fix $m\in M$ and let $N$ be the normal space to the orbit $Gm$ at $m$. Then the associated bundle $G\times_{G_m} N$ is $G$-equivariantly diffeomorphic to the normal bundle of $Gm$ taking $[\Id_G,0]$ to $m$. The action of $G_m$ on $N$ is the natural one induced by $G$ and is called the slice representation. Moreover, $G$ acts on $G\times_{G_m} N$ on the first factor by left multiplication.
\end{theorem}

The stabilizers in different points of an orbit are related by the following adjoint formula:
\[
G_{gm}=Ad_g (G_m),
\]
where $g\in G$ and $m\in M$. It follows that to each orbit there exists a conjugacy class of subgroups of $G$. Given a subgroup $H$ of $G$, we denote by $(H)$ the conjugacy class of $H$ and we define:
\[
M_{(H)}:=\{m\in M: (G_m)=(H)\}.
\]

\begin{definition}
A stratification of a topological space $M$ is a decomposition into smooth submanifolds (called strata): $M=\cup_i M_i$, such that:
\begin{enumerate}
	\item each compact set of $M$ intersects finitely many $M_i$,
	\item if $M_i\cap \overline{M}_j\neq \emptyset$, then $M_i\subset\overline{M}_j$. 
\end{enumerate} 
\end{definition}

\begin{theorem}[Orbit type stratification]
	The decompositions:
	 \[
	 M=\bigcup_{(H)} M_{(H)}, \qquad M/G=\bigcup_{(H)} M_{(H)}/G
	 \]
	 are stratifications of $M$ and of $M/G$, respectively. Indeed, each $M_{(H)}\subset M$ is a smooth embedded submanifold which induces a smooth structure on $M_{(H)}/G$ via the quotient map. With respect to these smooth structures, the quotient map $p_{(H)}:M_{(H)}\to M_{(H)}/G$ is a fibre bundle of fibre $G/H$.
\end{theorem}

\begin{theorem}[principal orbit type \cite{MontgomerySamelsonYang1956}]
	If $M$ is connected, then there exists a unique conjugacy class $(H_P)$ such that $H_P\leq G_m$ for every $m\in M$, up to conjugation. The corresponding strata $M_P:=M_{(H_P)}\subset M$ and $M_P/G\subset M_P$ are open, dense and conneted. 
\end{theorem}
Let $m\in M\setminus M_P$. If $\dim(G_m)=\dim(H_P)$, then $(G_m)$ is called an exceptional orbit type for the action. Otherwise, it is a called a singular.

\section{Blow-up and regularity of calibrated submanifolds}\label{sec: blow-up and regularity}  In this appendix, we recall some basic preliminary results that we will use to study the singularities of associative and coassociative submanifolds. 

The first result, due to Madsen and Swann, claims that the blow-up of any torsion-free $\G2$-structure converges to the standard local model. 

\begin{theorem}[Madsen--Swann \cite{MadsenSwann2018}]\label{thm: MS rescalings}
Let $\vphi_0$ be the standard $\G2$-structure of $\R^7$ and let $\vphi$ be a torsion-free $\G2$-structure on $B_2 (0)\subset \R^7$ such that $\vphi(0)=\vphi_0(0)$. Then for $t>0$, the rescaled $\G2$-structure $\vphi_t:=t^{-3}\lambda_t^* \vphi$ is such that $\vphi_1=\vphi$ and we have that $\vphi_t\to\vphi_0$ as $t\to0$ on $B_1 (0)$ in the $C^k$-norm for every $k\geq0$, where $\lambda_t(x):=tx$ for every $x\in\R^7$. Moreover, the same holds for the $\vphi_t$-induced Riemannian metric $g_t=t^{-2} \lambda_t^* g$ and dual form $(\ast\vphi)_t=t^{-4}\lambda^\ast_t(\ast\vphi)$, where $g$ is the Riemannian metric induced by $\vphi$ and $\ast$ is the related Hodge dual.  
\end{theorem}

Moreover, Harvey and Lawson showed that under the blow-up procedure calibrated integer rectifiable currents remain calibrated, and converge to a calibrated tangent cone.

 A result due to Simon \cite{Simon1983a}*{Corollary p. 564}, together with Allard's regularity theorem (see \cite{Simon1983b}*{Chapter 5}), allows us to study the geometry of calibrated currents with mild singularities.

\begin{theorem}\label{thm: regularity calibrated currents}
	If $L$ is a $\vphi$-calibrated integer rectifiable current in $(B_2(0),\vphi)$ of density 1 away from 0 and has a tangent cone $C$ at $0$ that is non-singular (i.e. $C\setminus\{0\}$ is smooth), then $C$ is the unique tangent cone and, in a smaller neighborhood of $0$, $L$ is smooth everywhere apart from $0$, where the singularity is modeled on $C$. Moreover, if $C$ is also flat, then $L$ is smooth at $0$. The same result holds for $\ast\vphi$-calibrated integer rectifiable currents. 
\end{theorem}

Since we are interested in $G$-invariant submanifolds, for some compact Lie group $G$ acting effectively on $M$, we study how vector fields behave under blow-up. These vector fields will be chosen to be the generators of the action.

\begin{proposition}\label{prop: Killing vector fields are still killing under blowup}
Let $X$ be a vector field on $(B_2(0),\vphi)$ such that $\mathcal{L}_X \vphi=0$. Then the rescaled vector field $X^t:=\lambda_t^* X=t^{-1} (X\circ \lambda_t)$ is such that $\mathcal{L}_{X^t} \vphi_t=0$. Moreover, the same holds for $f(t)X^t$, where $f\in C^{\infty} (\R^+;\R)$.
\end{proposition}
\begin{proof}
 It follows from a straightforward application of Cartan's formula and $\lambda_t^* (i_X \vphi)=i_{\lambda_t^* X} \lambda_t^* \vphi$.
\end{proof}
Since $[X^t, Y^t]=\lambda_t^*[X,Y]$ for every $X,Y$ vector fields, the generators of a $G$-action defined for $t=1$ will give vector fields satisfying the same equations for every $t>0$. However, if we let $t$ go to $0$, $X^t$ does not necessarily converge. Indeed, if we write
\[
X(x)=\sum_{i=1}^7 a_i (x)\partial_i,
\]
for some functions $a_i$ on $B_2(0)$, then
\[
X^t(x)=t^{-1}\sum_{i=1}^7 a_i (tx)\partial_i, 
\]
which does not converge if some $a_i(0)\neq0$.

\begin{lemma}\label{lemma: def tildeX}
	 If $X$ is a real-analytic vector field on $(B_2(0),\vphi)$, we can always find a minimal integer $\alpha\leq 1$ such that $\tilde{X}^t:=t^\alpha X^t$ converges smoothly to some non-zero vector field $\tilde{X}$ as $t\to 0$.
\end{lemma}
Clearly, $\alpha=1$ if and only if $X(0)\neq 0$. Moreover, if $\mathcal{L}_{X^t}\vphi_t=0$, then \cref{prop: Killing vector fields are still killing under blowup} implies $0=\mathcal{L}_{\tilde{X}^t}\vphi_t\to \mathcal{L}_{\tilde{X}}\vphi_0$.

In a similar fashion, given a $1$-form $\omega$ one can define $\omega_t,\tilde\omega_t$ and $\tilde\omega$.
\begin{lemma}\label{lemma: rescaled multimoment}
    Given three vector fields $X,Y,Z$ on $(B_2(0),\vphi)$ as in \cref{thm: MS rescalings}, then for $t\to 0$ the following equations hold:
    \begin{enumerate}
        \item $\widetilde{({X}\lrcorner {Y}\lrcorner\vphi)}_t={\tilde X^t}\lrcorner {\tilde Y^t}\lrcorner\vphi_t\to {\tilde X}\lrcorner {\tilde Y}\lrcorner\vphi_0,$
        \item $\widetilde{({ X}\lrcorner {Y}\lrcorner {Z}\lrcorner\ast\vphi)}_t={\tilde X^t}\lrcorner {\tilde Y^t}\lrcorner {\tilde Z^t}\lrcorner\ast\vphi_t\to {\tilde X}\lrcorner {\tilde Y}\lrcorner {\tilde Z}\lrcorner\ast\vphi_0.$
            \end{enumerate}
\end{lemma}

The following lemma shows that if $X$ is a Killing vector field one can choose coordinates in which $\alpha$ is either $0$ or $1$.
\begin{lemma}
\label{lemma: normal coordinates}
Let $X_1,\dots X_k$ be Killing vector fields on $(M,\vphi)$ generated by an automorphic group action $G$, such that $X_1,\dots, X_l$ vanish at $p$ and $X_{l+1},\dots, X_k$ do not vanish at $p$. Then we can choose normal coordinates around $p$ such that:
\begin{align*}
    \tilde{X}_i &= \tilde{X}^t_i =  X_i \text{ if } i\leq l, \\
    \tilde{X}_i &= X_i(0)\neq 0 \text{ if } i \geq l+1
\end{align*}
and $\vphi(0)=\vphi_0$, where the $\tilde{X}_i$ are as defined in \cref{lemma: def tildeX}. In particular, this means that the $\alpha_i$ relative to $\tilde{X}^t$ is zero in the first case and one in the second.
\end{lemma}
\begin{proof}
When $i \geq l+1$, the statement holds in any coordinates and is a direct consequence of $X_i$ being continuous. 

 Normal coordinates are defined via the exponential map $\exp_p:B_\epsilon (0)\subset T_p M \to \mathcal{U}\subset M$. Because of the slice theorem, this map is $G$-equivariant and the stabilizer group $G_p$, has Lie algebra which is generated by $X_1,\dots, X_l$. So, in normal coordinates, the vector fields $X_1,\dots, X_l$ generate a linear action on $T_pM$. This means they agree with their first order approximation and the statement follows. We can use the freedom to choose a basis of $T_pM$ such that $\vphi(0)=\vphi_0$ since $\GL(7,\R)$ acts transitively on positive 3-forms on $\R^7$.
\end{proof}

\begin{remark}\label{rmk: Lie groups action limit in blow-up}
Observe that it makes sense to talk about the blow-up limit of a $G$-action in this setup. Indeed, given a Lie group action $G$ on $M$, this induces a Lie algebra action of $\mathfrak{g}$ on $M$. Now, \cref{lemma: normal coordinates} describes the blow-up limit of the $\mathfrak{g}$ action, and from this we can reconstruct a (local) group action.
\end{remark}

We restrict our attention to the case where the group $G$ is $\T^2\times\SU(2)$, or some discrete quotient of it. If $U_1,U_2$ are the generators of the $\T^2$-component and $V_1,V_2,V_3$ are generators of the $\SU(2)$-component, then for every $l,m=1,2$ and all $(i,j,k)$ cyclic permutation of $(1,2,3)$, they satisfy:
$$
[U_1,U_2]=0,\hspace{10pt} [{U}_l,{V}_m]=0, \hspace{10pt} [V_i,V_j]=\epsilon_{ijk}V_k.
$$
It follows that the vector fields $\tilde{U}^t_1,\tilde{U}^t_2,\tilde{V}^t_1,\tilde{V}^t_2,\tilde{V}^t_3$ are such that: 
\begin{align}
    [\tilde{U}^t_1,\tilde{U}^t_2]=0, \hspace{10pt} [\tilde{U}^t_l,\tilde{V}^t_m]=0,
    \end{align}
    \begin{align}\label{eqn: T2SU(2) equation rescaling}
        [\tilde{V}^t_i,\tilde{V}^t_j]=t^{\alpha_i+\alpha_j-\alpha_k}\tilde{V}^t_k,
    \end{align}
where $\alpha_i$ is the $\alpha$ defining $\tilde{V}^t_i$.

\bibliography{refs}
\bibliographystyle{plain}
\printaddress
\end{document}